\documentclass[11pt]{article}

\usepackage{amsmath,amssymb,amsthm,amscd,amsfonts}
\usepackage{epsfig,pinlabel,paralist}
\usepackage{mathtools}
\usepackage[vmargin=1in, hmargin=1.25in]{geometry}
\usepackage[font=small,format=plain,labelfont=bf,up,textfont=it,up]{caption}
\usepackage{calc}
\usepackage{etoolbox}
\usepackage{microtype}
\usepackage[all,cmtip]{xy}

\usepackage[bookmarks, bookmarksdepth=2, colorlinks=true, linkcolor=blue, citecolor=blue, urlcolor=blue]{hyperref}

\apptocmd{\thebibliography}{\raggedright}{}{}

\DeclareFontFamily{U}{matha}{}
\DeclareFontShape{U}{matha}{m}{n}{
  <-5.5>    matha5
  <5.5-6.5> matha6 
  <6.5-7.5> matha7
  <7.5-8.5> matha8
  <8.5-9.5> matha9
  <9.5-11>  matha10
  <11->     matha12
}{}
\DeclareSymbolFont{matha}{U}{matha}{m}{n}
\DeclareFontSubstitution{U}{matha}{m}{n}
\DeclareFontFamily{U}{mathx}{\hyphenchar\font45}
\DeclareFontShape{U}{mathx}{m}{n}{<-> mathx10}{}
\DeclareSymbolFont{mathx}{U}{mathx}{m}{n}
\DeclareFontSubstitution{U}{mathx}{m}{n}

\DeclareMathDelimiter{\ldbrack}{4}{matha}{"76}{mathx}{"30}
\DeclareMathDelimiter{\rdbrack}{5}{matha}{"77}{mathx}{"38}

\numberwithin{equation}{section}

\theoremstyle{plain}
\newtheorem{theorem}{Theorem}[section]
\newtheorem{maintheorem}{Theorem}

\newtheorem{proposition}[theorem]{Proposition}
\newtheorem{lemma}[theorem]{Lemma}
\newtheorem{principle}[theorem]{Principle}

\newtheorem{conjecture}[theorem]{Conjecture}

\newtheorem*{claim}{Claim}

\newtheorem{casea}{Case}

\newtheorem{steps}{Step}

\newcommand{\bX}{\mathbb{X}}
\newcommand{\bY}{\mathbb{Y}}
\DeclareMathOperator{\Ab}{Ab}
\newcommand{\m}{\to}
\providecommand{\Byk}{Bykovski\u{\i}}
\newcommand{\F}{\mathbb{F}}
\providecommand{\bT}{\mathbb T}

\DeclareMathOperator{\coker}{coker}

\providecommand{\TD}{\cT^{\pm}}
\providecommand{\bTD}{\bT^{\pm}}

\theoremstyle{definition}

\newtheorem{definition}[theorem]{Definition}
\newtheorem{notation}[theorem]{Notation}

\theoremstyle{remark}
\newtheorem{remark}[theorem]{Remark}

\newtheorem{example}[theorem]{Example}

\DeclareMathOperator{\GL}{GL}
\DeclareMathOperator{\SL}{SL}

\newcommand\C{\ensuremath{\mathbb{C}}}
\newcommand\Z{\ensuremath{\mathbb{Z}}}
\newcommand\Q{\ensuremath{\mathbb{Q}}}

\newcommand\Field{\ensuremath{\mathbb{F}}}

\DeclareMathOperator{\HH}{H}
\DeclareMathOperator{\CC}{C}
\DeclareMathOperator{\EE}{E}
\newcommand\RC{\ensuremath{\widetilde{\CC}}}

\newcommand\RH{\ensuremath{\widetilde{\HH}}}

\DeclareMathOperator{\link}{Link}
\newcommand\hlink{\ensuremath{\widehat{\link}}}
\DeclareMathOperator{\rk}{rk}

\newcommand\Span[1]{\ensuremath{\langle #1 \rangle}}

\newcommand\Set[2]{\ensuremath{\left\{\text{#1 $|$ #2}\right\}}}

\newcommand\bbH{\ensuremath{\mathbb{H}}}

\newcommand\cT{\ensuremath{\mathcal{T}}}

\newcommand\obbH{\ensuremath{\overline{\bbH}}}
\newcommand\oV{\ensuremath{\overline{V}}}
\newcommand\oW{\ensuremath{\overline{W}}}
\newcommand\osigma{\ensuremath{\overline{\sigma}}}

\newcommand\uR{\ensuremath{\underline{R}}}

\newcommand\fD{\ensuremath{\mathfrak{D}}}

\newcommand\Figure[4]{
\begin{figure}[t]
\centering
\centerline{\psfig{file=#2,scale=#4}}
\caption{#3}
\label{#1}
\end{figure}}

\DeclareMathOperator{\Gr}{Gr}
\DeclareMathOperator{\Or}{Or}
\newcommand{\Poset}{\ensuremath{\mathcal{P}}}

\DeclareMathOperator{\St}{St}
\DeclareMathOperator{\LetterB}{B}
\DeclareMathOperator{\LetterBA}{BA}
\DeclareMathOperator{\LetterBD}{BD}
\DeclareMathOperator{\LetterBDA}{BDA}
\DeclareMathOperator{\LetterBAO}{BAO}
\DeclareMathOperator{\height}{ht}

\newcommand{\BT}{\ensuremath{\LetterB^{\times}}}
\newcommand{\B}{\ensuremath{\LetterB^{\pm}}}

\newcommand{\BA}{\ensuremath{\LetterBA^{\pm}}}
\newcommand{\BTA}{\ensuremath{\LetterBA^{\times}}}

\newcommand{\BD}{\ensuremath{\LetterBD^{\pm}}}
\newcommand{\BDA}{\ensuremath{\LetterBDA^{\pm}}}

\newcommand{\BAO}{\ensuremath{\LetterBAO^{\pm}}}

\newcommand\Disc[1]{\ensuremath{\ldbrack #1 \rdbrack}}
\newcommand\Sphere[1]{\ensuremath{\overline{\ldbrack #1 \rdbrack}}}

\newcommand{\p}[1]{{\bf #1.}}

\title{\vspace{-40pt}On the top dimensional cohomology groups of congruence subgroups of $\SL_n(\Z)$\vspace{-15pt}}
\author{Jeremy Miller\thanks{Supported in part by NSF grant DMS-1709726} \and Peter Patzt \and Andrew Putman\thanks{Supported in part by NSF grant DMS-1811322}}
\date{}

\begin{document}

\vspace{-10pt}
\maketitle

\vspace{-18pt}
\begin{abstract}
\noindent
Let $\Gamma_n(p)$ be the level-$p$ principal congruence subgroup of $\SL_n(\Z)$.  Borel--Serre proved that
the cohomology of $\Gamma_n(p)$ vanishes above degree $\binom{n}{2}$.  We study the cohomology
in this top degree $\binom{n}{2}$.  Let $\cT_n(\Q)$ denote the Tits building of $\SL_n(\Q)$.
Lee--Szczarba conjectured that $\HH^{\binom{n}{2}}(\Gamma_n(p))$ is isomorphic to
$\RH_{n-2}(\cT_n(\Q)/\Gamma_n(p))$ and proved that this holds for $p=3$.  We partially prove
and partially disprove this conjecture by showing that a natural map 
$\HH^{\binom{n}{2}}(\Gamma_n(p)) \rightarrow \RH_{n-2}(\cT_n(\Q)/\Gamma_n(p))$
is always surjective, but is only injective for $p \leq 5$.  In particular, we completely
calculate $\HH^{\binom{n}{2}}(\Gamma_n(5))$ and improve known lower bounds for
the ranks of $\HH^{\binom{n}{2}}(\Gamma_n(p))$ for $p \geq 5$.
\end{abstract}

\setlength{\parskip}{0pt}
\tableofcontents
\setlength{\parskip}{\baselineskip}

\section{Introduction}
\label{section:introduction}

The cohomology of arithmetic groups plays a fundamental role in algebraic $K$-theory and number theory.
The most basic examples of arithmetic groups are $\SL_n(\Z)$ and its finite-index subgroups.  For
$n \geq 3$, the congruence subgroup property \cite{BassLazardSerre, Mennicke} says that every finite-index subgroup of $\SL_n(\Z)$
contains a {\em principal congruence subgroup}, i.e.\ the kernel $\Gamma_n(\ell)$ of the map
$\SL_n(\Z) \rightarrow \SL_n(\Z/\ell)$ that reduces coefficients modulo $\ell$.  In this paper,
we study the high-dimensional cohomology of $\Gamma_n(p)$ for a prime $p$.

\p{Stable and unstable cohomology}
Borel \cite{BorelStability} calculated $\HH^i(\Gamma_n(p);\Q)$ when $n \gg i$; the resulting cohomology
groups are known as the {\em stable cohomology}.  Borel--Serre \cite{BorelSerreCorners} later showed
that the rational cohomological dimension of $\Gamma_n(p)$ is $\binom{n}{2}$, so
$\HH^i(\Gamma_n(p);\Q) = 0$ for $i > \binom{n}{2}$.  This even holds integrally if $\Gamma_n(p)$ is torsion-free,
i.e.\ if $p \geq 3$.  The ``most unstable'' cohomology group of $\Gamma_n(p)$ is thus in degree
$\binom{n}{2}$.  Our main theorem calculates this when $p \leq 5$ and greatly strengthens the known
lower bounds on it when $p>5$, partially proving and partially disproving a conjecture of
Lee--Szczarba \cite{LeeSzczarba}.

\p{Duality}
Stating this conjecture requires some preliminaries.
Borel--Serre \cite{BorelSerreCorners} proved that $\Gamma_n(p)$ is a rational duality
group of dimension $\binom{n}{2}$, which implies that
\[\HH^{\binom{n}{2}-i}(\Gamma_n(p);\Q) \cong \HH_i(\Gamma_n(p);\fD \otimes \Q) \quad \text{for all $i$}\]
for a $\Gamma_n(p)$-module $\fD$ called the {\em dualizing module}.  This holds integrally
if $p \geq 3$.  In particular,
\[\HH^{\binom{n}{2}}(\Gamma_n(p);\Q) \cong \HH_0(\Gamma_n(p);\fD \otimes \Q) \cong (\fD \otimes \Q)_{\Gamma_n(p)},\]
where the subscript indicates that we are taking coinvariants.

\p{Steinberg modules}
The dualizing module $\fD$ has the following beautiful description.  For a field $\Field$, 
let $\cT_n(\Field)$ be the {\em Tits building} for $\SL_n(\Field)$, that is, the simplicial
complex whose $k$-simplices are flags
\[0 \subsetneq V_0 \subsetneq \cdots \subsetneq V_k \subsetneq \Field^n.\]
This is an $(n-2)$-dimensional simplicial complex, and the Solomon--Tits theorem \cite{Solomon, BrownBuildings} says
that $\cT_n(\Field)$ is homotopy equivalent to a wedge of spheres of dimension $(n-2)$.  The
{\em Steinberg module} for $\SL_n(\Field)$, denoted $\St_n(\Field)$, is $\RH_{n-2}(\cT_n(\Field))$.  
The action of $\SL_n(\Field)$ on $\cT_n(\Field)$ descends to an action on $\St_n(\Field)$.  
For $\Field = \Q$, the group $\Gamma_n(p)$ acts on $\St_n(\Q)$ via the inclusion
$\Gamma_n(p) \hookrightarrow \SL_n(\Z) \hookrightarrow \SL_n(\Q)$, and
Borel--Serre proved that the dualizing module $\fD$ for $\Gamma_n(p)$ is $\St_n(\Q)$.

\p{A first source of cohomology}
The cohomology groups we are interested in are thus isomorphic to $(\St_n(\Q) \otimes \Q)_{\Gamma_n(p)}$, with the $\otimes \Q$ unnecessary if $p \geq 3$.
One simple way to get elements of this is as follows.  There is a bijection between
subspaces of $\Q^n$ and direct summands of $\Z^n$ that takes $V \subset \Q^n$
to $V \cap \Z^n$.  The direct summand $V \cap \Z^n$ can be reduced modulo $p$, giving
a subspace of $\Field_p^n$.  This construction gives a map 
$\cT_n(\Q) \rightarrow \cT_n(\Field_p)$, and passing to homology yields a map
\[\St_n(\Q) = \RH_{n-2}(\cT_n(\Q)) \rightarrow \RH_{n-2}(\cT_n(\Field_p)) = \St_n(\Field_p).\]
It is not hard to see that this map is a surjection.  Since it is $\Gamma_n(p)$-invariant,
it factors through a surjection
\[\HH^{\binom{n}{2}}(\Gamma_n(p)) \cong (\St_n(\Q))_{\Gamma_n(p)} \twoheadrightarrow \St_n(\Field_p).\]
Lee--Szczarba \cite{LeeSzczarba} proved that this map is an isomorphism for $p=3$.  Using
their techniques, it is not hard to see that it is also an isomorphism for $p=2$ (after 
tensoring with $\Q$).

\p{Larger primes}
It is tempting to think that this might hold for all $p$, but unfortunately 
this is false.  Indeed, the Solomon--Tits theorem \cite{Solomon, BrownBuildings} also
says that $\St_n(\Field_p)$ is a free $\Z$-module of rank
\begin{equation}
\label{eqn:steinberg}
p^{\binom{n}{2}},
\end{equation}
but a theorem of Paraschivescu \cite{Para} says that the rank of $\HH^{\binom{n}{2}}(\Gamma_n(p))$ is at least
\begin{equation}
\label{eqn:para}
\left(\frac{p-1}{2}\right)^{n-1} p^{\binom{n}{2}}
\end{equation}
for primes $p \geq 3$.  The equation
\eqref{eqn:para} is greater than \eqref{eqn:steinberg} for primes $p \geq 5$.

\p{A source of additional cohomology}
The quotient map $\cT_n(\Q) \rightarrow \cT_n(\Q) / \Gamma_n(p)$
induces a $\Gamma_n(p)$-invariant map
\begin{equation}
\label{eqn:stnquotient}
\St_n(\Q) = \RH_{n-2}(\cT_n(\Q)) \longrightarrow \RH_{n-2}(\cT_n(\Q) / \Gamma_n(p)).
\end{equation}
It will follow from our results (see below) that this map is surjective, so the coinvariants
$(\St_n(\Q))_{\Gamma_n(p)}$ are at least as large as $\RH_{n-2}(\cT_n(\Q) / \Gamma_n(p))$.  
For $p \leq 3$, it turns out that $\cT_n(\Q) / \Gamma_n(p) \cong \cT_n(\Field_p)$, 
so the map \eqref{eqn:stnquotient} is really the map to $\St_n(\Field_p)$ we
discussed above.  However, for $p \geq 5$ the building
$\cT_n(\Field_p)$ is a proper quotient of $\cT_n(\Q) / \Gamma_n(p)$ and the map
\eqref{eqn:stnquotient} detects more of $\HH^{\binom{n}{2}}(\Gamma_n(p))$ than just an
$\St_n(\Field_p)$-factor.  See Proposition \ref{proposition:tquotient} and 
Remark \ref{remark:orientationirrelevant} for more details about all of this.

\begin{remark}
We will show that \eqref{eqn:stnquotient} also detects more cohomology than
Paraschivescu's bound \eqref{eqn:para}.
\end{remark} 

\p{Lee--Szczarba conjecture and our main theorem}
In \cite[remark on p.\ 28]{LeeSzczarba}, Lee--Szczarba conjectured that 
\eqref{eqn:stnquotient} detects all of $\HH^{\binom{n}{2}}(\Gamma_n(p))$:

\begin{conjecture}[Lee--Szczarba]
\label{conjecture:leeszczarba}
For a prime $p$ and $n \geq 2$, the map
\[(\St_n(\Q))_{\Gamma_n(p)} \longrightarrow \RH_{n-2}(\cT_n(\Q) / \Gamma_n(p))\]
induced by \eqref{eqn:stnquotient} is an isomorphism.  
\end{conjecture}

As we said above, they proved this for $p=3$, and it is not hard to use their techniques to also
prove it for $p=2$ (though they did not do this).  However, Ash \cite{AshAnnouncement} proved
that Conjecture \ref{conjecture:leeszczarba} fails for $n=3$ and $p \geq 7$.  The proofs
of the results in \cite{AshAnnouncement} were never published, but they follow easily from
the results in \cite{LeeSchwermer}.  The methods of these papers were specific to $n=3$, and
it seems hard to extend them to higher $n$.

Our main theorem completely characterizes when Conjecture \ref{conjecture:leeszczarba} holds:

\begin{maintheorem}
\label{theorem:main}
For a prime $p$ and $n \geq 2$, the map
\begin{equation}
\label{eqn:weprove}
(\St_n(\Q))_{\Gamma_n(p)} \longrightarrow \RH_{n-2}(\cT_n(\Q) / \Gamma_n(p))
\end{equation}
induced by \eqref{eqn:stnquotient} is a surjection.  However, it is an injection if and only if $p \leq 5$.
\end{maintheorem}

We thus see that Conjecture \ref{conjecture:leeszczarba} is true for $p \leq 5$, but is false for larger
primes.  In addition to dealing with the case $p=5$, our techniques also give a new proof for $p=2$ and
$p=3$.

\p{Even more cohomology}
Our proof that \eqref{eqn:weprove} is not injective for $p>5$ actually gives explicit new
cohomology classes, which allow us to give the following even better lower bound on the rank
of $\HH^{\binom{n}{2}}(\Gamma_n(p))$ for $p>5$.  For a vector space $V$,
let $\Gr_k(V)$ be the Grassmannian of $k$-dimensional subspaces of $V$.
See Theorem \ref{theorem:recursive} below
for a calculation of the rank of $\RH_{n-2}(\cT_n(\Q) / \Gamma_n(p))$, which
shows up in the following theorem.

\begin{maintheorem}
\label{theorem:evenbetter}
Fix a prime $p \geq 3$.  For $n \geq 1$, let $t_n$ be the rank of $\RH_{n-2}(\cT_n(\Q) / \Gamma_n(p))$.  Also, set $t_0 = 1$.  Then for $n \geq 3$, the rank of
$\HH^{\binom{n}{2}}(\Gamma_n(p))$ is at least
\[t_n+\frac{(p+2)(p-3)(p-5)(p-1)}{24} \cdot |\Gr_2(\Field_p^n)| \cdot t_{n-2}\]
with equality if $p=3$ or $p=5$.
\end{maintheorem}

\begin{remark}
The astute reader will notice that the term $\frac{(p+2)(p-3)(p-5)}{24}$ in Theorem \ref{theorem:evenbetter} is
the genus of the modular curve of level $p$.  This is not a coincidence. Indeed, we construct a surjective homomorphism from the kernel of $$ \HH^{\binom{n}{2}}(\Gamma_n(p)) \m \RH_{n-2}(\cT_n(\Q) / \Gamma_n(p))  $$ to a group obtained by inducing up the first homology group of the level $p$-modular curve tensored with $\RH_{n-4}(\cT_{n-2}(\Q) / \Gamma_n(p))  $. 
\end{remark}

\p{Size of quotient space}
Recall that the rank of $\RH_{n-2}(\cT_n(\Field_p))$ is
$p^{\binom{n}{2}}$.
There does not seem to be a similar simple closed form expression for the rank of
$\RH_{n-2}(\cT_n(\Q) / \Gamma_n(p))$.  However, we will establish the following
recursive formula for it.  

\begin{maintheorem}
\label{theorem:recursive}
Fix a prime $p \geq 3$.  For $n \geq 1$, let $t_n$ be the rank of $\RH_{n-2}(\cT_n(\Q)/\Gamma_n(p))$.
We then have $t_1 = 1$ and
\[t_n = \left(\frac{p-3}{2}+\left(\frac{p-1}{2}\right)\cdot p^{n-1}\right) t_{n-1}
+ \frac{(p-1)(p-3)}{4} \sum_{k=1}^{n-2} p^k \cdot |\Gr_k(\Field_p^{n-1})| \cdot t_k t_{n-k-1}\]
for $n \geq 2$.
\end{maintheorem}

\begin{remark}
An easy calculation shows that $|\Gr_k(\Field_p^n)| = \prod_{i=0}^{k-1} \frac{p^n-p^i}{p^k-p^i}$.
\end{remark}

\p{Relation with Paraschivescu's bound}
Recall from above that Paraschivescu \cite{Para} proved that for $p \geq 3$, the rank
of $\HH^{\binom{n}{2}}(\Gamma_n(p)) = (\St_n(\Q))_{\Gamma_n(p)}$ is at least
\[t'_n = \left(\frac{p-1}{2}\right)^{n-1} p^{\binom{n}{2}}.\]
Letting $t_n$ be as in Theorem \ref{theorem:recursive}, Theorem \ref{theorem:evenbetter} 
shows that the rank of $\HH^{\binom{n}{2}}(\Gamma_n(p))$ 
is also at least $t_n$.  For $p \geq 5$ and $n \geq 2$, our bound $t_n$ is always stronger than
Paraschivescu's bound $t'_n$.  Indeed, $t'_n$ satisfies the recursive formula
\[t'_1 = 1 \quad \text{and} \quad t'_n = \left(\frac{p-1}{2}\right) p^{n-1} t'_{n-1} \quad \quad (n \geq 2).\]
We thus have $t_1 = t'_1$, and for $n \geq 2$ and $p \geq 5$ we have
\begin{align*}
t_n &= \left(\frac{p-3}{2}+\left(\frac{p-1}{2}\right)\cdot p^{n-1}\right) t_{n-1}
+ \frac{(p-1)(p-3)}{4} \sum_{k=1}^{n-2} p^k \cdot |\Gr_k(\Field_p^{n-1})| \cdot t_k t_{n-k-1}\\
&> \left(\frac{p-1}{2}\right)p^{n-1} t_{n-1} \geq \left(\frac{p-1}{2}\right)p^{n-1} t'_{n-1} = t'_n.
\end{align*}

\p{Comments on bounds}
To the best of our knowledge, Theorem \ref{theorem:evenbetter} gives the best known lower bounds on these ranks
for general $n$.  It gives a complete calculation of $\HH^{\binom{n}{2}}(\Gamma_n(5))$; see
Table \ref{table:5ranks} for a table of values for $n\le 15$.  This table was produced in less than a second
using a personal computer, which can compute all $t_n$ for $n \leq 200$ within a minute.
There have been extensive computer calculations
of the cohomology of finite-index subgroups of $\SL_n(\Z)$ for small $n$ 
using the theory of Voronoi tessellations (see, e.g.\ 
\cite{AGM1,AGM2,AGM3}).  However, for $n \geq 5$ we believe that the 
computation in Theorem \ref{theorem:evenbetter} is beyond the reach of such 
computer calculations with present technology using those techniques.  

\begin{table}
\begin{footnotesize}
\begin{tabular}{r|r}
$n$ & $\rk \HH^{\binom{n}{2}}(\Gamma_n(5))$\\
\hline
1&1\\
2&11\\
3&621\\
4&176331\\
5&250654141\\
6&1781972405051\\
7&63346001119010061\\
8&11259312615761079960171\\
9&10006344346503001479394156381\\
10&44464067922769996760030750509009691\\
11&987899991107026778582667588995859270541101\\
12&109745515200463561297438405787408294210000904481611\\
13&60957982865169441101378571385234702783255341037103258372221\\
14&169295103797089744818524470008237065225058191012577153712309414663931\\
15&2350867829470159774034814041007591566603522538519291648712545382850352884817741
\end{tabular}
\end{footnotesize}
\caption{Calculations of the ranks of $\HH^{\binom{n}{2}}(\Gamma_n(5))$ for $n\le 15$.}
\label{table:5ranks}
\end{table}

\p{Outline}
The proof of Theorem \ref{theorem:main} has two main ingredients: connectivity/non-connectivity results for 
certain simplicial complexes built from bases of $\Field_p^n$, and a spectral sequence argument. 

The connectivity results are proven in \S \ref{section:simplicialcomplexes},
where the primary difficulty is the case $p=5$.  For $p=2$ or $3$, the field $\Field_p$ has the property 
that every unit lifts to a unit in $\Z$, which greatly simplifies the arguments.  Although this is not 
true for $p=5$, there still are not ``too many'' units that do not come from units in $\Z$.  For example, 
a key property about the number $5$ that we use is that if $a$ and $b$ are units in $\Field_5$ which 
do not lift to units in $\Z$, then there is a choice of signs such that $1=\pm a \pm b$. 

The spectral sequence arguments are in \S \ref{section:proofs}, which contains the proof
of Theorem \ref{theorem:main}.  For $p \leq 5$, this spectral sequence argument 
is relatively standard and is similar to the one used
by Church--Putman \cite{CP}.  However, for $p>5$, it is more novel.  We use the 
failure of certain simplicial complexes to be highly acyclic to produce elements in the kernel of the 
map 
\[\HH^{\binom{n}{2}}(\Gamma_n(p)) \longrightarrow \RH_{n-2}(\cT_n(\Q)/\Gamma_n(p)).\] 
These classes in the kernel that we describe are all induced up in some sense from classes in the 
kernel for $n=2$.  This new spectral sequence argument that we introduce in this paper has had 
applications to other questions concerning the cohomology of arithmetic groups (see, e.g., \cite{MPWY}). 

We close with the computational \S \ref{section:computation}, which proves Theorems \ref{theorem:evenbetter} and \ref{theorem:recursive}.

\p{Acknowledgments}
We thank Ben McReynolds and Simon Rubinstein--Salzedo for helpful conversations.  We would also like to thank
Avner Ash, Tom Church, and Paul Gunnells for helpful comments on previous versions of this paper.
Finally, the second and third authors would like to thank MSRI for their hospitality.

\section{Simplicial complexes associated to free \texorpdfstring{$R$}{R}-modules}
\label{section:simplicialcomplexes}

Fix a commutative ring $R$.  Our proof will require various simplicial complexes associated to a free $R$-module.
The rings we will make serious use of are $R = \Z$ and $R$ a field.  

\subsection{Complexes of bases and augmented bases}
\label{section:complexesbases}

We start by discussing four versions of these complexes: the complexes
of partial $R^{\times}$-bases, augmented partial $R^{\times}$-bases, partial $\pm$-bases,
and augmented partial $\pm$-bases.

\subsubsection{Partial \texorpdfstring{$R^{\times}$}{Rx}-bases}
\label{section:rxbases}

Let $R^{\times}$ be the set of units in $R$.  We make the following definition.

\begin{definition}
An {\em $R^{\times}$-vector} in $R^n$ is a set of the form $\Set{$c \vec{v}$}{$c \in R^{\times}$}$
for a nonzero $\vec{v} \in R^n$.  Given a nonzero $\vec{v} \in R^n$, we will write $[\vec{v}]$ for the
associated $R^{\times}$-vector.
\end{definition}

We then make the following definition.

\begin{definition}
A {\em partial basis} for $R^n$ is a set
of elements of $R^n$ that is a subset of a free basis for $R^n$.  A
{\em partial $R^{\times}$-basis} for $R^n$ is a set $\{[\vec{v}_1],\ldots,[\vec{v}_k]\}$ of
$R^{\times}$-vectors in $R^n$ such that
the set $\{\vec{v}_1,\ldots,\vec{v}_k\}$ is a partial
basis for $R^n$.  This does not depend on the choice of the representatives $\vec{v}_i$.
\end{definition}

We now turn these into a simplicial complex as follows.  Here and throughout the rest of this paper, we will define
simplicial complexes by specifying that their simplices are certain sets.  What we mean by this is that
the $k$-simplices are such sets containing $(k+1)$-elements, and the face relations between simplices
are simply inclusions of sets.

\begin{definition}
The {\em complex of partial $R^{\times}$-bases}
for $R^n$, denoted $\BT_n(R)$, is the simplicial complex whose simplices are
partial $R^{\times}$-bases for $R^n$.
\end{definition}

To understand $\BT_n(R)$ in an inductive way, we will need to understand links of simplices
in it.  We thus make the following definition.

\begin{definition}
Let $\{\vec{e}_1,\ldots,\vec{e}_{n+m}\}$
be the standard basis for $R^{n+m}$.  Define
$\BT_{n,m}(R) = \link_{\BT_{n+m}(R)} \{[\vec{e}_1],\ldots,[\vec{e}_m]\}$.
\end{definition}

Recall that a simplicial complex $X$ is {\em Cohen--Macaulay} of dimension $r$ if
it satisfies the following conditions:
\begin{compactitem}
\item $X$ is $r$-dimensional and $(r-1)$-connected, and
\item for all $k$-simplices $\sigma$ of $X$, the complex $\link_X(\sigma)$ is
$(r-k-1)$-dimensional and $(r-k-2)$-connected.
\end{compactitem}
Church--Putman \cite{CP} proved the following.

\begin{theorem}[{\cite[Theorem 4.2]{CP}}]
\label{theorem:ztbasescon}
The complex $\BT_{n,m}(\Z)$ is Cohen--Macaulay of dimension $(n-1)$ for all $n,m \geq 0$.
\end{theorem}

We will need the analogous fact with $\Z$ replaced by a field:

\begin{proposition}
\label{proposition:tbasescon}
For a field $\Field$, the complex $\BT_{n,m}(\Field)$ is Cohen--Macaulay of dimension $(n-1)$
for all $n, m \geq 0$.
\end{proposition}
\begin{proof}
Since the link of a $k$-simplex in $\BT_{n,m}(\Field)$ is isomorphic to $\BT_{n-k-1,m+k+1}(\Field)$,
it is enough to prove that $\BT_{n,m}(\Field)$ is $(n-2)$-connected for all $n,m \geq 0$.
Let $\LetterB_{n,m}(\Field)$ be the complex defined just like $\BT_{n,m}(\Field)$ but
using actual vectors rather than $R^{\times}$-vectors.  By
\cite[Theorem 2.6]{VanDerKallen}, the complex $\LetterB_{n,m}(\Field)$ is
$(n-2)$-connected.  Let $\pi\colon \LetterB_{n,m}(\Field) \rightarrow \BT_{n,m}(\Field)$
be the projection.  For each $R^{\times}$-vector $v$ in $\Field^{n+m}$, choose
an arbitrary $\vec{v} \in v$.  We can then define a simplicial map
$\phi\colon \BT_{n,m}(\Field) \rightarrow \LetterB_{n,m}(\Field)$ via the formula
$\phi(v) = \vec{v}$ for all vertices $v$ of $\BT_{n,m}(\Field)$.  We clearly
have $\pi \circ \phi = \text{id}$, so $\phi$ is an embedding and
$\pi$ is a retraction of $\LetterB_{n,m}(\Field)$ onto the image of $\phi$.
This implies that $\BT_{n,m}(\Field)$ is $(n-2)$-connected, as desired.
\end{proof}

\begin{remark}
Rather than deducing Proposition \ref{proposition:tbasescon} from
\cite[Theorem 2.6]{VanDerKallen}, it could instead be proved by
imitating the proof of \cite[Theorem 4.2]{CP}.  We
proved it the way we did above to emphasize that the essential core
of the result was in \cite{VanDerKallen}.
\end{remark}

\subsubsection{Augmented partial \texorpdfstring{$R^{\times}$}{Rx}-bases}
\label{section:augmentedrxbases}

We now add certain simplices to $\BT_{n,m}(R)$.  The key definition is as follows.

\begin{definition}
An {\em augmented partial $R^{\times}$-basis} for $R^n$ is a
set $\{[\vec{v}_0],\ldots,[\vec{v}_k]\}$ of $R^{\times}$-vectors in $R^n$ that can be reordered
such that the following hold:
\begin{compactitem}
\item $\{[\vec{v}_1],\ldots,[\vec{v}_k]\}$ is a partial $R^{\times}$-basis for $R^n$.
\item There exist
units $\lambda,\nu \in R^{\times}$ such that $\vec{v}_0 = \lambda \vec{v}_1 + \nu \vec{v}_2$.
The existence of $\lambda$ and $\nu$ does not depend on the choice of the
representatives $\vec{v}_1$ and $\vec{v}_2$.
\end{compactitem}
We will call the $R^{\times}$-vectors $\{[\vec{v}_0],[\vec{v}_1],[\vec{v}_2]\}$
the {\em additive core} of $\{[\vec{v}_0],\ldots,[\vec{v}_k]\}$.
\end{definition}

\begin{remark}
The additive core of an augmented partial $R^{\times}$-basis $\sigma$ for $R^n$ can be
characterized intrinsically as the set of all $v \in \sigma$ such that
$\sigma \setminus \{v\}$ is a partial $R^{\times}$-basis for $R^n$.
\end{remark}

A subset of an augmented partial $R^{\times}$-basis is either an augmented partial $R^{\times}$-basis
(if the subset contains the entire additive core) or a partial $R^{\times}$-basis
(if the subset does not contain the entire additive core).  We thus can make the following definition.

\begin{definition}
The {\em complex of augmented partial $R^{\times}$-bases} for
$R^n$, denoted $\BTA_n(R)$, is the simplicial complex whose simplices consist
of partial $R^{\times}$-bases and augmented partial $R^{\times}$-bases for $R^n$.
\end{definition}

Again, we will need to study links of simplices in $\BTA_n(R)$.  However, for technical reasons
we will not study the entire link, but rather the following subcomplex of it.

\begin{definition}
Let $\sigma = \{[\vec{v}_1],\ldots,[\vec{v}_k]\}$ be a simplex of $\BTA_n(R)$.
The {\em augmented link} of $\sigma$, denoted
$\hlink_{\BTA_n(R)}(\sigma)$, is the full subcomplex of $\link_{\BTA_n(R)}(\sigma)$ spanned
by vertices $[\vec{w}]$ of $\link_{\BTA_n(R)}(\sigma)$ such that
$\vec{w} \notin \Span{\vec{v}_1,\ldots,\vec{v}_k}$.  This
definition does not depend on the choice of the representatives $\vec{v}_i$ or $\vec{w}$.
\end{definition}

The simplices of $\hlink_{\BTA_n(R)}(\sigma)$ fall into the following three classes:

\begin{definition}
Let $\sigma = \{[\vec{v}_1],\ldots,[\vec{v}_k]\}$ be a
simplex of $\BTA_n(R)$.  Let $\eta$ be a simplex of $\hlink_{\BTA_n(R)}(\sigma)$.
Then one of the following three conditions hold.
\begin{compactitem}
\item $\eta$ is a partial $R^{\times}$-basis for $R^n$.  We will then call $\eta$ a
{\em standard simplex}.
\item $\eta$ is an augmented partial $R^{\times}$-basis for $R^n$, i.e.\ we can write
$\eta = \{[\vec{w}_0],\ldots,[\vec{w}_{\ell}]\}$ such that
$\vec{w}_0 = \lambda \vec{w}_1 + \nu \vec{w}_2$ with $\lambda, \nu \in R^{\times}$.
We will then call $\eta$ an {\em internally additive simplex}.
\item We can write $\eta = \{[\vec{w}_0],\ldots,[\vec{w}_{\ell}]\}$
with $\vec{w}_0 = \lambda \vec{w}_1 + \nu \vec{v}_i$ for some $\lambda,\nu \in R^{\times}$
and some $1 \leq i \leq k$.  We will then call $\eta$ an {\em externally additive simplex}.
\end{compactitem}
We will sometimes call a simplex that is either internally or externally additive
simply an {\em additive simplex}.
\end{definition}

Just like for $\BT_{n,m}(R)$, we make the following definition.

\begin{definition}
Let $\{\vec{e}_1,\ldots,\vec{e}_{n+m}\}$
be the standard basis for $R^{n+m}$.  Define
$\BTA_{n,m}(R) = \hlink_{\BTA_{n+m}(R)} \{[\vec{e}_1],\ldots,[\vec{e}_m]\}$.
\end{definition}

Church--Putman \cite{CP} proved the following, which is an analogue of Theorem \ref{theorem:ztbasescon} for
$\BTA_{n,m}(\Z)$.

\begin{theorem}[{\cite[Theorem C$'$]{CP}}]
\label{theorem:zaugtbasescon}
The complex $\BTA_{n,m}(\Z)$ is Cohen--Macaulay of dimension $n$ for all $n \geq 1$ and $m \geq 0$ with
$n+m \geq 2$.
\end{theorem}

We will need the analogous fact with $\Z$ replaced by a field:

\begin{proposition}
\label{proposition:augtbasescon}
For a field $\Field$, the complex $\BTA_{n,m}(\Field)$ is Cohen--Macaulay of dimension $n$
for all $n \geq 1$ and $m \geq 0$ with $n+m \geq 2$.
\end{proposition}
\begin{proof}
We start by briefly describing the proof of \cite[Theorem C$'$]{CP}.
Consider a vertex $v$ of $\BTA_{n,m}(\Z)$.  Pick $\vec{v} \in v$, and write
$\vec{v} = (a_1,\ldots,a_{n+m}) \in \Z^{n+m}$.  Define
\[R(v) = |a_{n+m}| \in \Z_{\geq 0}.\]
This does not depend on the choice of $\vec{v}$.  In
\cite[Theorem C$'$]{CP}, the function $R(v)$ is used
as a sort of Morse function, and spheres in $\BTA_{n,m}(\Z)$ are homotoped
so as to decrease the largest value of $R(v)$ as $v$ ranges over the vertices
of the sphere.  This homotopy makes use of the division algorithm in $\Z$ via
the observation that if $\{v_1,\ldots,v_k\}$ is a standard simplex of
$\BTA_{n,m}(\Z)$ with $R(v_1) > 0$ and if $\vec{v}_i \in v_i$ are representatives, then we
can find $\nu_2,\ldots,\nu_k \in \Z$ such that letting
$\vec{v}'_i = \vec{v}_i + \nu_i \vec{v}_1$ and $v'_i = [\vec{v}'_i]$
for $2 \leq i \leq k$, we have
\[R(v'_i) < R(v_1) \quad \quad (2 \leq i \leq k),\]
while $\{v_1,v_2',\ldots,v_k'\}$ is still a standard simplex of $\BTA_{n,m}(\Z)$.  To
extend this to the augmented simplices, an elaborate analysis of the process
of carrying during integer multiplication is required.

A very similar (but much easier proof) works for $\BTA_{n,m}(\Field)$.  The
appropriate complexity function $R$ is defined as follows.  Consider a vertex
$v$ of $\BTA_{n,m}(\Field)$.  Pick $\vec{v} \in v$, and write
$\vec{v} = (a_1,\ldots,a_{n+m}) \in \Field^{n+m}$.  Define
\[R(v) = \begin{cases}
1 & \text{if $a_{n+m} \neq 0$},\\
0 & \text{if $a_{n+m} = 0$}.
\end{cases}\]
The division algorithm is much easier in a field.  Indeed, the appropriate
analogue of the above fact is that if $\{v_1,\ldots,v_k\}$ is a standard simplex of
$\BTA_{n,m}(\Field)$ with $R(v_1) = 1$ and if $\vec{v}_i \in v_i$ are representatives, then we
can find $\nu_2,\ldots,\nu_k \in \Field$ such that letting
$\vec{v}'_i = \vec{v}_i + \nu_i \vec{v}_1$ and $v'_i = [\vec{v}'_i]$
for $2 \leq i \leq k$, we have
\[R(v'_i) = 0 \quad \quad (2 \leq i \leq k),\]
while $\{v_1,v_2',\ldots,v_k'\}$ is still a standard simplex of $\BTA_{n,m}(\Z)$.
With this definition, the entire proof of \cite[Theorem C$'$]{CP}
goes through with minimal changes.  We omit the details.
\end{proof}

\subsubsection{Partial \texorpdfstring{$\pm$}{plus-minus}-bases}
\label{section:pmbases}

We now turn to a different complex where we only
allow multiplication by $-1$.  We start with the following. 

\begin{definition}
A {\em $\pm$-vector} in $R^n$ is a set $v = \{\vec{v},-\vec{v}\}$ with
$\vec{v} \in R^n$ nonzero.  Given a nonzero $\vec{v} \in R^n$, we will write $\pm \vec{v}$ for
the associated $\pm$-vector $\{\vec{v},-\vec{v}\}$.
\end{definition}

We then make the following definition.

\begin{definition}
A {\em partial $\pm$-basis} for $R^n$ is a set $\{\pm \vec{v}_1,\ldots,\pm \vec{v}_k\}$ of
$\pm$-vectors in $R^n$ such that 
the set $\{\vec{v}_1,\ldots,\vec{v}_k\}$ is a partial
basis for $R^n$.  This does not depend on the choice of the representatives $\vec{v}_i$.
\end{definition}

These form a simplicial complex:

\begin{definition}
The {\em complex of partial $\pm$-bases}
for $R^n$, denoted $\B_n(R)$, is the simplicial complex whose simplices are
partial $\pm$-bases for $R^n$.
\end{definition}

To understand $\B_n(R)$ in an inductive way, we will need to understand links of simplices
in it.  We thus make the following definition.

\begin{definition}
Let $\{\vec{e}_1,\ldots,\vec{e}_{n+m}\}$
be the standard basis for $R^{n+m}$.  Define 
$\B_{n,m}(R) = \link_{\B_{n+m}(R)} \{\pm \vec{e}_1,\ldots,\pm \vec{e}_m\}$.
\end{definition}

\begin{remark}
Since $\Z^{\times} = \{\pm 1\}$, we have $\B_{n,m}(\Z) = \BT_{n,m}(\Z)$.
\end{remark}

The following is the analogue for $\B_{n,m}(\Field)$ of Proposition \ref{proposition:tbasescon}.

\begin{proposition}
\label{proposition:basescon}
For a field $\Field$, the complex $\B_{n,m}(\Field)$ is Cohen--Macaulay of dimension $(n-1)$
for all $n, m \geq 0$.
\end{proposition}
\begin{proof}
We compare $\B_{n,m}(\Field)$ with $\BT_{n,m}(\Field)$.  For each vertex $v$ of $\BT_{n,m}(\Field)$, choose
an arbitrary element $\vec{v} \in v$.  Define $\Lambda = \Field^{\times} / \{\pm 1\}$, considered as a multiplicative
group.  Vertices of $\B_{n,m}(\Field)$ can then be identified with pairs $(v,\lambda)$ with $v$ a vertex of $\BT_{n,m}(\Field)$
and $\lambda \in \Lambda$ via the identification that takes $(v,\lambda)$ to $\pm (\lambda \vec{v})$.  This
expression makes sense even though $\lambda \in \Lambda$ rather than $\Field^{\times}$ since we are considering
$\pm$-vectors.  Under this identification, a set $\{(v_0,\lambda_0),\ldots,(v_k,\lambda_k)\}$ forms a
$k$-simplex precisely when $\{v_0,\ldots,v_k\}$ forms a $k$-simplex of $\BT_{n,m}(\Field)$ (which implies that
the $v_i$ are distinct).  This means that $\B_{n,m}(\Field)$ is a {\em complete join complex} over
$\BT_{n,m}(\Field)$ in the sense of \cite[Definition 3.2]{HatcherWahl}.  
Proposition \ref{proposition:tbasescon} says that $\BT_{n,m}(\Field)$
is Cohen--Macaulay of dimension $(n-1)$.  
In \cite[Proposition 3.5]{HatcherWahl}, Hatcher--Wahl proved that if $Y$ is a complete
join complex over a complex $X$ that is Cohen--Macaulay of dimension $r$, then
$Y$ is also Cohen--Macaulay of dimension $r$.  Applying this to our situation, we deduce
that $\B_{n,m}(\Field)$ is Cohen--Macaulay of dimension $(n-1)$, as desired.
\end{proof}

\subsubsection{Augmented \texorpdfstring{$\pm$}{plus-minus}-bases}
\label{section:augpmbases}

We now define the augmented version of $\B_{n,m}(R)$.  The key definition is as follows.

\begin{definition}
An {\em augmented partial $\pm$-basis} for $R^n$ is a
set $\{\pm \vec{v}_0,\ldots,\pm \vec{v}_k\}$ of $\pm$-vectors in $R^n$ that can be reordered
such that the following hold:
\begin{compactitem}
\item $\{\pm \vec{v}_1,\ldots,\pm \vec{v}_k\}$ is a partial $\pm$-basis for $R^n$.
\item There exist
units $\lambda,\nu \in R^{\times}$ such that $\vec{v}_0 = \lambda \vec{v}_1 + \nu \vec{v}_2$.
The existence of $\lambda$ and $\nu$ does not depend on the choice of the
representatives $\vec{v}_1$ and $\vec{v}_2$ -- making the other choice merely 
multiplies them by $-1$.
\end{compactitem}
We will call the $\pm$-vectors $\{\pm \vec{v}_0,\pm \vec{v}_1,\pm \vec{v}_2\}$ 
the {\em additive core} of $\{\pm \vec{v}_0,\ldots,\pm \vec{v}_k\}$.
\end{definition}

\begin{remark}
The additive core of an augmented partial $\pm$-basis $\sigma$ for $R^n$ can be
characterized intrinsically as the set of all $v \in \sigma$ such that
$\sigma \setminus \{v\}$ is a partial $\pm$-basis for $R^n$.
\end{remark}

A subset of an augmented partial $\pm$-basis is either an augmented partial $\pm$-basis
(if the subset contains the entire additive core) or a partial $\pm$-basis
(if the subset does not contain the entire additive core).  We thus can make the following definition.

\begin{definition}
The {\em complex of augmented partial $\pm$-bases} for
$R^n$, denoted $\BA_n(R)$, is the simplicial complex whose simplices consist
of partial $\pm$-bases and augmented partial $\pm$-bases for $R^n$.
\end{definition}

Again, we will need to study links of simplices in $\BA_n(R)$.  Just like
for $\BTA_n(R)$, we make the following definition.

\begin{definition}
Let $\sigma = \{\pm \vec{v}_1,\ldots,\pm \vec{v}_k\}$ be a simplex of $\BA_n(R)$.
The {\em augmented link} of $\sigma$, denoted
$\hlink_{\BA_n(R)}(\sigma)$, is the full subcomplex of $\link_{\BA_n(R)}(\sigma)$ spanned
by vertices $\pm \vec{w}$ of $\link_{\BA_n(R)}(\sigma)$ such that
$\vec{w} \notin \Span{\vec{v}_1,\ldots,\vec{v}_k}$.  This
definition does not depend on the choice of the representatives $\vec{v}_i$ or $\vec{w}$.
\end{definition}

The simplices of $\hlink_{\BA_n(R)}(\sigma)$ fall into the following three classes:

\begin{definition}
Let $\sigma = \{\pm \vec{v}_1,\ldots,\pm \vec{v}_k\}$ be a 
simplex of $\BA_n(R)$.  Let $\eta$ be a simplex of $\hlink_{\BA_n(R)}(\sigma)$.  
Then one of the following three conditions hold.
\begin{compactitem}
\item $\eta$ is a partial $\pm$-basis for $R^n$.  We will then call $\eta$ a
{\em standard simplex}.
\item $\eta$ is an augmented partial $\pm$-basis for $R^n$, i.e.\ we can write
$\eta = \{\pm \vec{w}_0,\ldots,\pm \vec{w}_{\ell}\}$ such that
$\vec{w}_0 = \lambda \vec{w}_1 + \nu \vec{w}_2$ with $\lambda, \nu \in R^{\times}$.  
We will then call $\eta$ an {\em internally additive simplex}.
\item We can write $\eta = \{\pm \vec{w}_0,\ldots,\pm \vec{w}_{\ell}\}$ 
with $\vec{w}_0 = \lambda \vec{w}_1 + \nu \vec{v}_i$ for some $\lambda,\nu \in R^{\times}$
and some $1 \leq i \leq k$.  We will then call $\eta$ an {\em externally additive simplex}.
\end{compactitem}
We will sometimes call a simplex that is either internally or externally additive
simply an {\em additive simplex}.
\end{definition}

Just like for $\BTA_{n,m}(R)$, we make the following definition.

\begin{definition}
Let $\{\vec{e}_1,\ldots,\vec{e}_{n+m}\}$
be the standard basis for $R^{n+m}$.  Define
$\BA_{n,m}(R) = \hlink_{\BA_{n+m}(R)} \{\pm \vec{e}_1,\ldots,\pm \vec{e}_m\}$.
\end{definition}

\begin{remark}
Since $\Z^{\times} = \{\pm 1\}$, we have $\BA_{n,m}(\Z) = \BTA_{n,m}(\Z)$.
\end{remark}

The analogue of Proposition \ref{proposition:augtbasescon} is the following.

\begin{proposition}
\label{proposition:augbasescon}
For a field $\Field$, the complex $\BA_{n,m}(\Field)$ is Cohen--Macaulay of dimension $n$
for all $n \geq 1$ and $m \geq 0$ with $n+m \geq 2$.
\end{proposition}
\begin{proof}
Just like we did in the proof of Proposition \ref{proposition:basescon}, 
we compare $\BA_{n,m}(\Field)$ with $\BTA_{n,m}(\Field)$.  For each vertex $v$ of $\BTA_{n,m}(\Field)$, choose
an arbitrary element $\vec{v} \in v$.  Define $\Lambda = \Field^{\times} / \{\pm 1\}$, considered as a multiplicative
group.  Vertices of $\BA_{n,m}(\Field)$ can then be identified with pairs $(v,\lambda)$ with $v$ a 
vertex of $\BTA_{n,m}(\Field)$ and $\lambda \in \Lambda$ via the identification that takes 
$(v,\lambda)$ to $\pm (\lambda \vec{v})$.  This expression makes sense even though $\lambda \in \Lambda$ 
rather than $\Field^{\times}$ since we are considering
$\pm$-vectors.  Under this identification, a set $\{(v_0,\lambda_0),\ldots,(v_k,\lambda_k)\}$ forms a
$k$-simplex precisely when $\{v_0,\ldots,v_k\}$ forms a $k$-simplex of $\BTA_{n,m}(\Field)$; for additive
simplices, the additive core of $\{(v_0,\lambda_0),\ldots,(v_k,\lambda_k)\}$ consists of the $v_i$
that lie in the additive core of $\{v_0,\ldots,v_k\}$.
This means that $\BA_{n,m}(\Field)$ is a complete join complex over
$\BTA_{n,m}(\Field)$ in the sense of \cite[Definition 3.2]{HatcherWahl}.  By \cite[Proposition 3.5]{HatcherWahl},
the fact that $\BA_{n,m}(\Field)$ is Cohen--Macaulay of dimension $n$ now follows from
Proposition \ref{proposition:augtbasescon}, which says that $\BTA_{n,m}(\Field)$ is Cohen--Macaulay of dimension $n$.
\end{proof}

\begin{remark}
It is tempting to try to prove Proposition \ref{proposition:augbasescon} by mimicking the proof
of the analogous result over $\Z$ from \cite{CP}, just like we did
for Proposition \ref{proposition:augtbasescon}.  Since we will only use Proposition \ref{proposition:augbasescon}
and not Proposition \ref{proposition:augtbasescon} later in the paper, this would allow us to 
totally ignore the complexes of $R^{\times}$-bases.  Unfortunately, it turns out that the
proof in \cite{CP} breaks down for $\BA_{n,m}(\Field)$ (a certain retraction
it uses breaks), so this strategy does not work.  This was why we ended up introducing
the complexes of $R^{\times}$-bases.
\end{remark}

\subsection{Complexes of determinant-\texorpdfstring{$1$}{1} partial \texorpdfstring{$\pm$}{plus-minus}-bases}
\label{section:det}

For our proof, what we really need are certain subcomplexes of the complexes of partial $\pm$-bases
where we impose a determinant condition.

\subsubsection{Determinant-\texorpdfstring{$1$}{1} partial \texorpdfstring{$\pm$}{plus-minus}-bases}
\label{section:det1}

We make the following definition.

\begin{definition}
A partial $\pm$-basis $\{\pm \vec{v}_1,\ldots,\pm \vec{v}_k\}$ for $R^n$ 
is a {\em determinant-$1$ partial $\pm$-basis} if it satisfies the following 
condition.
\begin{compactitem}
\item If $k = n$, then we require that the determinant of the matrix
$(\vec{v}_1\ \cdots\ \vec{v}_n)$ whose columns are the $\vec{v}_i$ is equal
to either $1$ or $-1$.  This does not depend on the choice of the
$\vec{v}_i$ or their ordering.
\item If $k < n$, then no additional condition needs to be satisfied.\qedhere
\end{compactitem}
\end{definition}

These form a simplicial complex:

\begin{definition}
The {\em complex of determinant-$1$ partial $\pm$-bases}
for $R^n$, denoted $\BD_n(R)$, is the simplicial complex whose simplices are
determinant-$1$ partial $\pm$-bases for $R^n$.
\end{definition}

Just like for $\B_n(R)$, we need to consider links as well:

\begin{definition}
Let $\{\vec{e}_1,\ldots,\vec{e}_{n+m}\}$
be the standard basis for $R^{n+m}$.  Define
$\BD_{n,m}(R) = \link_{\BD_{n+m}(R)} \{\pm \vec{e}_1,\ldots,\pm \vec{e}_m\}$.
\end{definition}

\begin{remark}
Since $\Z^{\times} = \{\pm 1\}$, we have $\BD_{n,m}(\Z) = \B_{n,m}(\Z) = \BT_{n,m}(\Z)$.
\end{remark}

With these definitions, we have the following key lemma.  Recall that $\Gamma_n(p)$
is the level-$p$ congruence subgroups of $\SL_n(\Z)$.

\begin{lemma}
\label{lemma:quotientbd}
For a prime $p$, we have $\B_{n}(\Z) / \Gamma_n(p) \cong \BD_n(\Field_p)$ for all
$n \geq 1$.
\end{lemma}
\begin{proof}
For a commutative ring $R$, the complex $\BD_n(R)$ can be viewed as the one
whose simplices are collections of $\pm$-vectors $\{v_1,\ldots,v_k\}$ in
$R^n$ such that there exist representatives $\vec{v}_i \in v_i$ that arise
as some of the columns in a matrix in $\SL_n(R)$.  We remark that we only need matrices
of determinant $1$ (rather than $\pm 1$) since we are free to multiply the $\vec{v}_i$
by $-1$ as needed.  In light of the fact that $\BD_{n}(\Z) = \B_n(\Z)$, the lemma 
now immediately follows from the classical fact that the group homomorphism
\[\SL_n(\Z) \longrightarrow \SL_n(\Field_p)\]
that reduces matrices modulo $p$ is surjective with kernel $\Gamma_n(p)$.
\end{proof}

\subsubsection{Augmented determinant-\texorpdfstring{$1$}{1} partial \texorpdfstring{$\pm$}{plus-minus}-bases}
\label{section:det1aug}

We make the following definition.

\begin{definition}
An {\em augmented determinant-$1$ partial $\pm$-basis} for $R^n$ is a
set $\{\pm \vec{v}_0,\ldots,\pm \vec{v}_k\}$ of $\pm$-vectors in $R^n$ that can be reordered
such that the following hold:
\begin{compactitem}
\item $\{\pm \vec{v}_1,\ldots,\pm \vec{v}_k\}$ is a determinant-$1$ 
partial $\pm$-basis for $R^n$.
\item There exist $\lambda,\nu \in \{\pm 1\}$ 
such that $\vec{v}_0 = \lambda \vec{v}_1 + \nu \vec{v}_2$.
\end{compactitem}
We will call the $\pm$-vectors $\{\pm \vec{v}_0,\pm \vec{v}_1,\pm \vec{v}_2\}$
the {\em additive core} of $\{\pm \vec{v}_0,\ldots,\pm \vec{v}_k\}$.
\end{definition}

\begin{remark}
The additive core of an augmented determinant-$1$ partial $\pm$-basis $\sigma$ for $R^n$ can be
characterized intrinsically as the set of all $v \in \sigma$ such that
$\sigma \setminus \{v\}$ is a determinant-$1$ partial $\pm$-basis for $R^n$.
\end{remark}

A subset of an augmented determinant-$1$ partial $\pm$-basis is either an augmented 
determinant-$1$ partial $\pm$-basis
(if the subset contains the entire additive core) or a 
determinant-$1$ partial $\pm$-basis
(if the subset does not contain the entire additive core; this uses the fact that the constants
$\lambda$ and $\nu$ are $\pm 1$ rather than general units).  
We thus can make the following definition.

\begin{definition}
The {\em complex of augmented determinant-$1$ partial $\pm$-bases} for
$R^n$, denoted $\BDA_n(R)$, is the simplicial complex whose simplices consist
of determinant-$1$ partial $\pm$-bases and augmented determinant-$1$
partial $\pm$-bases for $R^n$.
\end{definition}

We now make a series of definitions that are very similar to the ones we made
for $\BA_n(R)$.

\begin{definition}
Let $\sigma = \{\pm \vec{v}_1,\ldots,\pm \vec{v}_k\}$ be a simplex of $\BDA_n(R)$.
The {\em augmented link} of $\sigma$, denoted
$\hlink_{\BDA_n(R)}(\sigma)$, is the full subcomplex of $\link_{\BDA_n(R)}(\sigma)$ spanned
by vertices $\pm \vec{w}$ of $\link_{\BDA_n(R)}(\sigma)$ such that
$\vec{w} \notin \Span{\vec{v}_1,\ldots,\vec{v}_k}$.  This
definition does not depend on the choice of the representatives $\vec{v}_i$ or $\vec{w}$.
\end{definition}

The simplices of $\hlink_{\BDA_n(R)}(\sigma)$ fall into the following three classes:

\begin{definition}
Let $\sigma = \{\pm \vec{v}_1,\ldots,\pm \vec{v}_k\}$ be a
simplex of $\BDA_n(R)$.  Let $\eta$ be a simplex of $\hlink_{\BDA_n(R)}(\sigma)$.
Then one of the following three conditions hold.
\begin{compactitem}
\item $\eta$ is a determinant-$1$ partial $\pm$-basis for $R^n$.  We will then call $\eta$ a
{\em standard simplex}.
\item $\eta$ is an augmented determinant-$1$ partial $\pm$-basis for $R^n$, i.e.\ we can write
$\eta = \{\pm \vec{w}_0,\ldots,\pm \vec{w}_{\ell}\}$ such that
$\vec{w}_0 = \lambda \vec{w}_1 + \nu \vec{w}_2$ with $\lambda, \nu \in \{\pm 1\}$.
We will then call $\eta$ an {\em internally additive simplex}.
\item We can write $\eta = \{\pm \vec{w}_0,\ldots,\pm \vec{w}_{\ell}\}$
with $\vec{w}_0 = \lambda \vec{w}_1 + \nu \vec{v}_i$ for some $\lambda,\nu \in \{\pm 1\}$
and some $1 \leq i \leq k$.  We will then call $\eta$ an {\em externally additive simplex}.
\end{compactitem}
We will sometimes call a simplex that is either internally or externally additive 
simply an {\em additive simplex}.
\end{definition}

\begin{definition}
Let $\{\vec{e}_1,\ldots,\vec{e}_{n+m}\}$
be the standard basis for $R^{n+m}$.  Define
$\BDA_{n,m}(R) = \hlink_{\BDA_{n+m}(R)} \{\pm \vec{e}_1,\ldots,\pm \vec{e}_m\}$.
\end{definition}

\begin{remark}
Since $\Z^{\times} = \{\pm 1\}$, we have $\BDA_{n,m}(\Z) = \BA_{n,m}(\Z) = \BTA_{n,m}(\Z)$.
\end{remark}

The analogue of Lemma \ref{lemma:quotientbd} for $\BDA_n(\Z) = \BA_n(\Z)$ is as follows.

\begin{lemma}
\label{lemma:quotientbda}
For a prime $p$, we have 
$\BA_n(\Z) / \Gamma_n(p) \cong \BDA_{n}(\Field_p)$ for all $n \geq 2$.
\end{lemma}
\begin{proof}
For a standard simplex $\{\pm \vec{v}_1,\ldots,\pm \vec{v}_k\}$ of $\BA_n(\Z) = \BDA_n(\Z)$, there
are precisely two choices of $\pm \vec{v}_0$ such that
$\{\pm \vec{v}_0,\ldots,\pm \vec{v}_k\}$ is an additive simplex whose additive core
is $\{\pm \vec{v}_0,\pm \vec{v}_1,\pm \vec{v}_2\}$, namely
\[\pm \vec{v}_0 = \pm(\vec{v}_1 + \vec{v}_2) \quad \text{and} \quad \pm \vec{v}_0 = \pm(\vec{v}_1 - \vec{v}_2).\]
A similar observation holds for $\BDA_n(\Field_p)$ (unless $p=2$, in which case both of the above choices
are the same).  From this, the lemma easily follows from Lemma \ref{lemma:quotientbd}.
\end{proof}

\subsubsection{The case \texorpdfstring{$n=2$}{n=2}}
\label{section:det1n2}

We now specialize to the case $n=2$, where these complexes have a simple description.

\begin{lemma}
\label{lemma:n2}
For a prime $p \geq 3$, the complex $\BDA_2(\Field_p)$ is homeomorphic to a closed oriented surface of genus $\frac{(p+2)(p-3)(p-5)}{24}$.  Also, the complex $\BDA_2(\Field_2)$ is contractible.
\end{lemma}
\begin{proof}
The complex $\BDA_2(\Field_2)$ is easily seen to be a single triangle with vertices
$\pm(1,0)$ and $\pm(0,1)$ and $\pm(1,1)$, and is thus contractible.  Assume now that
$p$ is an odd prime.

Consider the usual bordification of the upper half plane $\bbH^2$ whose points are
\[\obbH^2 = \bbH^2 \cup (\Q \cup \{\infty\}) \subset \C \cup \{\infty\}.\]
In this bordification, the topology on $\obbH^2$ restricts to the usual topology on $\bbH^2$, but the topology
on $\obbH^2$ is not the subspace topology from $\C \cup \{\infty\}$, but rather one where open horoballs
centered at the ideal points $\Q \cup \{\infty\}$ form neighborhood bases of these ideal points.
The group $\SL_2(\Z)$ acts on $\obbH^2$ by linear fractional transformations, and the quotient
$\obbH^2 / \Gamma_2(p)$ is the level-$p$ modular curve.  This modular curve is a closed
oriented surface of genus $\frac{(p+2)(p-3)(p-5)}{24}$; see \cite[Theorem 8]{GunningBook}.  It is enough,
therefore, to prove that there is an $\SL_2(\Z)$-equivariant homeomorphism between
$\BDA_2(\Z)$ and $\obbH^2$.

This homeomorphism is implicit in \cite[pp.\ 1002--1004]{CP}, so we
only briefly describe it:
\begin{compactitem}
\item For a vertex $\pm (a,b)$ of $\BDA_2(\Z)$, the associated point of $\obbH^2$ is $a/b \in \Q \cup \{\infty\}$.
\item For an edge $e$ of $\BDA_2(\Z)$, the associated portion of $\obbH^2$ is the hyperbolic
geodesic joining the ideal points corresponding to the endpoints of $e$.
\item For a triangle $t$ of $\BDA_2(\Z)$, the associated portion of $\obbH^2$ is the
hyperbolic ideal triangle whose boundary consists of the geodesics corresponding to the boundary
of $t$.\qedhere
\end{compactitem}
\end{proof}

\subsubsection{The unaugmented determinant-\texorpdfstring{$1$}{1} complex is highly connected}
\label{section:det1con}

We now turn to proving that the complexes $\BD_{n,m}(\Field)$ are Cohen--Macaulay.

\begin{proposition}
\label{proposition:bases1con}
For a field $\Field$, the complex $\BD_{n,m}(\Field)$ is Cohen--Macaulay of dimension $(n-1)$ 
for all $n, m \geq 0$.
\end{proposition}

The heart of the proof of Proposition \ref{proposition:bases1con} is the following.

\begin{lemma}
\label{lemma:bases1retract}
For a field $\Field$, the complex $\BD_{n,m}(\Field)$ is a retract of
$\B_{n,m}(\Field)$ for all $n,m \geq 0$.
\end{lemma}

Before proving Lemma \ref{lemma:bases1retract}, we derive Proposition \ref{proposition:bases1con} from it.

\begin{proof}[Proof of Proposition \ref{proposition:bases1con}, assuming Lemma \ref{lemma:bases1retract}]
Combining Lemma \ref{lemma:bases1retract} with Proposition \ref{proposition:basescon} (which
says that $\B_{n,m}(\Field)$ is $(n-2)$-connected), we deduce that
$\BD_{n,m}(\Field)$ is $(n-2)$-connected.  Since
the link of a $k$-simplex in $\BD_{n,m}(\Field)$ is isomorphic to $\BD_{n-k-1,m+k+1}(\Field)$,
this implies that $\BD_{n,m}(\Field)$ is Cohen--Macaulay of dimension $(n-1)$, as desired.
\end{proof}

\begin{proof}[Proof of Lemma \ref{lemma:bases1retract}]
Let $\{\vec{e}_1,\ldots,\vec{e}_{n+m}\}$ be the standard basis for the vector space $\Field^{n+m}$.
To define a retraction $\rho \colon \B_{n,m}(\Field) \rightarrow \BD_{n,m}(\Field)$, it is enough
to say what $\rho$ does to a simplex $\sigma$ of $\B_{n,m}(\Field)$ that does not lie in
$\BD_{n,m}(\Field)$.  The only such simplices are $(n-1)$-dimensional simplices
$\sigma = \{v_1,\ldots,v_n\}$ such that $\{\pm \vec{e}_1,\ldots,\pm \vec{e}_m,v_1,\ldots,v_n\}$ is
not a determinant-$1$ total $\pm$-basis for $\Field^{n+m}$.
Arbitrarily pick some $\vec{v}_i \in v_i$ for $1 \leq i \leq n$, and let $d \neq \pm 1$ be
the determinant of the matrix 
\[(\vec{e}_1\ \ \ \cdots\ \ \ \vec{e}_m\ \ \ \vec{v}_1\ \ \ \cdots\ \ \ \vec{v}_n).\] 
Let $S(\sigma)$ be the result of subdividing $\sigma$ with a new vertex $x_{\sigma}$.  The
top-dimensional simplices of $S(\sigma)$ are then of the form
\[\{v_1,\ldots,\widehat{v_i},\ldots,v_n,x_{\sigma}\} \quad \quad (1 \leq i \leq n).\]
Define 
\[\rho|_{\sigma}\colon \sigma \cong S(\sigma) \longrightarrow \B_{n,m}(\Field)\]
to be the map that fixes the vertices $v_1,\ldots,v_n$ and takes the vertex $x_{\sigma}$
to $\frac{1}{d} (\vec{v}_1 + \cdots + \vec{v}_n)$.  We must check that this extends over
the top-dimensional simplices of $S(\sigma)$, which follows from the calculation
\begin{align*}
&\det\left(\vec{e}_1\ \ \ \cdots\ \ \ \vec{e}_m\ \ \ \vec{v}_1\ \ \ \cdots\ \ \ \widehat{\vec{v}_i}\ \ \ \cdots\ \ \ \vec{v}_n\ \ \ \frac{1}{d} \left(\vec{v}_1 + \cdots + \vec{v}_n\right)\right)\\
&\quad \quad = \frac{1}{d} \det\left(\vec{e}_1\ \ \ \cdots\ \ \ \vec{e}_m\ \ \ \vec{v}_1\ \ \ \cdots\ \ \ \widehat{\vec{v}_i}\ \ \ \cdots\ \ \ \vec{v}_n\ \ \ \vec{v}_i\right) = \pm d/d = \pm 1. \qedhere
\end{align*}
\end{proof}

\subsubsection{The augmented determinant-\texorpdfstring{$1$}{1} complex is highly connected}
\label{section:det1acon}

We now prove that the complex $\BDA_{n,m}(\Field)$ is $(n-2)$-connected.
We remark that it is $n$-dimensional, so this is a weaker range of connectivity than
would be implied by it being Cohen--Macaulay.

\begin{proposition}
\label{proposition:bases1acon}
For a field $\Field$, the complex $\BDA_{n,m}(\Field)$ is $(n-2)$-connected
for all $n, m \geq 0$.
\end{proposition}

\begin{remark}
For $\Field = \Field_p$ with $p \leq 5$, we will improve this to $(n-1)$-connected
in Proposition \ref{proposition:bases1conimproved} below.
\end{remark}

Proposition \ref{proposition:bases1con} implies that $\BD_{n,m}(\Field)$ is $(n-2)$-connected,
so Proposition \ref{proposition:bases1acon} is an immediate consequence of the following lemma.

\begin{lemma}
\label{lemma:bases1surject}
For a field $\Field$, the inclusion map $\BD_{n,m}(\Field) \hookrightarrow \BDA_{n,m}(\Field)$
induces a surjection on $\pi_k$ for $0 \leq k \leq n-1$ for all $n,m \geq 0$.
\end{lemma}
\begin{proof}
Let $X$ be a compact simplicial complex of dimension at most $(n-1)$ and let
$\phi\colon X \rightarrow \BDA_{n,m}(\Field)$ be a simplicial map.  It is enough to 
prove that $\phi$ can be homotoped to a map whose image is contained in
$\BD_{n,m}(\Field)$.

If the image of $\phi$ is not contained in $\BD_{n,m}(\Field)$, then the image of $\phi$
contains either a $2$-dimensional internally additive simplex or a $1$-dimensional
externally additive simplex.  Let $\sigma$ be a simplex of $X$ whose image is
of this form whose dimension $\ell$ is as large as possible.  
Since $\phi$ need not be injective, it might be the case that $\ell > \dim(\phi(\sigma)) \in \{1,2\}$.

Let $\ast$ be the simplicial join, so $\sigma \ast \link_X(\sigma) \subset X$.  Let
\[f\colon \sigma \ast \link_X(\sigma) \rightarrow \BDA_{n,m}(\Field)\]
be the restriction of $\phi$.  What we will do is construct a subdivision $Z$ of $\sigma \ast \link_X(\sigma)$
along with a map $g\colon Z \rightarrow \BDA_{n,m}(\Field)$ with the following properties:
\begin{compactitem}
\item No simplices of $\partial \sigma \ast \link_X(\sigma)$ are subdivided when forming $Z$.
\item $f$ and $g$ restrict to the same map on $\partial \sigma \ast \link_X(\sigma)$.
\item $f$ and $g$ are homotopic through maps fixing $\partial \sigma \ast \link_X(\sigma)$.
\item There are no simplices of dimension at least $\ell$ in $Z$ that map to either $2$-dimensional internally
simplices or $1$-dimensional externally additive simplices.
\end{compactitem}
From this, we see that we can subdivide $X$ to replace $\sigma \ast \link_X(\sigma)$ with $Z$ and then
homotope $\phi$ so as to replace $f$ by $g$.  This eliminates $\sigma$, and repeating this over and
over again homotopes $\phi$ to a map whose image is contained in $\BD_{n,m}(\Field)$, as desired.

It remains to construct $Z$ and $g$.  We will show how to do this when $\eta = \phi(\sigma)$ is a $2$-dimensional
internally additive simplex.  The case where $\eta$ is a $1$-dimensional externally additive simplex is similar.
Write $\eta = \{\pm \vec{v}_0,\pm \vec{v}_1,\pm \vec{v}_2\}$ with
$\vec{v}_0 = \lambda \vec{v}_1 + \nu \vec{v}_2$ for some $\lambda,\nu \in \{\pm 1\}$.  

Since the dimension of $\sigma$ is as large as possible, we have
\[f\left(\link_X\left(\sigma\right)\right) \subset \link_{\BDA_{n,m}(\Field)}\left(\eta\right).\]
Setting $\eta' = \{\pm \vec{v}_1, \pm\vec{v}_2\}$, the key observation is that
\[\link_{\BDA_{n,m}(\Field)}\left(\eta\right) = \link_{\BD_{n,m}(\Field)}\left(\eta'\right) \cong \BD_{n-2,m+2}(\Field).\]
Proposition \ref{proposition:bases1con} says that $\BD_{n-2,m+2}(\Field)$ is $(n-4)$-connected.  Since $X$
has dimension at most $(n-1)$ and $\sigma$ has dimension $\ell \geq 2$, the complex
$\link_X(\sigma)$ has dimension at most $(n-4)$.  We conclude that the map
\begin{equation}
\label{eqn:tocone}
\link_X\left(\sigma\right) \longrightarrow \link_{\BDA_{n,m}(\Field)}\left(\eta\right)
\end{equation}
obtained by restricting $f$ is nullhomotopic.

Letting $\{p_0\}$ denote a $1$-point space, we conclude that
\eqref{eqn:tocone} extends to a continuous map
\[F\colon \{p_0\} \ast \link_X\left(\sigma\right) \rightarrow \link_{\BDA_{n,m}(\Field)}\left(\eta\right)\]
that is simplicial with respect to some subdivision $Z'$ of its domain that does not subdivide any
simplices of $\link_X(\sigma)$.  Define
\[Z = \partial \sigma \ast Z' \cong \partial \sigma \ast \{p_0\} \ast \link_X\left(\sigma\right) \cong \sigma \ast \link_X\left(\sigma\right).\]
The $\cong$ here are topological homeomorphisms where the domain is a subdivision of the codomain.  Finally, define
$g\colon Z \rightarrow \BDA_{n,m}(\Field)$ to be
\[Z = \partial \sigma \ast Z' \cong \partial \sigma \ast \left(\{p_0\} \ast \link_X\left(\sigma\right)\right) \xrightarrow{f|_{\partial \sigma} \ast F} \BDA_{n,m}(\Field).\]
It is clear that this has the desired properties.
\end{proof}

\subsection{Improving the connectivity for small primes}
\label{section:smallprimes}

In this section, we show that the connectivity range for $\BDA_n(\Field_p)$ can be
improved for $p \leq 5$.  We state our result and give the skeleton of its
proof in \S \ref{section:smallprimesstatement}.  This depends on several lemmas
which are proved in subsequent sections.

\subsubsection{Statement and skeleton of proof}
\label{section:smallprimesstatement}

Our result is as follows.

\begin{proposition}
\label{proposition:bases1conimproved}
For a prime $p \leq 5$, the complex $\BDA_n(\Field_p)$ is $(n-1)$-connected
for $n \geq 1$.
\end{proposition}

\begin{remark}
For primes $p>5$, Lemma \ref{lemma:n2} implies that this is false for $n=2$.
We do not know whether or not it holds for $p>5$ and $n \geq 3$.
\end{remark}

\begin{proof}[{Skeleton of proof of Proposition \ref{proposition:bases1conimproved}}]
We outline the proof of the proposition, reducing it to several lemmas.  For $n=1$, the
complex $\BDA_n(\Field_p)$ is a single point and the proposition is trivial, so we can assume
that $n \geq 2$.  For $p \in \{2,3\}$, we
have $\BDA_n(\Field_p) = \BA_n(\Field_p)$, so the proposition follows from
Proposition \ref{proposition:augbasescon}.  We thus only need to deal with the case $p=5$.  

The proof will be by induction on $n$.  The base case $n=2$ follows from Lemma \ref{lemma:n2}, which
says that $\BDA_2(\Field_5)$ is homeomorphic to a $2$-sphere.  Assume now that $n>2$ and that
the result is true for all smaller $n$.  Proposition \ref{proposition:bases1acon} says that
$\BDA_n(\Field_5)$ is $(n-2)$-connected, so we must only show that $\pi_{n-1}(\BDA_n(\Field_5)) = 0$.

Lemma \ref{lemma:bases1surject} says that the inclusion $\iota\colon \BD_n(\Field_5) \hookrightarrow \BDA_n(\Field_5)$
induces a surjection on $\pi_{n-1}$, so it is enough to prove that it also induces the zero map
on $\pi_{n-1}$.  We will do this by identifying generators for $\pi_{n-1}(\BD_n(\Field_5))$ and
then showing that these generators all lie in the kernel of the map
$\iota_{\ast}\colon \pi_{n-1}(\BD_n(\Field_5)) \rightarrow \pi_{n-1}(\BDA_n(\Field_5))$.  Since $n \geq 3$, Proposition
\ref{proposition:bases1con} says that $\BD_n(\Field_5)$ is $1$-connected, so we can ignore
basepoints and represent elements of $\pi_{n-1}(\BD_n(\Field_5))$ by unbased maps of $(n-1)$-spheres
into $\BD_n(\Field_5)$.

Lemma \ref{lemma:bases1retract} says that there is a retraction
$\rho\colon \B_n(\Field_5) \rightarrow \BD_n(\Field_5)$, so if $S$ is a generating set
for $\pi_{n-1}(\B_n(\Field_5))$, then $\Set{$\rho_{\ast}(s)$}{$s \in S$}$ is a generating
set for $\pi_{n-1}(\BD_n(\Field_5))$.  To describe generators for $\pi_{n-1}(\B_n(\Field_5))$,
we first introduce some notation.

\begin{notation}
Let $X$ be a simplicial complex and let $\Delta^{k-1}$ be an $(k-1)$-simplex.
\begin{compactitem}
\item Let $v_1,\ldots,v_k$ be (not necessarily distinct) vertices of $X$ such that $\{v_1,\ldots,v_k\}$
is a simplex.  Define $\Disc{v_1,\ldots,v_k}$ to be the map
\[\Disc{v_1,\ldots,v_k}\colon \Delta^{k-1} \longrightarrow X\]
taking the vertices of $\Delta^{k-1}$ to the $v_i$.
\item Let $v_1,\ldots,v_k$ be (not necessarily distinct) vertices of $X$ such that $\{v_1,\ldots,\widehat{v_i},\ldots,v_k\}$
is a simplex of $X$ for all $1 \leq i \leq k$.  Define $\Sphere{v_1,\ldots,v_k}$ to be the map
\[\Sphere{v_1,\ldots,v_k}\colon \partial \Delta^{k-1} \longrightarrow X\]
taking the vertices of $\partial \Delta^{k-1}$ to the $v_i$.
\item Let $Y$ and $Z$ be simplicial complexes and let $f\colon Y \rightarrow X$ and $g\colon Z \rightarrow X$
be simplicial maps.  Assume that for all simplices $\sigma$ of $Y$ and $\eta$ of $Z$, the join
$f(\sigma) \ast g(\eta)$ is a simplex of $X$.  Then let $f \ast g$ denote the natural
map $f \ast g \colon Y \ast Z \rightarrow X$.
\end{compactitem}
\end{notation}

The following lemma now gives generators for
$\pi_{n-1}(\B_n(\Field_5))$.  It will be proved in \S \ref{section:identifygen}.
For a finite-dimensional $\Field_5$-vector space $V$, we write $\B(V)$ for the complex
of partial $\pm$-bases of $V$, so $\B_n(\Field_5) = \B(\Field_5^n)$.

\begin{lemma}
\label{lemma:identifygen}
For $n \geq 3$, the group $\pi_{n-1}(\B_n(\Field_5))$ is generated by the following two families
of generators.
\begin{compactitem}
\item The {\bf initial D-triangle maps}.  Let $\sigma = \{\pm \vec{v}_0, \pm \vec{v}_1, \pm \vec{v}_2\}$ be a $2$-dimensional
additive simplex of $\BA_n(\Field_5)$ with $\vec{v}_0 = \lambda \vec{v}_1+\nu \vec{v}_2$ for some $\lambda, \nu \in \{\pm 1\}$.
Let $f\colon S^{n-3} \rightarrow \link_{\BA_n(\Field_5)}(\sigma)$
be a simplicial map for some triangulation of $S^{n-3}$.  The associated initial D-triangle map is then
\[\Sphere{\pm \vec{v}_0, \pm \vec{v}_1, \pm \vec{v}_2} \ast f \colon \partial \Delta^2 \ast S^{n-3} \cong S^{n-1} \longrightarrow \B_n(\Field_5).\]
\item The {\bf initial D-suspend maps}.  Let $\vec{v} \in \Field_5^n$ be a nonzero vector, let $W \subset \Field_5^n$ 
be an $(n-1)$-dimensional subspace such that $\Field_5^n = \Span{\vec{v}} \oplus W$, and let $\vec{w} \in W$ be nonzero.  
Let $f\colon S^{n-2} \rightarrow \B(W)$ be a simplicial map for some triangulation of $S^{n-2}$.  The associated
initial D-suspend map is then
\[\Sphere{\pm \vec{v}, \pm (\vec{v}+\vec{w})} \ast f\colon \partial \Delta^1 \ast S^{n-2} \cong S^{n-1} \longrightarrow \B_n(\Field_5).\]
\end{compactitem}
\end{lemma}

\begin{remark}
The ``D'' in D-triangle and D-suspend maps are there to distinguish them from more general 
ones we will introduce in the next section.  Since $n \geq 3$ in Lemma \ref{lemma:identifygen},
the $\sigma$ in the definition of an initial D-triangle map is actually a simplex
of $\BDA_n(\Field_5)$; however, we define it like we did since later we will talk about
them when $n=2$, in which case we do {\em not} want to require a determinant condition.
\end{remark}

To finish the proof, it is now enough to prove the following two lemmas.

\begin{lemma}[Kill initial D-triangle maps]
\label{lemma:kill1}
For some $n \geq 3$, let $g\colon S^{n-1} \rightarrow \B_n(\Field_5)$ be an initial D-triangle map,
let $\rho\colon \B_n(\Field_5) \rightarrow \BD_n(\Field_5)$ be
the retraction given by Lemma \ref{lemma:bases1retract}, and let $\iota\colon \BD_n(\Field_5) \hookrightarrow \BDA_n(\Field_5)$
be the inclusion.  Then $\iota \circ \rho \circ g\colon S^{n-1} \rightarrow \BDA_n(\Field_5)$ is nullhomotopic.
\end{lemma}

\begin{lemma}[Kill initial D-suspend maps]
\label{lemma:kill2}
For some $n \geq 3$, let $g\colon S^{n-1} \rightarrow \B_n(\Field_5)$ be an initial D-suspend map,
let $\rho\colon \B_n(\Field_5) \rightarrow \BD_n(\Field_5)$ be
the retraction given by Lemma \ref{lemma:bases1retract}, and let $\iota\colon \BD_n(\Field_5) \hookrightarrow \BDA_n(\Field_5)$be the inclusion.  Assume that $\pi_{n-2}(\BDA_{n-1}(\Field_5)) = 0$.
Then $\iota \circ \rho \circ g\colon S^{n-1} \rightarrow \BDA_n(\Field_5)$ is nullhomotopic.
\end{lemma}

We will prove Lemma \ref{lemma:kill1} in \S \ref{section:kill1} and Lemma \ref{lemma:kill2} in
\S \ref{section:kill2}.
\end{proof}

Here is an outline of the remainder of this section.  In \S \ref{section:identifygen},
we will prove Lemma \ref{lemma:identifygen} above.  Next, in \S \ref{section:theretract}
we will prove some preliminary results about the retraction given by Lemma \ref{lemma:bases1retract}.
Finally, in \S \ref{section:kill1} and \S \ref{section:kill2} we will prove Lemmas \ref{lemma:kill1}
and \ref{lemma:kill2}.

\subsubsection{Identifying the generators}
\label{section:identifygen}

This section proves Lemma \ref{lemma:identifygen}, which identifies generators
for $\pi_{n-1}(\B_n(\Field_5))$.  The main idea of our proof will be to include $\B_n(\Field_5)$ into
$\BA_n(\Field_5)$, which by Proposition \ref{proposition:augbasescon} is $(n-1)$-connected.  We will
construct our generators inductively, and this section
is the one where it will be important for us to use the complexes $\B_{n,m}(\Field_5)$ built from links.

We start by proving two results that work over any field.  Our initial results will
be phrased in terms of homology groups rather than homotopy groups since that
is how our proofs function (and it allows us to avoid worrying about basepoints).
We will later use the Hurewicz theorem to translate this into information about homotopy
groups.  Throughout this section, our convention is that $S^{-1}$ is the empty set.

\begin{lemma}[Inductive generators]
\label{lemma:identifygenrel}
Let $\Field$ be a field.  Let $n \geq 1$ and $m \geq 0$ be such that $n+m \geq 2$.
Then the group $\RH_{n-1}(\B_{n,m}(\Field))$ is generated by the images of the fundamental classes
under the following two families of maps.
\begin{compactitem}
\item The {\bf initial triangle maps}, which require $n \geq 2$.  Let $\sigma = \{\pm \vec{v}_0, \pm \vec{v}_1, \pm \vec{v}_2\}$ be a $2$-dimensional
internally additive simplex of $\BA_{n,m}(\Field)$, so $\vec{v}_0 = \lambda \vec{v}_1+\nu \vec{v}_2$ for some
$\lambda,\nu \in \Field^{\times}$.  Let $f\colon S^{n-3} \rightarrow \link_{\BA_{n,m}(\Field)}(\sigma)$
be a simplicial map for some triangulation of $S^{n-3}$.  The associated initial triangle map is then
\[\Sphere{\pm \vec{v}_0, \pm \vec{v}_1, \pm \vec{v}_2} \ast f \colon \partial \Delta^2 \ast S^{n-3} \cong S^{n-1} \longrightarrow \B_{n,m}(\Field).\]
\item The {\bf initial external suspend maps}, which 
require $m \geq 1$.  Let $\sigma = \{\pm \vec{v}_0, \pm \vec{v}_1\}$ be
a $1$-dimensional externally additive simplex of $\BA_{n,m}(\Field)$.
Let $f\colon S^{n-2} \rightarrow \link_{\BA_{n,m}(\Field)}(\sigma)$ be a simplicial map for some
triangulation of $S^{n-2}$.  The associated initial external suspend map is then 
\[\Sphere{\pm \vec{v}_0, \pm \vec{v}_1} \ast f\colon \partial \Delta^1 \ast S^{n-2} \cong S^{n-1} \longrightarrow \B_{n,m}(\Field).\]
\end{compactitem}
\end{lemma}
\begin{proof}
Proposition \ref{proposition:augbasescon} says that $\BA_{n,m}(\Field)$ is $(n-1)$-connected, so the long
exact sequence in homology for the pair $(\BA_{n,m}(\Field),\B_{n,m}(\Field))$ contains the segment
\[\HH_n(\BA_{n,m}(\Field),\B_{n,m}(\Field)) \rightarrow \RH_{n-1}(\B_{n,m}(\Field)) \rightarrow 0.\]
The group $\RH_{n-1}(\B_{n,m}(\Field))$ is thus generated by the image under the boundary map of generators for
$\HH_n(\BA_{n,m}(\Field),\B_{n,m}(\Field))$.

For a $k$-simplex $\{\pm \vec{v}_0,\ldots,\pm \vec{v}_k\}$ of $\BA_{n,m}(\Field)$, write
$[\pm \vec{v}_0,\ldots,\pm \vec{v}_k]$ for the associated element of the relative simplicial chains
$\CC_k(\BA_{n,m}(\Field),\B_{n,m}(\Field))$.  We thus have $[\pm \vec{v}_0,\ldots,\pm \vec{v}_k] = 0$
if $\{\pm \vec{v}_0,\ldots,\pm \vec{v}_k\}$ is a standard simplex.  We now identify two important
subcomplexes of $\CC_{\bullet}(\BA_{n,m}(\Field),\B_{n,m}(\Field))$.

\begin{steps}
\label{step:internal}
Let $s = (\pm \vec{v}_0, \pm \vec{v}_1, \pm \vec{v}_2)$ be an ordered internally additive
simplex of $\BA_{n,m}(\Field)$.  We then define a subcomplex $D_{\bullet}(s)$ of the
chain complex $\CC_{\bullet}(\BA_{n,m}(\Field),\B_{n,m}(\Field))$ such that the image
of the composition
\[\HH_n(D_{\bullet}(s)) \rightarrow \HH_n(\BA_{n,m}(\Field),\B_{n,m}(\Field)) \rightarrow \RH_{n-1}(\B_{n,m}(\Field))\]
is contained in the subgroup generated by the images of the fundamental classes under the initial triangle maps.
\end{steps}
\begin{proof}[Proof of Step \ref{step:internal}]
For all $k$, let $D_k(s)$ be the subgroup of $\CC_k(\BA_{n,m}(\Field),\B_{n,m}(\Field))$ spanned by elements
of the form $[\pm \vec{v}_0,\ldots,\pm \vec{v}_k]$, where $\{\pm \vec{v}_0,\ldots,\pm \vec{v}_k\}$ is
an internally additive simplex of $\BA_{n,m}(\Field)$ starting with the elements of $s$.  For any $0 \leq i \leq 2$, deleting
$\pm \vec{v}_i$ from this gives a standard simplex, so
\begin{align*}
\partial [\pm \vec{v}_0,\ldots,\pm \vec{v}_k] &= \sum_{i=0}^k [\pm \vec{v}_0,\ldots,\widehat{\pm \vec{v}_i},\ldots,\pm \vec{v}_k] \\
&= \sum_{i=3}^k (-1)^i [\pm \vec{v}_0,\ldots,\widehat{\pm \vec{v}_i},\ldots,\pm \vec{v}_k].
\end{align*}
It follows that $D_{\bullet}(s)$ is a subcomplex of the chain complex
$\CC_{\bullet}(\BA_{n,m}(\Field),\B_{n,m}(\Field))$.  Moreover, our boundary formula also implies that
\begin{align*}
D_{\bullet}(s) &\cong \RC_{\bullet-3}(\link_{\BA_{n,m}(\Field)}(\{\pm \vec{v}_0, \pm \vec{v}_1, \pm \vec{v}_2\}))\\
&\cong \RC_{\bullet-3}(\link_{\B_{n,m}(\Field)}(\{\pm \vec{v}_1, \pm \vec{v}_2\})).
\end{align*}
The complex 
\[\link_{\B_{n,m}(\Field)}(\{\pm \vec{v}_1, \pm \vec{v}_2\}) \cong \B_{n-2,m+2}(\Field)\] 
is $(n-4)$-connected by Proposition \ref{proposition:basescon}, so
\[\RH_{n-3}(\link_{\B_{n,m}(\Field)}(\{\pm \vec{v}_1, \pm \vec{v}_2\}))\]
is generated by the images of fundamental classes under maps
\[f\colon S^{n-3} \rightarrow \link_{\B_{n,m}(\Field)}(\{\pm \vec{v}_1, \pm \vec{v}_2\})\]
that are simplicial for some triangulation of $S^{n-3}$.  The claim about the image
of $\HH_n(D_{\bullet}(s))$ in $\RH_{n-1}(\B_{n,m}(\Field))$ follows.
\end{proof}

\begin{steps}
\label{step:external}
Let $t = (\pm \vec{v}_0, \pm \vec{v}_1)$ be an ordered externally additive
simplex of $\BA_{n,m}(\Field)$.  We then define a subcomplex $E_{\bullet}(t)$ of the
chain complex $\CC_{\bullet}(\BA_{n,m}(\Field),\B_{n,m}(\Field))$ such that the image
of the composition
\[\HH_n(E_{\bullet}(t)) \rightarrow \HH_n(\BA_{n,m}(\Field),\B_{n,m}(\Field)) \rightarrow \RH_{n-1}(\B_{n,m}(\Field))\]
is contained in the subgroup generated by the images of the fundamental classes under the initial external suspend maps.
\end{steps}
\begin{proof}[Proof of Step \ref{step:external}]
For all $k$, let $E_k(t)$ be the subgroup of $\CC_k(\BA_{n,m}(\Field),\B_{n,m}(\Field))$ spanned by elements
of the form $[\pm \vec{v}_0,\ldots,\pm \vec{v}_k]$, where $\{\pm \vec{v}_0,\ldots,\pm \vec{v}_k\}$ is
an externally additive simplex of $\BA_{n,m}(\Field)$ starting with the elements of $t$.  Just like
in Step \ref{step:internal}, this is a subcomplex of $\CC_{\bullet}(\BA_{n,m}(\Field),\B_{n,m}(\Field))$.
Generators for the image of $\HH_n(E_{\bullet}(t))$ in $\RH_{n-1}(\B_{n,m}(\Field))$ can also be
calculated just like in Step \ref{step:internal}, so we omit the details.
\end{proof}

To conclude the proof, let $I$ be the set of all $2$-dimensional internally additive simplices
of $\BA_{n,m}(\Field)$ and let $J$ be the set of all $1$-dimensional externally additive simplices
of $\BA_{n,m}(\Field)$.  We thus have $I = \emptyset$ if $n=1$ and $J = \emptyset$ if $m=0$.  Endow each element of $I$ and $J$
with an arbitrary ordering.  Examining the above constructions, we then see that we have an isomorphism
\[\CC_{\bullet}(\BA_{n,m}(\Field),\B_{n,m}(\Field)) \cong \left(\bigoplus_{s \in I} D_{\bullet}\left(s\right)\right) \oplus \left(\bigoplus_{t \in J} E_{\bullet}\left(t\right)\right)\]
of chain complexes.  The above two steps show that the image in $\RH_{n-1}(\B_{n,m}(\Field))$ of the $n^{\text{th}}$
homology group of
each term on the right-hand side of this isomorphism is contained in the subgroup generated by the generators claimed in
the lemma.  The lemma follows.
\end{proof}

\begin{lemma}[Absolute generators]
\label{lemma:identifygenab}
Let $\Field$ be a field.  Let $n \geq 1$ and $m \geq 0$ be such that $n+m \geq 2$.
Then the group $\RH_{n-1}(\B_{n,m}(\Field))$ is generated by the images of the fundamental 
classes under maps of the form
\[f_1 \ast \cdots \ast f_k\colon \partial \Delta^{r_1} \ast \cdots \ast \partial \Delta^{r_k} \cong S^{n-1} \rightarrow \B_{n,m}(\Field),\]
where the $f_i$ are as follows.  There exists a decomposition 
$\Field^{n+m} = \Field^m \oplus A_1 \oplus \cdots \oplus A_k$, and
for $1 \leq i \leq k$ the map $f_i$ falls into one of the following two classes:
\begin{compactitem}
\item A {\bf triangle}.  There exists a $2$-dimensional internally additive simplex $\{\pm \vec{v}_0,\pm \vec{v}_1, \pm \vec{v}_2\}$
of $\BA_{n,m}(\Field)$ such that
\[f_i = \Sphere{\pm \vec{v}_0,\pm \vec{v}_1,\pm \vec{v}_2}\colon \partial \Delta^2 \rightarrow \B_{n,m}(\Field)\]
and such that $A_i = \Span{\vec{v}_0,\vec{v}_1,\vec{v}_2}$.  Note that $A_i$ is $2$-dimensional.
\item A {\bf suspend}.  There exist nonzero vectors $\vec{v} \in A_i$ and 
$\vec{w} \in \Field^m \oplus A_1 \oplus \cdots \oplus A_{i-1}$
and some $\lambda \in \Field^{\times}$ such that
\[f_i = \Sphere{\pm \vec{v}, \pm (\lambda \vec{v}+\vec{w})}\colon \partial \Delta^1 \rightarrow \B_{n,m}(\Field)\]
and such that $A_i = \Span{\vec{v}}$.  Note that $A_i$ is $1$-dimensional.
\end{compactitem}
\end{lemma}
\begin{proof}
To simplify our exposition, we will abuse notation and identify maps of spheres into $\B_{n,m}(\Field)$
with the associated elements of reduced homology.
Let $\Lambda_{n,m}$ be the subgroup of $\RH_{n-1}(\B_{n,m}(\Field))$ generated
by the indicated generators.  We must prove that $\Lambda_{n,m} = \RH_{n-1}(\B_{n,m}(\Field))$.
We will prove this by induction on $n$.

The base case $n=1$ follows immediately
from Lemma \ref{lemma:identifygenrel}.  Indeed, in this base case,
for dimension reasons there are no initial triangle maps, so Lemma \ref{lemma:identifygenrel}
says that $\RH_{n-1}(\B_{n,m}(\Field))$ is generated by initial suspend
maps, which in this degenerate case are simply 
$f_1\colon \partial \Delta^1 \rightarrow \B_{n,m}(\Field)$ with $f_1$ a suspend.

Assume now that $n \geq 2$ and that the lemma is true for all smaller $n$.
Applying Lemma \ref{lemma:identifygenrel}, it is enough to prove that $\Lambda_{n,m}$
contains all initial triangle maps and initial suspend maps.  The proofs
of these two facts are similar, so we will show how to prove that
initial triangle maps are in $\Lambda_{n,m}$ and leave the case of initial suspend
maps to the reader.

Consider an initial triangle map
\begin{equation}
\label{eqn:makeitgen}
\Sphere{\pm \vec{v}_0, \pm \vec{v}_1, \pm \vec{v}_2} \ast f \colon \partial \Delta^2 \ast S^{n-3} \cong S^{n-1} \longrightarrow \B_{n,m}(\Field).
\end{equation}
By definition, $\sigma = \{\pm \vec{v}_0, \pm \vec{v}_1, \pm \vec{v}_2\}$ is
a $2$-dimensional internally additive simplex of $\BA_{n,m}(\Field)$ and
$f\colon S^{n-3} \rightarrow \B_{n,m}(\Field)$ is a simplicial map for
some triangulation of $S^{n-3}$ whose image lies in
\[\link_{\BA_{n,m}(\Field)}(\sigma) = \link_{\B_{n,m}}(\{\pm \vec{v}_1, \pm \vec{v}_2\}) \cong \B_{n-2,m+2}(\Field).\]
Let 
\[\Psi\colon \B_{n-2,m+2}(\Field) \rightarrow \link_{\BA_{n,m}(\Field)}(\sigma)\]
be this isomorphism.  By induction, $\Lambda_{n-2,m+2} = \RH_{n-3}(\B_{n-2,m+2}(\Field))$.  
For each generator $f'$ of $\Lambda_{n-2,m+2}$, the map
\[\Sphere{\pm \vec{v}_0, \pm \vec{v}_1, \pm \vec{v}_2} \ast \Psi(f')\]
is a generator for $\Lambda_{n,m}$.
Since $\Psi^{-1}(f) \in \RH_{n-3}(\B_{n-2,m+2}(\Field)) = \Lambda_{n-2,m+2}$ 
can be expressed as a product of these generators, it follows that \eqref{eqn:makeitgen}
lies in $\Lambda_{n,m}$, as desired.
\end{proof}

We now give a useful variant of Lemma \ref{lemma:identifygenab} for $\Field = \Field_5$.

\begin{lemma}[{Absolute generators, $\Field_5$}]
\label{lemma:identifygenab5}
Let $n \geq 1$ and $m \geq 0$ be such that $n+m \geq 2$.
Then the group $\RH_{n-1}(\B_{n,m}(\Field_5))$ is generated by the images of the fundamental
classes under maps of the form
\[f_1 \ast \cdots \ast f_k\colon \partial \Delta^{r_1} \ast \cdots \ast \partial \Delta^{r_k} \cong S^{n-1} \rightarrow \B_{n,m}(\Field_5),\]
where the $f_i$ are as follows.  There exists a decomposition
$\Field_5^{n+m} = \Field_5^m \oplus A_1 \oplus \cdots \oplus A_k$, and
for $1 \leq i \leq k$ the map $f_i$ falls into one of the following three classes:
\begin{compactitem}
\item A {\bf D-triangle}.  There is a $2$-dimensional internally additive simplex $\{\pm \vec{v}_0,\pm \vec{v}_1, \pm \vec{v}_2\}$ of $\BA_{n,m}(\Field_5)$ with $\vec{v}_0 = \lambda \vec{v}_1+\nu \vec{v}_2$ for some $\lambda,\nu \in \{\pm 1\}$ such that
\[f_i = \Sphere{\pm \vec{v}_0,\pm \vec{v}_1,\pm \vec{v}_2}\colon \partial \Delta^2 \rightarrow \B_{n,m}(\Field_5)\]
and such that $A_i = \Span{\vec{v}_0,\vec{v}_1,\vec{v}_2}$.  Note that $A_i$ is $2$-dimensional.
\item A {\bf D-suspend}.  There are nonzero vectors $\vec{v} \in A_i$ and $\vec{w} \in \Field_5^m \oplus A_1 \oplus \cdots \oplus A_{i-1}$
and such that
\[f_i = \Sphere{\pm \vec{v}, \pm (\vec{v}+\vec{w})}\colon \partial \Delta^1 \rightarrow \B_{n,m}(\Field_5)\]
and such that $A_i = \Span{\vec{v}}$.  Note that $A_i$ is $1$-dimensional.
\item A {\bf double-suspend}.  There is a nonzero vector $\vec{v} \in A_i$ such that
\[f_i = \Sphere{\pm \vec{v}, \pm 2\vec{v}}\colon \partial \Delta^1 \rightarrow \B_{n,m}(\Field_5)\]
and such that $A_i = \Span{\vec{v}}$.  Note that $A_i$ is $1$-dimensional.
\end{compactitem}
Moreover, if $m=0$ then at least one of the $f_i$ is either a D-triangle or a D-suspend.
\end{lemma}
\begin{proof}
To simplify our exposition, we will abuse notation and identify maps of spheres into $\B_{n,m}(\Field_5)$
with the associated elements of reduced homology.
Lemma \ref{lemma:identifygenab} says that $\RH_{n-1}(\B_{n,m}(\Field_5))$ is generated
by maps $f_1 \ast \cdots \ast f_k$, where each $f_i$ is either a triangle or a
suspend.  To express this in terms of our new generators, it is enough to show
how to write triangles and suspends as sums of D-triangles, D-suspends, and
double-suspends.

We start with triangles.  Consider a triangle 
$\Sphere{\pm \vec{v}_0,\pm \vec{v}_1, \pm \vec{v}_2}$.  By definition,
$\{\pm \vec{v}_0, \pm \vec{v}_1, \pm \vec{v}_2\}$ is an internally
additive simplex of $\BA_{n,m}(\Field_5)$.  We thus have 
$\vec{v}_0 = \lambda \vec{v}_1+\nu\vec{v}_2$ with $\lambda,\nu \in \Field_5^{\times}$.
We remark that no reordering of the $\vec{v}_i$ is necessary for this.  Multiplying
$\vec{v}_1$ and/or $\vec{v}_2$ by $-1$ if necessary, we can assume that
$\lambda,\nu \in \{1,2\}$.  There are now three cases.

If $\lambda=\nu=1$, then our triangle is already a D-triangle.

\Figure{figure:breakup}{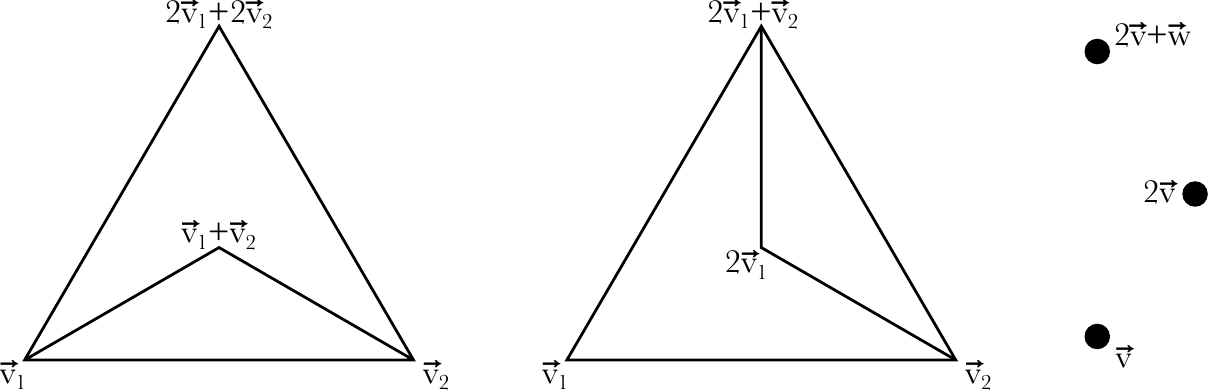}{Decomposing triangles/suspends into sums of D-triangles, D-suspends, and double suspends.  To avoid clutter, we omit the $\pm$'s.}{100}

If $\lambda=\nu=2$, then as in Figure \ref{figure:breakup} we can write
\begin{align*}
\Sphere{\pm (2 \vec{v}_1 + 2 \vec{v}_2), \pm \vec{v}_1, \pm \vec{v}_2} = &\Sphere{\pm(\vec{v}_1+\vec{v}_2), \pm \vec{v}_1, \pm \vec{v}_2}  \\
&+ \Sphere{\pm (\vec{v}_1+\vec{v}_2),\pm 2(\vec{v}_1+\vec{v}_2)} \ast \Sphere{\pm \vec{v}_2, \pm \vec{v}_1} \\
= &\Sphere{\pm(\vec{v}_1+\vec{v}_2), \pm \vec{v}_1, \pm \vec{v}_2}  \\
&+ \Sphere{\pm (\vec{v}_1+\vec{v}_2),\pm 2(\vec{v}_1+\vec{v}_2)} \ast \Sphere{\pm \vec{v}_2, \pm (\vec{v}_2 - (\vec{v}_1+\vec{v}_2))}.
\end{align*}
The right-hand side of our equation consists of a D-triangle and the join of
a double-suspend and a D-suspend.

Assume now that one of $\lambda$ and $\nu$ is $1$ and the other is $2$.  Swapping
them if necessary, we can assume that $\lambda=2$ and $\nu=1$.  As in
Figure \ref{figure:breakup}, we can write
\begin{align*}
\Sphere{\pm (2\vec{v}_1+\vec{v}_2), \pm \vec{v}_1, \pm \vec{v}_2} = &\Sphere{\pm (2 \vec{v}_1+\vec{v}_2),\pm 2\vec{v}_1, \pm \vec{v}_2,} \\
&- \Sphere{\pm \vec{v}_1,\pm 2\vec{v}_1} \ast \Sphere{\pm \vec{v}_2, \pm (\vec{v}_2+2\vec{v}_1)}.
\end{align*}
The right-hand side of our equation consists of a D-triangle and the join
of a double-suspend and a D-suspend.

Having dealt with triangles, we now must deal with suspends.  Consider
a suspend $\Sphere{\pm \vec{v},\pm(\lambda \vec{v} + \vec{w})}$.  We thus
have $\lambda \in \Field_5^{\times}$.  Multiplying $\vec{v}$ by $-1$ if
necessary, we can assume that $\lambda \in \{1,2\}$.  If $\lambda=1$, then
our suspend is already a D-suspend.  If $\lambda=2$, then as in Figure \ref{figure:breakup},
we can write
\[\Sphere{\pm \vec{v},\pm(2\vec{v} + \vec{w})} = \Sphere{\pm \vec{v},\pm 2\vec{v}} + \Sphere{\pm 2\vec{v},\pm(2 \vec{v} + \vec{w})}.\]
This is the sum of a double-suspend and a D-suspend.

All that remains to prove is the final claim of the lemma: if $m=0$, then
in our generators we can require at least one of the $f_i$ to either
be a D-triangle or a D-suspend.  For this, observe that the condition $m=0$
ensures that in the generators $f_1 \ast \cdots f_k$ given by
Lemma \ref{lemma:identifygenab}, the term $f_1$ must be a triangle (there
is no way to choose a nonzero $\vec{w}$ as in the definition of a suspend for it).
When we expand out the triangle $f_1$ as above, every term that appears has
either a D-triangle or a D-suspend in it.  The lemma follows.
\end{proof}

We finally prove Lemma \ref{lemma:identifygen}.

\begin{proof}[Proof of Lemma \ref{lemma:identifygen}]
Fix some $n \geq 3$.  Recall that our goal is to prove that $\pi_{n-1}(\B_n(\Field_5))$ is generated
by the initial D-triangle maps and the initial D-suspend maps.
Proposition \ref{proposition:basescon} says
that $\B_n(\Field_5)$ is $(n-2)$-connected, so the Hurewicz theorem gives an isomorphism
$\pi_{n-1}(\B_n(\Field_5)) \cong \HH_{n-1}(\B_n(\Field_5))$.
It is thus enough to prove that 
$\HH_{n-1}(\B_n(\Field_5))$ is generated by the images of the fundamental classes under these generators.  To simplify our expressions, we will abuse notation and identify our generators with the
images of the fundamental classes in $\HH_{n-1}(\B_n(\Field_5))$ under them.

Consider one of the generators
\[f_1 \ast \cdots \ast f_k\colon \partial \Delta^{r_1} \ast \cdots \ast \partial \Delta^{r_k} \cong S^{n-1} \rightarrow \B_{n}(\Field_5)\]
for $\HH_{n-1}(\B_n(\Field_5))$ identified by Lemma \ref{lemma:identifygenab5}.  Let
$\Field_5^n = A_1 \oplus \cdots \oplus A_k$ be the associated direct sum decomposition.  We will prove
that up to signs, in $\HH_{n-1}(\B_n(\Field_5))$ the element $f_1 \ast \cdots \ast f_k$ equals
either an initial D-triangle map or an initial D-suspend map.

Assume first that there exists some $1 \leq i_0 \leq k$ such that $f_{i_0}$ is a D-triangle.  We then have
$r_{i_0} = 2$.  As in the definition of a D-triangle, write
\[f_{i_0} = \Sphere{\pm \vec{v}_0, \pm \vec{v}_1, \pm \vec{v}_2}\colon \partial \Delta^2 \rightarrow \B_n(\Field_5)\]
for an additive simplex $\sigma = \{\pm \vec{v}_0, \pm \vec{v}_1, \pm \vec{v}_2\}$ of $\BA_n(\Field_5)$ with $\vec{v}_0 = \lambda \vec{v}_1+\nu\vec{v}_2$ for some $\lambda, \nu \in \{\pm 1\}$.
Set
\[f = f_1 \ast \cdots \ast \widehat{f_{i_0}} \ast \cdots \ast f_k\colon \partial \Delta^{r_1} \ast \cdots \ast \widehat{\partial \Delta^{r_{i_0}}} \ast \cdots \ast \partial \Delta^{r_k} \cong S^{n-3} \rightarrow \B_{n}(\Field_5).\]
We thus have that the image of $f$ lies in $\link_{\BA_n(\Field_5)}(\sigma)$.
Up to signs, in $\HH_{n-1}(\B_n(\Field_5))$ the element $f_1 \ast \cdots \ast f_k$ equals the
initial D-triangle map
\[\Sphere{\pm \vec{v}_0, \pm \vec{v}_1, \pm \vec{v}_2} \ast f\colon \partial \Delta^2 \ast S^{n-3} \rightarrow \B_n(\Field_5),\]
as desired.

We thus can assume that none of the $f_i$ are D-triangles.  Since at least one of the $f_i$ is either
a D-triangle or a D-suspend, there must exist some $1 \leq i_0 \leq k$ such that $f_{i_0}$ is a
D-suspend.  Pick $i_0$ such that it is as large as possible.  Set
\[W = A_1 \oplus \cdots \oplus \widehat{A_{i_0}} \oplus \cdots \oplus A_k,\]
and as in the definition of a D-suspend write
\[f_{i_0} = \Sphere{\pm \vec{v}, \pm (\vec{v}+\vec{w})}\colon \partial \Delta^1 \rightarrow \B_n(\Field_5).\]
We thus have $\vec{v} \in A_{i_0}$ and $\vec{w} \in W$.  Moreover, setting
\[f = f_1 \ast \cdots \ast \widehat{f_{i_0}} \ast \cdots \ast f_k\colon \partial \Delta^{r_1} \ast \cdots \ast \widehat{\partial \Delta^{r_{i_0}}} \ast \cdots \ast \partial \Delta^{r_k} \cong S^{n-2} \rightarrow \B_{n}(\Field_5)\]
we have that the image of $f$ lies in $\B(W)$ (this is where we use the fact that $i_0$ is as 
large as possible).  Up to signs, in $\HH_{n-1}(\B_n(\Field_5))$ the element $f_1 \ast \cdots \ast f_k$ equals the
initial D-suspend map
\[\Sphere{\pm \vec{v}, \pm (\vec{v}+\vec{w})} \ast f\colon \partial \Delta^2 \ast S^{n-3} \rightarrow \B_n(\Field_5),\]
as desired.
\end{proof}

\subsubsection{The retraction}
\label{section:theretract}

We now discuss the retraction $\rho\colon \B_n(\Field_5) \rightarrow \BD_n(\Field_5)$ provided
by Lemma \ref{lemma:bases1retract}.  In fact, for later use we will extend it to the following
larger complex.

\begin{definition}
Let $\BAO_n(\Field_5)$ be the subcomplex of $\BA_n(\Field_5)$ consisting of
$\BDA_n(\Field_5)$ along with all standard simplices of $\BA_n(\Field_5)$.
\end{definition}

We will construct a retraction $\rho\colon \BAO_n(\Field_5) \rightarrow \BDA_n(\Field_5)$
that extends the one given by Lemma \ref{lemma:bases1retract}.
The only simplices of $\BAO_n(\Field_5)$ that do not lie in $\BDA_n(\Field_5)$ are of the form
$\sigma = \{\pm \vec{v}_1,\ldots,\pm \vec{v}_n\}$ with $\det(\vec{v}_1\ \cdots\ \vec{v}_n) = \pm 2$.
Letting $S(\sigma)$ be the result of subdividing $\sigma$ with a new vertex $x_{\sigma}$, the
map $\rho$ is defined by setting $\rho(x_{\sigma}) = \pm \vec{w}$ and extending linearly, where
$\vec{w} \in \Field_5^n$ is chosen such that
\[\det(\vec{v}_1\ \cdots\ \widehat{\vec{v}_i}\ \cdots\ \vec{v}_n\ \vec{w}) = \pm 1 \quad \quad \text{for all $1 \leq i \leq n$}.\]
The only possible choices for $\vec{w}$ are of the form
\[\vec{w} = 2c_1 \vec{v}_1 + \cdots + 2c_n \vec{v}_n \quad \text{for some $c_1,\ldots,c_n \in \{\pm 1\}$}.\]
It is annoying that $\rho$ depends on the choice of these $c_i$; however, the following lemma implies that
all possible choices result in homotopic $\rho$:

\begin{lemma}
\label{lemma:retractindependent}
For some $n \geq 2$, let $\vec{v}_1,\ldots,\vec{v}_n \in \Field_5^n$ be such that
$\det(\vec{v}_1\ \cdots\ \vec{v}_n) = \pm 2$.  Let $\vec{w}_1, \vec{w}_2 \in \Field_5^n$ be
such that
\[\det(\vec{v}_1\ \cdots\ \widehat{\vec{v}_i}\ \cdots\ \vec{v}_n\ \vec{w}_j) = \pm 1 \quad \text{for all $1 \leq i \leq n$ and $1 \leq j \leq 2$}.\]
Then the maps
\[\Disc{\pm \vec{w}_1} \ast \Sphere{\pm \vec{v}_1,\ldots,\pm \vec{v}_n}\colon \Delta^0 \ast \partial \Delta^{n-2} \rightarrow \BDA_n(\Field_5)\]
and
\[\Disc{\pm \vec{w}_2} \ast \Sphere{\pm \vec{v}_1,\ldots,\pm \vec{v}_n}\colon \Delta^0 \ast \partial \Delta^{n-2} \rightarrow \BDA_n(\Field_5)\]
are homotopic relative to $\partial(\Delta^0 \ast \partial \Delta^{n-2}) = \Delta^{n-2}$.
\end{lemma}

Before we prove this lemma, we highlight how we will use it:

\begin{principle}
\label{principle:changesub}
Given a map $f\colon S^{n-1} \rightarrow \BAO_n(\Field_5)$ that is simplicial with respect to a triangulation
of $S^{n-1}$, if we want to prove that $\rho \circ f \colon S^{n-1} \rightarrow \BDA_n(\Field_5)$ is
nullhomotopic in $\BDA_n(\Field_5)$, then we can choose any way we want to subdivide the image of any
$(n-1)$-simplex $\sigma$ in $S^{n-1}$ such that $f(\sigma)$ is not a simplex of $\BDA_n(\Field_5)$.
\end{principle}

Indeed, by Lemma \ref{lemma:retractindependent} we can make an initial homotopy of $\rho \circ f$ to
change the original subdivision coming from $\rho$ to our arbitrary one.

We now turn to the proof of Lemma \ref{lemma:retractindependent}.  This proof will require the following lemma.

\begin{lemma}
\label{lemma:multiadd}
For some $n \geq 2$, let $\{\pm \vec{v}_1,\ldots,\pm \vec{v}_n\}$ be an $(n-1)$-simplex in $\BD_n(\Field_5)$.
Then the map
\[\Sphere{\pm \vec{v}_1,\ldots,\pm \vec{v}_n,\pm(\vec{v}_1+\cdots+\vec{v}_n)}\colon \partial \Delta^n \rightarrow \BD_n(\Field_5)\]
is nullhomotopic in $\BDA_n(\Field_5)$.
\end{lemma}
\begin{proof}
Using Lemma \ref{lemma:quotientbd}, we can find an $(n-1)$-simplex $\{\pm \vec{V}_1,\ldots,\pm \vec{V}_n\}$
in $\B_n(\Z)$ mapping to $\{\pm \vec{v}_1,\ldots,\pm \vec{v}_n\}$ under the projection
$\B_n(\Z) = \BD_n(\Z) \rightarrow \BD_n(\Field_5)$.  Changing the signs of the $\vec{V}_i$, we can assume
that $\vec{V}_i \in \Z^n$ projects to $\vec{v}_i \in \Field_5^n$ for all $i$.  We then have a map
\[\Sphere{\pm \vec{V}_1,\ldots,\pm \vec{V}_n,\pm(\vec{V}_1+\cdots+\vec{V}_n)}\colon \partial \Delta^n \rightarrow \B_n(\Z)\]
whose postcomposition with the projection $\B_n(\Z) \rightarrow \BD_n(\Field_5)$ is the map
we are trying to prove is nullhomotopic.
Theorem \ref{theorem:zaugtbasescon} says that $F$ is nullhomotopic in $\BTA_n(\Z) = \BA_n(\Z)$.
Composing this homotopy with the projection $\BA_n(\Z) \rightarrow \BDA_n(\Field_5)$ given by
Lemma \ref{lemma:quotientbda}, we get our desired homotopy.
\end{proof}

\begin{proof}[Proof of Lemma \ref{lemma:retractindependent}]
Write
\[\vec{w}_1 = 2c_1 \vec{v}_1 + \cdots + 2c_n \vec{v}_n \quad \text{and} \quad \vec{w}_2 = 2d_1 \vec{v}_1 + \cdots + 2d_n \vec{v}_n\]
with $c_i,d_i \in \{\pm 1\}$ for all $1 \leq i \leq n$.  It is enough to deal with the case where all but
one of the $c_i$ and $d_i$ are equal.  Reordering the $\vec{v}_i$, possibly multiplying them by $-1$,
and possibly flipping $\vec{w}_1$ and $\vec{w}_2$,
we can assume that $c_i = d_i = 1$ for $1 \leq i \leq n-1$ and that $c_n = 1$ and $d_n = -1$.  Since
$2c_n = 2$ and $2d_n = -2 = 3$, we thus have that $\vec{w}_2 = \vec{w}_1 + \vec{v}_n$.

\Figure{figure:retractindependent}{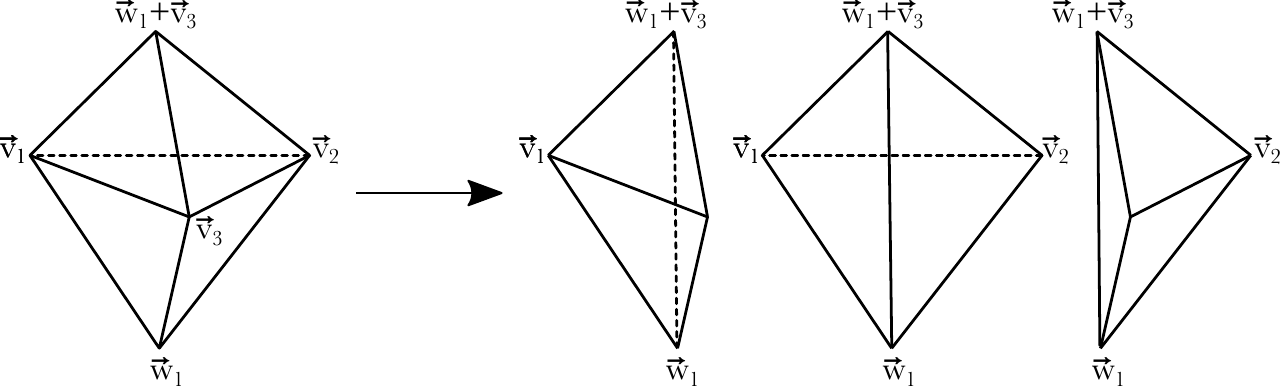}{The sphere in the proof of Lemma \ref{lemma:retractindependent} in the case $n=3$, along with the result of breaking it into $n=3$ spheres.  To avoid clutter, we omit the
$\pm$'s.}{100}

Our goal is equivalent to proving that
\[\Sphere{\pm \vec{w}_1,\pm(\vec{w}_1+\vec{v}_n)} \ast \Sphere{\pm \vec{v}_1,\ldots,\pm \vec{v}_n}\colon \partial \Delta^1 \ast \partial \Delta^{n-1} \cong S^{n-1} \rightarrow \BD_n(\Field_5)\]
is nullhomotopic in $\BDA_n(\Field_5)$; see Figure \ref{figure:retractindependent}.  As
is clear from that figure, as an element of $\pi_{n-1}(\BDA_n(\Field_5))$ our sphere is the sum
of the following $n$ spheres:
\[\Sphere{\pm \vec{v}_1,\ldots,\widehat{\pm \vec{v}_i},\ldots,\pm \vec{v}_n,\pm \vec{w}_1, \pm (\vec{w}_1+\vec{v}_n)} \quad \quad (1 \leq i \leq n).\]
For $1 \leq i \leq n-1$, these are the boundaries of additive simplices, and thus are trivially nullhomotopic
in $\BDA_n(\Field_5)$.  For $i=n$, since
\begin{align*}
\pm (\vec{w}_1+\vec{v}_n) &= \pm (2\vec{v}_1+\cdots+2\vec{v}_{n-1}+3\vec{v}_n) = \pm (3 \vec{v}_1+\cdots+3\vec{v}_{n-1}+2\vec{v}_n)\\
&= \pm (\vec{w}_1 + \vec{v}_1 + \cdots + \vec{v}_{n-1}),
\end{align*}
this is precisely the sphere that Lemma \ref{lemma:multiadd} says is nullhomotopic.  The lemma follows.
\end{proof}

\subsubsection{Killing initial D-triangle maps}
\label{section:kill1}

We now turn to proving Lemma \ref{lemma:kill1}, whose statement we will recall below.  This
will require the following lemma.

\Figure{figure:subdividedaugmented1}{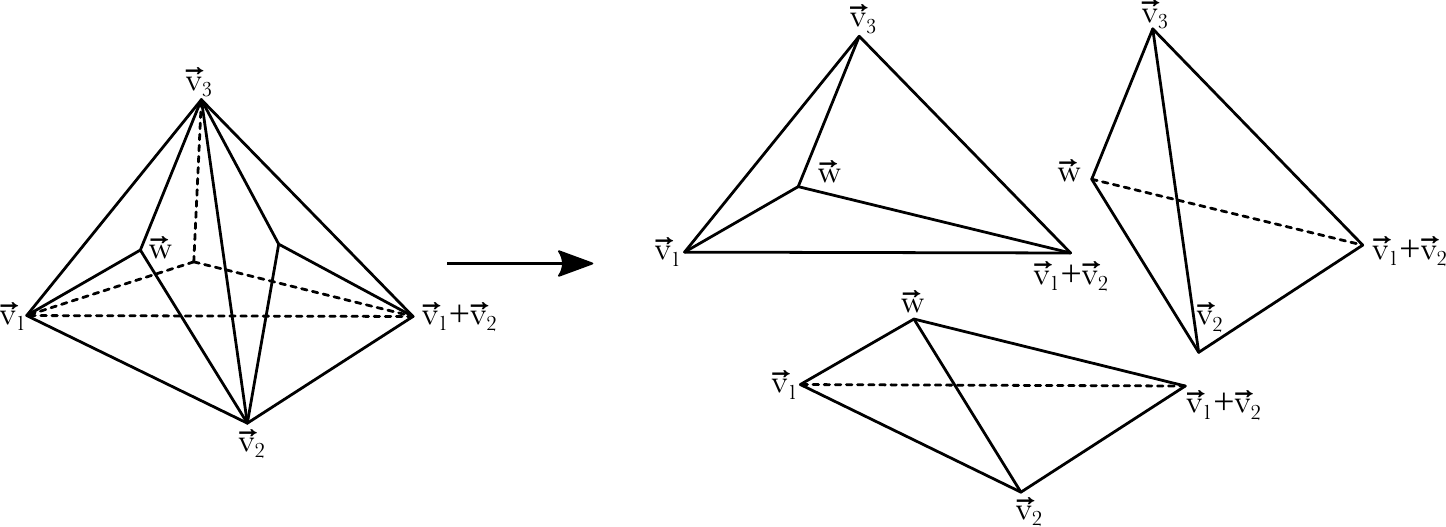}{On the left is the sphere $\rho \circ f$ in
the proof of Lemma \ref{lemma:killaugmented} in the case $n=3$ with its three subdivided faces.
On the right is the $n=3$ spheres it can be cut into (with the required subdivisions omitted to
improve readability).  To avoid clutter, we omit the $\pm$'s.}{100}

\begin{lemma}
\label{lemma:killaugmented}
For some $n \geq 2$, let $\rho\colon \BAO_n(\Field_5) \rightarrow \BDA_n(\Field_5)$ be
the retraction constructed in \S \ref{section:theretract}.  Let $\{\vec{v}_1,\ldots,\vec{v}_n\}$
be a basis of $\Field_5^n$.  Then
\[\rho \circ \Sphere{\pm (\vec{v}_1+\vec{v}_2), \pm \vec{v}_1,\ldots,\pm \vec{v}_n}\colon \partial \Delta^n \rightarrow \BDA_n(\Field_5)\]
is nullhomotopic.
\end{lemma}
\begin{proof}
Set $f = \Sphere{\pm (\vec{v}_1+\vec{v}_2), \pm \vec{v}_1,\ldots,\pm \vec{v}_n}$.
If $\det(\vec{v}_1\ \cdots\ \vec{v}_n) = \pm 1$, then $\rho \circ f = f$ and the image of
$f$ is the boundary of an additive simplex of $\BDA_n(\Field_5)$, so the lemma is trivial.  We
can thus assume that this determinant is $\pm 2$.

In the image of $\rho \circ f$, exactly $3$ faces of the image of $f$ are subdivided, namely
the images of
\begin{align*}
&\{\pm \vec{v}_1,\pm \vec{v}_2,\pm \vec{v}_3,\ldots,\pm \vec{v}_n\} \quad \text{and} \quad
\{\pm (\vec{v}_1+\vec{v}_2),\pm \vec{v}_2,\pm \vec{v}_3,\ldots,\pm \vec{v}_n\} \\
&\quad \quad \text{and} \quad
\{\pm \vec{v}_1,\pm(\vec{v}_1+\vec{v}_2),\pm \vec{v}_3,\ldots,\pm \vec{v}_n\}.
\end{align*}
See Figure \ref{figure:subdividedaugmented1}.  By Principle \ref{principle:changesub}, we can choose
the $\pm$-vector we use for each subdivision arbitrarily.  We will use $\pm \vec{w}$ with
\[\vec{w} = 2 \vec{v}_1 - 2 \vec{v}_2 + 2\vec{v}_3 + \cdots + 2\vec{v}_n\]
for $\{\pm \vec{v}_1,\pm \vec{v}_2,\pm \vec{v}_3,\ldots,\pm \vec{v}_n\}$ and leave the others
unspecified (for the moment).

As in Figure \ref{figure:subdividedaugmented1}, in $\pi_{n-1}(\BDA_n(\Field_5))$ the sphere
$\rho \circ f$ is the sum of the $n$ spheres $\rho \circ f_i$ with
\[f_i = \Sphere{\pm \vec{w}, \pm (\vec{v}_1+\vec{v}_2), \pm \vec{v}_1,\ldots,\widehat{\pm \vec{v}_i},\ldots,\pm \vec{v}_n} \quad \text{for $1 \leq i \leq n$}.\]
For $3 \leq i \leq n$, we have $\rho \circ f_i = f_i$ and the image of $f_i$ is the boundary
of an augmented simplex in $\BDA_n(\Field_5)$, so it is trivially nullhomotopic.
We thus must only deal with $i=1$ and $i=2$.  The proofs in these two cases are similar, so we
will do the case $i=2$ and leave the case $i=1$ to the reader.

\Figure{figure:subdividedaugmented2}{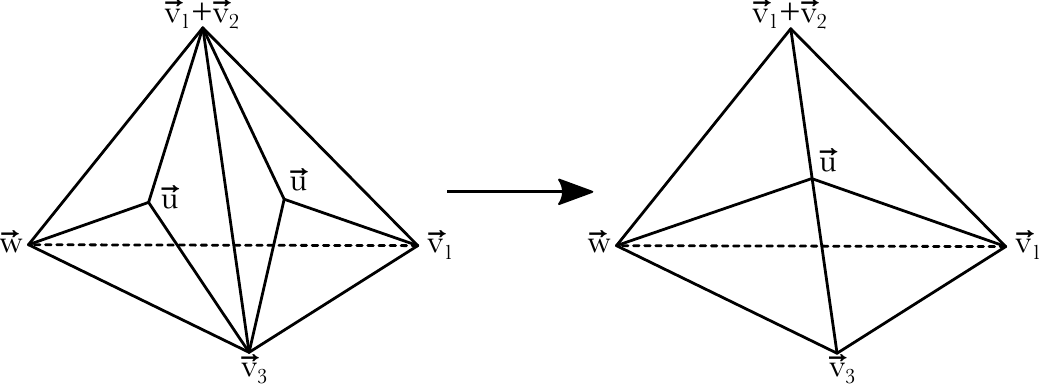}{On the left is the sphere appearing in
the case $i=2$ of the proof of Lemma \ref{lemma:killaugmented} for the case $n=3$.  On the right is the result
of homotoping it to the union of two tetrahedra.  To avoid clutter, we omit the $\pm$'s.}{100}

When forming $\rho \circ f_2$ for
\[f_2 = \Sphere{\pm \vec{w}, \pm (\vec{v}_1+\vec{v}_2),\pm \vec{v}_1,\pm \vec{v}_3,\ldots,\pm \vec{v}_n},\]
only two faces are subdivided, namely the images of
\begin{equation}
\label{eqn:twosubdivided}
\{\pm (\vec{v}_1+\vec{v}_2), \pm \vec{v}_1,\pm \vec{v}_3,\ldots,\pm \vec{v}_n\} \quad \text{and} \quad
\{\pm \vec{w}, \pm (\vec{v}_1+\vec{v}_2), \pm \vec{v}_3,\ldots,\pm \vec{v}_n\}.
\end{equation}
See Figure \ref{figure:subdividedaugmented2}.  The key observation is that by Principle \ref{principle:changesub}
we can use the same
vertex for both of these faces, namely $\pm \vec{u}$ with
\[\vec{u} = \vec{v}_1 - 2 \vec{v}_2 + 2 \vec{v}_3 + \cdots + 2\vec{v}_n.\]
This follows from the fact that
\[\vec{v}_1 - 2 \vec{v}_2 + 2 \vec{v}_3 + \cdots + 2\vec{v}_n = -2(\vec{v}_1+\vec{v}_2) - 2 \vec{v}_1 + 2 \vec{v}_3 + \cdots + 2 \vec{v}_n\]
and
\begin{align*}
\vec{v}_1 - 2 \vec{v}_2 + 2 \vec{v}_3 + \cdots + 2\vec{v}_n &= 2(\vec{v}_1+\vec{v}_2)+2 \vec{w} - 2 \vec{v}_3 - \cdots - 2 \vec{v}_n \\
&= 2(\vec{v}_1+\vec{v}_2) + 2(2 \vec{v}_1 - 2 \vec{v}_2 + 2\vec{v}_3 + \cdots + 2\vec{v}_n) - 2 \vec{v}_3 - \cdots - 2 \vec{v}_n.
\end{align*}
The two $(n-1)$-dimensional faces \eqref{eqn:twosubdivided} meet in
a common $(n-2)$-dimensional simplex
\[\eta = \{\pm (\vec{v}_1+\vec{v}_2),\pm \vec{v}_3,\ldots,\pm \vec{v}_n\}.\]
As in Figure \ref{figure:subdividedaugmented2}, we can homotope $\rho \circ f_2$
so as to replace the two subdivisions of the faces \eqref{eqn:twosubdivided} with
a single subdivision of the $(n-2)$-simplex $\eta$ by $\pm \vec{u}$.
The result is the sum in $\pi_{n-1}(\BDA_n(\Field_5))$ of $(n-1)$ different spheres
\begin{align*}
&\Sphere{\pm \vec{w}, \pm \vec{v}_1, \pm \vec{u}, \pm \vec{v}_3,\ldots,\pm \vec{v}_n}\\
&\quad \text{and} \quad
\Sphere{\pm \vec{w}, \pm \vec{v}_1, \pm \vec{u}, \pm(\vec{v}_1+\vec{v}_2),\pm \vec{v}_3,\ldots,\widehat{\pm \vec{v}_i},\ldots\pm \vec{v}_n} \quad \text{for $3 \leq i \leq n$.}
\end{align*}
These correspond to all the ways of replacing a vertex of $\eta$ with $\pm \vec{u}$ and
then adding the vertices $\pm \vec{w}$ and $\pm \vec{v}_1$ that do not appear in $\eta$.
Since $\vec{w} = \vec{u} + \vec{v}_1$, these are all
the boundaries of additive simplices in $\BDA_n(\Field_5)$, and hence are all
nullhomotopic.
\end{proof}

\begin{proof}[Proof of Lemma \ref{lemma:kill1}]
We first recall the statement.  For some $n \geq 3$, let $g\colon S^{n-1} \rightarrow \B_n(\Field_5)$ be an
initial D-triangle map, let $\rho\colon \B_n(\Field_5) \rightarrow \BD_n(\Field_5)$ be
the retraction given by Lemma \ref{lemma:bases1retract}, and let 
$\iota\colon \BD_n(\Field_5) \hookrightarrow \BDA_n(\Field_5)$
be the inclusion.  We must prove that $\iota \circ \rho \circ g\colon S^{n-1} \rightarrow \BDA_n(\Field_5)$ is nullhomotopic.

By definition, the initial D-triangle map $g$ is of the following form.
Let $\sigma = \{\pm \vec{v}_0, \pm \vec{v}_1, \pm \vec{v}_2\}$ be a $2$-dimensional
additive simplex of $\BA_n(\Field_5)$ such that $\vec{v}_0 = \lambda \vec{v}_1+\nu \vec{v}_2$
with $\lambda,\nu \in \{\pm 1\}$.  Multiplying $\vec{v}_1$ and/or $\vec{v}_2$ by
$-1$ if necessary, we can assume that $\lambda = \nu = 1$.
Let $f\colon S^{n-3} \rightarrow \link_{\BA_n(\Field_5)}(\sigma)$
be a simplicial map for some triangulation of $S^{n-3}$.  We then have 
\[g = \Sphere{\pm \vec{v}_0, \pm \vec{v}_1, \pm \vec{v}_2} \ast f = \Sphere{\pm(\vec{v}_1+\vec{v}_2),\pm \vec{v}_1,\pm \vec{v}_2} \ast f\colon \partial \Delta^2 \ast S^{n-3} \cong S^{n-1} \longrightarrow \B_n(\Field_5).\]
Our goal then is to show that the map
\[\iota \circ \rho \circ \left(\Sphere{\pm (\vec{v}_1 + \vec{v}_2), \pm \vec{v}_1, \pm \vec{v}_2} \ast f\right) \colon \partial \Delta^2 \ast S^{n-3} \rightarrow \BDA_n(\Field_5)\]
is nullhomotopic.

It is enough to show that it extends over $\Delta^2 \ast S^{n-3}$.  The only simplices
of $\Delta^2 \ast S^{n-3}$ whose image under this map are not simplices
of $\BDA_n(\Field_5)$ are of the form $\Delta^2 \ast \sigma$ where $\sigma$ maps
to a simplex $\{\pm \vec{v}_3,\ldots,\pm \vec{v}_n\}$ such that
$\det(\vec{v}_1\ \cdots\ \vec{v}_n) = \pm 2$.  By obstruction theory, it is enough
to show that $\partial(\Delta^2 \ast \sigma)$ is mapped to an $(n-1)$-sphere
that is nullhomotopic.  Since the restriction of our map
to $\partial(\Delta^2 \ast \sigma)$ is
\[\Sphere{\pm (\vec{v}_1+\vec{v}_2),\pm \vec{v}_1, \pm \vec{v}_2, \pm \vec{v}_3,\ldots,\pm \vec{v}_n},\]
this follows immediately from Lemma \ref{lemma:killaugmented}.
\end{proof}

\subsubsection{Killing initial D-suspend maps}
\label{section:kill2}

We now turn to proving Lemma \ref{lemma:kill2}, whose statement we will recall below.  This
will require two lemmas.

\begin{lemma}
\label{lemma:disguisedfix}
For some $n \geq 2$, let $\rho\colon \BAO_n(\Field_5) \rightarrow \BDA_n(\Field_5)$
be the retraction constructed in \S \ref{section:theretract}.
Let $\{\vec{v}_1,\ldots,\vec{v}_n\}$ be a basis for $\Field_5^n$ such that
$\det(\vec{v}_1\ \cdots\ \vec{v}_n)=\pm 2$.  Pick some $\vec{u} \in \Span{\vec v_1, \dots, \vec v_{n-1}} \subset \Field_5^n$.
Then the maps
\[\rho \circ \left(\Disc{\pm \vec{v}_n} \ast \Sphere{\pm \vec{v}_1,\ldots,\pm \vec{v}_{n-1}, \pm(2\vec{v}_1+\cdots+2\vec{v}_{n-1})}\right)\colon \Delta^0 \ast \partial \Delta^{n-1} \longrightarrow \BD_n(\Field_5)\]
and
\[\rho \circ \left(\Disc{\pm (\vec{v}_n+\vec{u})} \ast \Sphere{\pm \vec{v}_1,\ldots,\pm \vec{v}_{n-1}, \pm(2\vec{v}_1+\cdots+2\vec{v}_{n-1})}\right) \colon \Delta^0 \ast \partial \Delta^{n-1} \longrightarrow \BD_n(\Field_5)\]
are homotopic in $\BDA_n(\Field_5)$ through maps fixing
$\partial (\Delta^0 \ast \partial \Delta^{n-1}) = \partial \Delta^{n-1}$.
\end{lemma}
\begin{proof}
It is enough to deal with the case where $\vec{u} = \vec{v}_i$ for some $1 \leq i \leq n-1$; the general case can then
be deduced via a sequence of these homotopies.  Since everything is symmetric, we can in fact assume that
$\vec{u} = \vec{v}_1$.

\Figure{figure:disguisedfix}{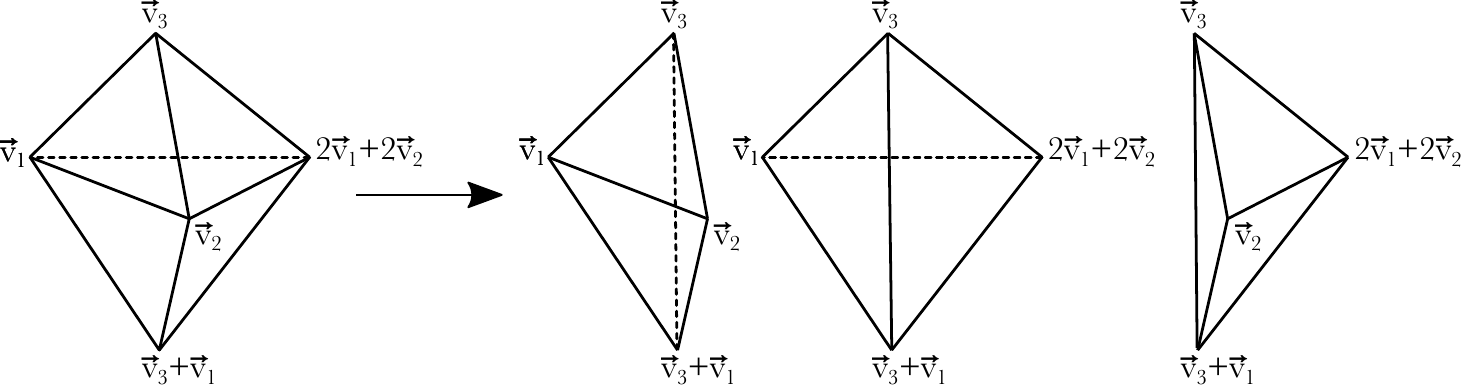}{The sphere in the proof of Lemma \ref{lemma:disguisedfix}
in the case $n=3$, along with the result of breaking it into $n=3$ spheres.  To avoid clutter, we omit the $\pm$'s.}{100}

Our goal is equivalent to showing that the map
\begin{align*}
&\rho \circ \left(\Sphere{\pm \vec{v}_n, \pm (\vec{v}_n+\vec{v}_1)} \ast \Sphere{\pm \vec{v}_1,\ldots,\pm \vec{v}_{n-1}, \pm(2\vec{v}_1+\cdots+2\vec{v}_{n-1})}\right)\colon \\
&\quad\quad \partial \Delta^1 \ast \partial \Delta^{n-1} \longrightarrow \BD_n(\Field_5)
\end{align*}
is nullhomotopic in $\BDA_n(\Field_5)$.  See Figure \ref{figure:disguisedfix}.  As is shown
in that figure, as an element of $\pi_{n-1}(\BDA_n(\Field_5))$ this is the sum of $n$ spheres.

The first is the sphere
\[\rho \circ \Sphere{\pm \vec{v}_n, \pm (\vec{v}_n+\vec{v}_1), \pm \vec{v}_1,\ldots,\pm \vec{v}_{n-1}}\colon \partial \Delta^n \longrightarrow \BDA_n(\Field_5),\]
which is nullhomotopic by Lemma \ref{lemma:killaugmented}.

The other $(n-1)$ are the spheres
\begin{align*}
&\rho \circ \Sphere{\pm \vec{v}_n, \pm (\vec{v}_n+\vec{v}_1), \pm \vec{v}_1,\ldots,\widehat{\pm \vec{v}_i},\ldots,\pm \vec{v}_{n-1}, \pm(2\vec{v}_1+\cdots+2\vec{v}_{n-1})}\colon \\
&\quad \quad \partial \Delta^1 \ast \partial \Delta^{n-1} \longrightarrow \BDA_n(\Field_5)
\end{align*}
for $1 \leq i \leq n-1$.  These are of two types:
\begin{compactitem}
\item For $2 \leq i \leq n-1$, these are nullhomotopic by Lemma \ref{lemma:killaugmented}.
\item For $i=1$, this is a bit more unusual.  The key observation here is that precisely one face of this is subdivided
by $\rho$, namely
\[\Disc{\pm \vec{v}_n, \pm (\vec{v}_n+\vec{v}_1), \pm \vec{v}_2,\ldots,\pm \vec{v}_{n-1}}.\]
By Principle \ref{principle:changesub}, we can choose the vertex we use in this subdivision arbitrarily.  If we use
$\pm \vec{w}$ with
\[\vec{w} = -2\vec{v}_n + 2(\vec{v}_n+\vec{v}_1) + 2\vec{v}_2 + \cdots + 2\vec{v}_{n-1} = 2\vec{v}_1+\cdots+2\vec{v}_{n-1},\]
then our sphere is the degenerate sphere
\[\Sphere{\pm (2\vec{v}_1+\cdots+2\vec{v}_{n-1}),\pm(2\vec{v}_1+\cdots+2\vec{v}_{n-1})} \ast \Sphere{\pm \vec{v}_n, \pm (\vec{v}_n+\vec{v}_1), \pm \vec{v}_2,\ldots,\pm \vec{v}_{n-1}},\]
which is trivially nullhomotopic.\qedhere
\end{compactitem}
\end{proof}

\Figure{figure:fixsubdivision}{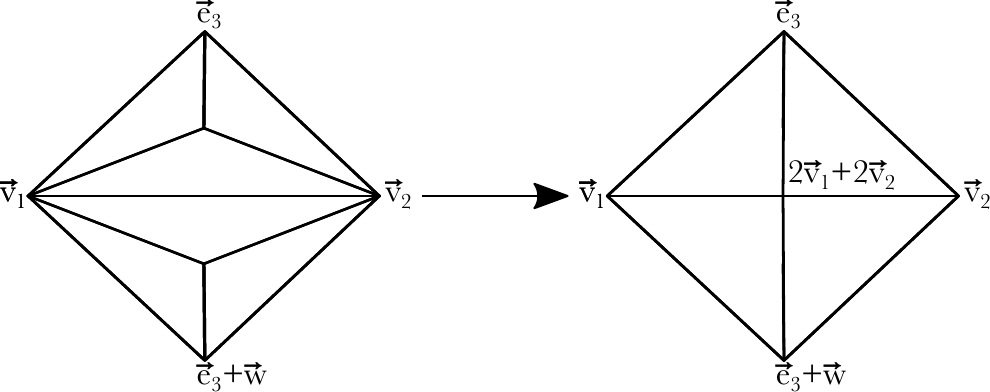}{The homotopy we are trying to achieve
in Lemma \ref{lemma:fixsubdivision} for $n=3$.  To avoid clutter, we omit the $\pm$'s.}{100}

\begin{lemma}
\label{lemma:fixsubdivision}
For some $n \geq 3$, let $\rho\colon \BAO_n(\Field_5) \rightarrow \BDA_n(\Field_5)$
and $\rho'\colon \BAO_{n-1}(\Field_5) \rightarrow \BDA_{n-1}(\Field_5)$ be
the retractions constructed in \S \ref{section:theretract}.
Let $\{\vec{e}_1,\ldots,\vec{e}_n\}$ be the standard basis of $\Field_5^n$,
let $\{\vec{v}_1,\ldots,\vec{v}_{n-1}\}$ be some basis of $\Field_5^{n-1} \subset \Field_5^n$,
and let $\vec{u} \in \Field_5^{n-1}$.  Then the maps
\[\rho \circ \left(\Sphere{\pm \vec{e}_n,\pm(\vec{e}_{n}+\vec{u})} \ast \Disc{\pm \vec{v}_1,\ldots,\pm \vec{v}_{n-1}}\right)\colon \partial \Delta^1 \ast \Delta^{n-1} \rightarrow \BDA_n(\Field_5)\]
and
\[\Sphere{\pm \vec{e}_n,\pm (\vec{e}_n+\vec{u})} \ast \left(\rho' \circ \Disc{\pm \vec{v}_1,\ldots,\pm \vec{v}_{n-1}}\right) \colon \partial \Delta^1 \ast \Delta^{n-1} \rightarrow \BDA_n(\Field_5)\]
are homotopic through maps fixing
$\partial (\partial \Delta^1 \ast \Delta^{n-1}) = \partial \Delta^1 \ast \partial \Delta^{n-1}$.
\end{lemma}
\begin{proof}
If $\det(\vec{v}_1\ \cdots\ \vec{v}_{n-1}\ \vec{e}_n) = \pm 1$, then these maps
are equal, so assume that this determinant is $\pm 2$.  In this case,
$\partial \Delta^1 \ast \Delta^{n-1}$ consists of two $n$-simplices the image of both
of which under the first map above are subdivided by $\rho$.  Moreover,
$\rho'$ subdivides the image of $\Disc{\pm \vec{v}_1,\ldots,\pm \vec{v}_{n-1}}$.  Using
Principle \ref{principle:changesub}, we can use $\pm \vec{w}$ with
\[\vec{w} = 2 \vec{v}_1 + \cdots + 2 \vec{v}_{n-1}\]
for this subdivision.  See Figure \ref{figure:fixsubdivision} for a picture of
the homotopy we are trying to achieve.  The key observation is that this is really a
disguised version of Lemma \ref{lemma:disguisedfix}; indeed, if you cut it open along
the central
\[\Sphere{\pm \vec{v}_1,\ldots,\pm \vec{v}_{n-1}, \pm (2 \vec{v}_1 + \cdots + 2 \vec{v}_{n-1})}\]
you get precisely the two discs that Lemma \ref{lemma:disguisedfix} claims are homotopic
via a homotopy fixing their boundary.  The lemma follows.
\end{proof}

\begin{proof}[Proof of Lemma \ref{lemma:kill2}]
We first recall the statement.
For some $n \geq 3$, let $g\colon S^{n-1} \rightarrow \B_n(\Field_5)$ be an initial D-suspend map,
let $\rho\colon \B_n(\Field_5) \rightarrow \BD_n(\Field_5)$ be
the retraction given by Lemma \ref{lemma:bases1retract}, and let $\iota\colon \BD_n(\Field_5) \hookrightarrow \BDA_n(\Field_5)$
be the inclusion.  Assume that $\pi_{n-2}(\BDA_{n-1}(\Field_5)) = 0$.
We must prove that $\iota \circ \rho \circ g\colon S^{n-1} \rightarrow \BDA_n(\Field_5)$ is nullhomotopic.

By definition, the initial D-suspend map $g$ is of the following form.
Let $\vec{v} \in \Field_5^n$ be a nonzero vector, let $W \subset \Field_5^n$
be an $(n-1)$-dimensional subspace such that $\Field_5^n = \Span{\vec{v}} \oplus W$, and let $\vec{w} \in W$ be nonzero.
Let $f\colon S^{n-2} \rightarrow \B(W)$ be a simplicial map for some triangulation of $S^{n-2}$.  We then have
\[g = \Sphere{\pm \vec{v}, \pm (\vec{v}+\vec{w})} \ast f\colon \partial \Delta^1 \ast S^{n-2} \cong S^{n-1} \longrightarrow \B_n(\Field_5).\]

Let $\{\vec{e}_1,\ldots,\vec{e}_n\}$ be the standard basis of $\Field_5^n$.
Changing coordinates with an element of $\SL_n(\Field_5)$, we can assume
that $\vec{v} = \vec{e}_n$ and that $W = \Field_5^{n-1}$.  Our map $f$
thus lands in $\B_{n-1}(\Field_5)$, and our goal is to prove that the
map
\[\iota \circ \rho \circ \left(\Sphere{\pm \vec{e}_n, \pm (\vec{e}_n+\vec{w})} \ast f\right) \colon \partial \Delta^1 \ast S^{n-2} \rightarrow \BDA_n(\Field_5)\]
is nullhomotopic.

Let $\rho'\colon \BAO_{n-1}(\Field_5) \rightarrow \BDA_{n-1}(\Field_5)$ be
the retraction constructed in \S \ref{section:theretract}.
Applying Lemma \ref{lemma:fixsubdivision} to $S^0 \ast \sigma$ for each $(n-2)$-simplex
$\sigma$ of $S^{n-2}$, we see that our map is homotopic to
\begin{equation}
\label{eqn:kill2tokill}
\Sphere{\pm \vec{e}_n, \pm (\vec{e}_n+\vec{w})} \ast (\rho' \circ f)\colon \partial \Delta^1 \ast S^{n-2} \rightarrow \BDA_n(\Field_5).
\end{equation}
Since $\BDA_{n-1}(\Field_5)$ is $(n-2)$-connected, the map $\rho' \circ f$ is nullhomotopic
in $\BDA_{n-1}(\Field_5)$.  Since the suspension of
$\BDA_{n-1}(\Field_5)$ with suspension points $\vec{e}_n$ and $\vec{e}_n+\vec{w}$ lies
in $\BDA_n(\Field_5)$, we conclude that \eqref{eqn:kill2tokill} is nullhomotopic,
as desired.
\end{proof}

\section{The Lee--Szczarba conjecture}
\label{section:proofs}

This section contains the proofs of our main results.  It has two sections.  In \S \ref{section:preliminaries},
we discuss some preliminary results, and in \S \ref{section:maintheorem}, we prove Theorem \ref{theorem:main}.  

\subsection{Preliminaries}
\label{section:preliminaries}

There are two sections of preliminaries.  In \S \ref{section:posets}, we
review the map-of-posets spectral sequence, and in \S \ref{section:titsquotient}, we give a concrete description
of the quotient of the Tits building $\cT_n(\Q)$ by the congruence subgroup $\Gamma_n(p)$.

\subsubsection{The map-of-posets spectral sequence}
\label{section:posets}

In this subsection, we review some results about the homology of posets with coefficients in 
a functor and about the map-of-posets spectral sequence.  Much of this is due to 
Quillen \cite{Q78} and Charney \cite{Charney}.  We begin with some definitions 
concerning posets. 

\begin{definition}
Let $\bX$ be a poset and $x \in \bX$.  We say $x$ has height $m$ and write $\height(x)=m$ if
$m$ is the largest integer such that there exists a chain
\[x_0 < \cdots < x_m = x \quad \quad \text{with $x_i \in \bX$ for all $0 \leq i \leq m$}.\]
We write $\bX_{>x}$ for the subposet of $\bX$ consisting of elements strictly larger than
$x$.  For a map $f\colon \bY \rightarrow \bX$ of posets, we write $f_{\leq x}$ for the
subposet of $\bY$ consisting of all $y \in \bY$ such that $f(y) \leq x$.
\end{definition}

A poset $\bX$ can be viewed as a category with a single morphism from 
$x \in \bX$ to $x' \in \bX$ precisely when $x \leq x'$.  Letting
$\Ab$ denote the category of abelian groups, we now recall the definition
of the homology of a poset with coefficients in a functor $F\colon \bX \rightarrow \Ab$.

\begin{definition}
Let $\bX$ be a poset and let $F\colon \bX \rightarrow \Ab$ be a functor.  Define
$\CC_{\bullet}(\bX;F)$ to be the following chain complex.  For $k \geq 0$, we set
\[\CC_k(\bX;F) = \bigoplus_{x_0<\dots<x_k} F(x_0),\] 
where the $x_i$ are understood to be elements of $\bX$.
The differential $\partial\colon \CC_k(\bX;F) \rightarrow \CC_{k-1}(\bX;F)$ is defined
to be $\sum_{i=0}^k (-1)^i \partial_i$, where 
$\partial_i\colon \CC_k(\bX;F) \rightarrow \CC_{k-1}(\bX;F)$ is as follows:
\begin{compactitem}
\item For $0 < i \leq k$, the map $\partial_i$ takes the $x_0 < \cdots < x_k$ summand
of $\CC_{k}(\bX;F)$ to the $x_0 < \cdots < \widehat{x_i} < \cdots < x_k$ summand
of $\CC_{k-1}(\bX;F)$ via the identity map $F(x_0) \rightarrow F(x_0)$.
\item The map $\partial_0$ takes the $x_0 < \cdots < x_k$ summand of
$\CC_k(\bX;F)$ to the $x_1 < \cdots < x_k$ summand of $\CC_{k-1}(\bX;F)$ via the
induced map $F(x_0) \rightarrow F(x_1)$.  
\end{compactitem}
We define $\HH_k(\bX;F) = \HH_k(\CC_{\bullet}(\bX;F))$.
\end{definition}

\begin{example}
Fix a poset $\bX$.  For a commutative ring $R$, we will write $\uR$ for the constant functor
on $X$ with value $R$.  We then have $\HH_k(\bX;\uR) \cong \HH_k(|\bX|;R)$, where
$|\bX|$ is the geometric realization of $\bX$.  We will often simply write this as $\HH_k(\bX;R)$. 
\end{example}

These homology groups can be very difficult to calculate.  One case where there is an
easy formula is where the functor $F$ is {\em supported on elements of height $m$}, i.e.\ where
$F(x) = 0$ for all $x \in \bX$ with $\height(x) \neq m$.  We then have the following lemma.
See e.g.\ \cite[Lemma 3.2]{MPWY} for a proof.

\begin{lemma}
\label{lemma:supportm}
Let $\bX$ be a poset and let $F\colon \bX \rightarrow \Ab$ be a functor that is supported
on elements of height $m$.  Then
\[\HH_k(\bX;F) \cong \bigoplus_{\height(x)=m} \RH_{k-1}(|\bX_{>x}|;F(x)),\]
where the coefficients $F(x)$ are simply regarded as an abelian group.
\end{lemma}

Our main interest in the homology of a poset with coefficients in a functor is
due to the following spectral sequence.  See
Quillen \cite[Section 7]{Q78} or Charney \cite[Section 1]{Charney} for a proof, and
see Remark \ref{remark:nonstandard} for why we use the nonstandard indices $(k,h)$.

\begin{theorem}[Map-of-posets spectral sequence]
\label{theorem:spectralsequence}
Let $f\colon \bY \m \bX$ be a map of posets.  Then there is a homologically graded spectral sequence
\[\EE^2_{kh}=\HH_k(\bX;[x \mapsto \HH_h(f_{\leq x})]) \implies \HH_{k+h}(\bY).\]
\end{theorem}

\begin{remark}
\label{remark:nonstandard}
We use the nonstandard indices $(k,h)$ since for us, $p$ is always a prime (so we cannot use $(p,q)$) and
$n$ is always a dimension (so we cannot use $(n,m)$).
\end{remark}

We will need a way to show that the map-of-posets spectral sequence vanishes in a large range.
The following lemma will be the key to this.

\begin{lemma}
\label{lemma:vanishingrange}
Let $f\colon \bY \m \bX$ be a map of posets and let $\EE^2_{kh}$ be the map-of-posets spectral sequence
for it given by Theorem \ref{theorem:spectralsequence}.  For some 
$d,e,r \geq 0$, assume that the following hold for all $x \in \bX$.
\begin{compactitem}
\item $\RH_h(|f_{\leq x}|)=0$ for all $h \notin [\height(x)+d-r, \height(x)+d]$.
\item $\RH_k(|\bX_{>x}|)=0$ for all $k \neq e-\height(x)-1$
\end{compactitem}
Then $\EE^2_{kh} = 0$ for all $k \geq 0$ and $h \geq 1$ satisfying $k+h \notin [d+e-r,d+e]$.
\end{lemma}

For the proof of Lemma \ref{lemma:vanishingrange}, we need the following lemma.

\begin{lemma} 
\label{lemma:supportrange}
Let $\bX$ be a poset and let $F\colon \bX \rightarrow \Ab$ be a functor.  
For some $b \geq a \geq 0$ and $e \geq 0$, assume that the following hold
for all $x \in \bX$.
\begin{compactitem}
\item $F(x) = 0$ whenever $\height(x) \notin [a,b]$.
\item $\RH_k(|\bX_{>x}|)=0$ for all $k \neq e-\height(x)-1$.
\end{compactitem}
Then $\HH_k(\bX;F) = 0$ for all $k \notin [e-b,e-a]$.
\end{lemma}
\begin{proof}
The proof will be by induction on $b-a$.  The base case $b-a=0$ follows
from Lemma \ref{lemma:supportm}, which says that setting $m = a = b$ we have
\[\HH_k(\bX;F) = \bigoplus_{\height(x)=m} \RH_{k-1}(|\bX_{>x}|;F(x)).\]
Since $F(x)$ here is just an abelian group, this vanishes by assumption
when $k-1 \neq e-m-1$.  Assume now that $b-a>0$.
Define $G\colon \bX \rightarrow \Ab$ via the formula
\[G(x) = \begin{cases}
F(x) & \text{if $a < \height(x) \leq b$},\\
0 & \text{otherwise}.
\end{cases}\]
We then have a short exact sequence of functors
\[0 \longrightarrow G \longrightarrow F \longrightarrow F/G \longrightarrow 0,\]
where $G(x) = 0$ for all $x \in \bX$ with $\height(x) \notin [a+1,b]$ and
$F/G(x) = 0 $ for all $x \in \bX$ with $\height(x) \neq a$.  The associated long exact sequence
in homology contains segments of the form
\[\HH_k(\bX;G) \longrightarrow \HH_k(\bX;F) \longrightarrow \HH_k(\bX;F/G).\]
Our inductive hypothesis says that $\HH_k(\bX;G) = 0$ for all
$k \notin [e-b,e-a-1]$ and that $\HH_k(\bX;F/G) = 0$ for all
$k \neq e-a$.  We conclude that $\HH_k(\bX;F) = 0$ for all $k$
such that $k \notin [e-b,e-a-1]$ and $k \neq e-a$, i.e.\ such that $k \notin [e-b,e-a]$.
\end{proof}

\begin{proof}[Proof of Lemma \ref{lemma:vanishingrange}]
Consider some $h \geq 1$.  Let $F_h\colon \bX \rightarrow \Ab$ be the functor defined via the formula
$F_h(x) = \HH_h(f_{\leq x})$.  By assumption, for all $x \in \bX$ we have that $F_h(x) = 0$
whenever $h \notin [\height(x)+d-r, \height(x)+d]$, i.e.\ whenever $\height(x) \notin [h-d,h-d+r]$.
Applying Lemma \ref{lemma:supportrange}, we see that $\EE^2_{kh} = \HH_k(\bX;F_h) = 0$ for all
$k \notin [e-(h-d+r),e-(h-d)]$, i.e.\ for all $k$ satisfying $k+h \notin [d+e-r,d+e]$, as desired.
\end{proof}

\subsubsection{The quotient of the Tits building by a congruence subgroup}
\label{section:titsquotient}

In order to prove/disprove the Lee--Szczarba conjecture, we need a concrete description of 
the quotient of the Tits building for $\Q$ by a congruence subgroup.  We begin by 
generalizing the definition of the Tits building to an arbitrary commutative ring.

\begin{definition}[Tits building]
Let $R$ be a commutative ring and let $V$ be a finite-rank free $R$-module.  Define
$\bT(V)$ to be the poset of proper nonzero direct summands of $R^n$, ordered by inclusion.
Also, let $\cT(V)$ denote the geometric realization of $\bT(V)$, viewed as a simplicial complex.
For $n \geq 1$, we will write $\bT_n(R) = \bT(R^n)$ and $\cT_n(R) = \cT(R^n)$.
\end{definition}

The following lemma helps clarify the action of $\SL_n(\Z)$ on $\bT_n(\Q)$.

\begin{lemma}
\label{lemma:identifybuilding}
For $n \geq 1$, we have $\bT_n(\Z) \cong \bT_n(\Q)$.
\end{lemma}
\begin{proof}
This follows from the fact that there is a bijection between subspaces of $\Q^n$ and direct summands
of $\Z^n$ taking a subspace $V \subset \Q^n$ to $V \cap \Z^n$ and a direct summand $W \subset \Z^n$
to $W \otimes \Q$.
\end{proof}

We now decorate our buildings by appropriate versions of orientations.

\begin{definition}[$\pm$-orientation]
Let $R$ be a commutative ring and let $V$ be a rank-$d$ free $R$-module, so $\wedge^d V \cong R^1$.
An {\em orientation} on $V$ is an element $\omega \in \wedge^d V$ that generates it as an $R$-module.  The
group $R^{\times}$ of units acts simply transitively on the set of orientations on $V$ by scalar multiplication.
A {\em $\pm$-orientation} on $V$ is a $\pm$-vector $\pm \omega$ such that $\omega$ is an orientation on $V$.
\end{definition}

\begin{example}
If $V$ is a rank-$d$ free $\Z$-module, then $\wedge^d V \cong \Z^1$.  Since the units of $\Z$ are $\{\pm 1\}$,
there is a unique $\pm$-orientation on $V$.
\end{example}

\begin{definition}[$\pm$-oriented Tits building]
Let $R$ be a commutative ring and let $V$ be a finite-rank free $R$-module.  Define
$\bTD(V)$ to be the poset of proper nonzero direct summands of $V$ equipped with 
a $\pm$-orientation.  The poset structure is simply inclusion; the $\pm$-orientations play no
role in it.  Let $\TD(V)$ denote the geometric realization of $\bTD(V)$, viewed
as a simplicial complex.  Finally, let $\bTD_n(R) = \bTD(R^n)$ and
$\TD_n(R) = \TD(R^n)$.  We call $\TD_n(R)$ the {\em $\pm$-oriented
Tits building}.
\end{definition}

\begin{remark}
\label{remark:orientationirrelevant}
We have $\TD_n(R) = \cT_n(R)$ if and only if $R^{\times} = \{\pm 1\}$.  In particular,
$\TD_n(\Z) = \cT_n(\Z)$ and $\TD_n(\Field_p) = \cT_n(\Field_p)$ if and only if $p \in \{2,3\}$.
\end{remark}

For a field $\Field$, the Solomon--Tits theorem \cite{Solomon, BrownBuildings} says that $\cT_n(\Field)$ is
Cohen--Macaulay of dimension $(n-2)$.  The following is the analogue of this for the $\pm$-oriented
Tits building.

\begin{lemma} 
\label{lemma:tdconn}
For any field $\F$ and any $n \geq 1$, the complex $\TD_n(\F)$ is Cohen--Macaulay
of dimension $(n-2)$.
\end{lemma}
\begin{proof}
As we said above, it follows from the Solomon--Tits theorem \cite{Solomon, BrownBuildings}
that $\cT_n(\Field)$ is Cohen--Macaulay of dimension $(n-2)$.  The complex
$\TD_n(\F)$ is a complete join complex over $\cT_n(\F)$ in the sense
of Hatcher--Wahl \cite[Definition 3.2]{HatcherWahl}, so the lemma
follows from \cite[Proposition 3.5]{HatcherWahl}.
\end{proof}

We now come to the main result of this section.

\begin{proposition} 
\label{proposition:tquotient} 
For all primes $p$ and all $n \geq 1$, we have $\cT_n(\Q)/\Gamma_n(p) \cong \TD_n(\Field_p)$.
\end{proposition}

For the proof of this proposition, we need two definitions and a lemma.  

\begin{definition}
Let $V$ be a rank-$n$ free $\Z$-module, let $\oV$ be an $n$-dimensional $\Field_p$-vector space, and
let $\pi\colon V \rightarrow \oV$ be a surjection (so $\ker(\pi) = p V$).  The image under $\pi$
of the unique $\pm$-orientation on $V$ is the $\pm$-orientation on $\oV$ that is {\em induced} by $\pi$.
\end{definition}

\begin{definition}
Let $V$ be a finite-dimensional vector space equipped with a $\pm$-orientation $\pm \omega$.  A basis
$\{\vec{x}_1,\ldots,\vec{x}_n\}$ for $V$ is {\em compatible} with $\pm \omega$ if 
$\pm \omega = \pm(\vec{x}_1 \wedge \cdots \wedge \vec{x}_n)$.
\end{definition}

\begin{lemma}
\label{lemma:liftbasis}
Let $V$ be a rank-$n$ free $\Z$-module, let $\oV$ be an $n$-dimensional $\Field_p$-vector space, and
let $\pi\colon V \rightarrow \oV$ be a surjection.  Let $\pm \omega$ be the $\pm$-orientation on $\oV$ induced
by $\pi$ and let $\{\vec{x}_1,\ldots,\vec{x}_n\}$ be a basis for $\oV$ that is compatible with $\pm \omega$.  For
some $0 \leq m < n$, let $\{\vec{X}_1,\ldots,\vec{X}_m\}$ be a partial basis for $V$ such that
$\pi(\vec{X}_i) = \vec{x}_i$ for $1 \leq i \leq m$.  We can then complete our partial basis to a basis
$\{\vec{X}_1,\ldots,\vec{X}_n\}$ for $V$ such that $\pi(\vec{X}_i) = \vec{x}_i$ for $1 \leq i \leq n$.
\end{lemma}
\begin{proof}
Let $W \subset V$ be the span of $\{\vec{X}_1,\ldots,\vec{X}_m\}$ and let $\GL(V,W)$ be the subgroup
of $\GL(V)$ consisting of automorphisms of $V$ acting as the identity on $W$.  Also, let
$\oW \subset \oV$ be the span of $\{\vec{x}_1,\ldots,\vec{x}_m\}$, let
$\SL^{\pm}(\oV)$ be the subgroup of $\GL(\oV)$ consisting of matrices with determinant $\pm 1$, and
let $\SL^{\pm}(\oV,\oW)$ be the subgroup of $\SL^{\pm}(\oV)$ consisting of automorphisms of $\oV$ with
determinant $\pm 1$ acting
as the identity on $\oW$.  We then have a surjection $\GL(V,W) \rightarrow \SL^{\pm}(\oV,\oW)$.  
The group $\GL(V,W)$ acts simply transitively
on the set of free bases for $V$ containing $\{\vec{X}_1,\ldots,\vec{X}_m\}$, 
and the group $\SL^{\pm}(\oV,\oW)$ acts simply transitively on the set
of bases for $\oV$ that contain $\{\vec{x}_1,\ldots,\vec{x}_m\}$ and are compatible with $\pm \omega$.  The lemma follows.
\end{proof}

\begin{proof}[Proof of Proposition \ref{proposition:tquotient}]
By Lemma \ref{lemma:identifybuilding}, the proposition is equivalent to the assertion
that $\cT_n(\Z)/\Gamma_n(p) \cong \TD_n(\Field_p)$.  Let $\pi\colon \Z^n \rightarrow \Field_p^n$ be
the mod-$p$ reduction map, and let $\psi\colon \cT_n(\Z) \rightarrow \TD_n(\Field_p)$
be the map taking a direct summand $V \subset \Z^n$ to $\pi(V) \subset \Field_p^n$ equipped
with the $\pm$-orientation induced by the restriction of $\pi$ to $V$.  The map $\psi$
is clearly $\Gamma_n(p)$-invariant, and thus induces a map $\cT_n(\Z)/\Gamma_n(p) \rightarrow \TD_n(\Field_p)$.
To prove this is an isomorphism, we must prove the following two facts.

\begin{claim}
Let $\osigma$ be a simplex of $\TD_n(\Field_p)$.  Then there exist a simplex $\sigma$ of
$\cT_n(\Z)$ such that $\psi(\sigma) = \osigma$.
\end{claim}
\begin{proof}[Proof of claim]
Let $\osigma$ be the flag
\begin{equation}
\label{eqn:surjectionflag}
0 \subsetneq \oV_0 \subsetneq \cdots \subsetneq \oV_k \subsetneq \Field_p^n,
\end{equation}
where $\oV_i$ is equipped with the $\pm$-orientation $\pm \omega_i$.  Set $n_i = \dim(\oV_i)$.  We can then find
a basis $\{\vec{x}_1,\ldots,\vec{x}_n\}$ for $\Field_p^n$ with the following two properties:
\begin{compactitem}
\item $\{\vec{x}_1,\ldots,\vec{x}_n\}$ is compatible with the $\pm$-orientation on $\Field_p^n$ induced by
the surjection $\pi\colon \Z^n \rightarrow \Field_p^n$.
\item For $0 \leq i \leq k$, the set $\{\vec{x}_1,\ldots,\vec{x}_{n_i}\}$ is a basis for
$\oV_i$ that is compatible with $\pm \omega_i$.
\end{compactitem}
Using Lemma \ref{lemma:liftbasis}, we can find a basis $\{\vec{X}_1,\ldots,\vec{X}_n\}$ such
that $\pi(\vec{X}_i) = \vec{x}_i$ for $1 \leq i \leq n$.  
For $0 \leq i \leq k$, let $V_i$ be the span of $\{\vec{X}_1,\ldots,\vec{X}_{n_i}\}$.  We thus have a flag
\[0 \subsetneq V_0 \subsetneq \cdots \subsetneq V_k \subsetneq \Z^n\]
of direct summands of $\Z^n$, and hence a simplex $\sigma$ of $\cT_n(\Z)$.  By construction, $\psi(\sigma) = \osigma$.
\end{proof}

\begin{claim}
Let $\sigma$ and $\sigma'$ be simplices of $\cT_n(\Z)$ such that $\psi(\sigma) = \psi(\sigma')$.  Then
there exists some $f \in \Gamma_n(p)$ such that $f(\sigma) = \sigma'$.
\end{claim}
\begin{proof}[Proof of claim]
Set $\osigma = \psi(\sigma) = \psi(\sigma')$.  Let $\osigma$ be the flag 
\begin{equation}
\label{eqn:injectionflag1}
0 \subsetneq \oV_0 \subsetneq \cdots \subsetneq \oV_k \subsetneq \Field_p^n,
\end{equation}
where $\oV_i$ is equipped with the $\pm$-orientation $\pm \omega_i$.  Set $n_i = \dim(\oV_i)$.  Finally, 
let $\sigma$ and $\sigma'$ be the flags
\[0 \subsetneq V_0 \subsetneq \cdots \subsetneq V_k \subsetneq \Z^n \quad \text{and} \quad 0 \subsetneq V'_0 \subsetneq \cdots \subsetneq V'_k \subsetneq \Z^n.\]
Just like in the previous claim, we can find a basis $\{\vec{x}_1,\ldots,\vec{x}_n\}$ for $\Field_p^n$ with
the following two properties:
\begin{compactitem}
\item $\{\vec{x}_1,\ldots,\vec{x}_n\}$ is compatible with the $\pm$-orientation on $\Field_p^n$ induced by
the surjection $\pi\colon \Z^n \rightarrow \Field_p^n$.
\item For $0 \leq i \leq k$, the set $\{\vec{x}_1,\ldots,\vec{x}_{n_i}\}$ is a basis for
$\oV_i$ that is compatible with $\pm \omega_i$.
\end{compactitem}
Applying Lemma \ref{lemma:liftbasis} recursively to each $V_i$ and then finally to $\Z^n$, we can find a free basis
$\{\vec{X}_1,\ldots,\vec{X}_n\}$ for $\Z^n$ such that $\pi(\vec{X}_i) = \vec{x}_i$ for all $1 \leq i \leq n$
and such that $\{\vec{X}_1,\ldots,\vec{X}_{n_i}\}$ is a basis for $V_i$ for all $1 \leq i \leq k$.
Similarly applying Lemma \ref{lemma:liftbasis} recursively to each $V'_i$ and then finally to $\Z^n$, we can find a free
basis $\{\vec{X}'_1,\ldots,\vec{X}'_n\}$ for $\Z^n$ such that $\pi(\vec{X}'_i) = \vec{x}_i$ for all $1 \leq i \leq n$
and such that $\{\vec{X}'_1,\ldots,\vec{X}'_{n_i}\}$ is a basis for $V'_i$ for all $1 \leq i \leq k$.
Let $f\colon \Z^n \rightarrow \Z^n$ be the automorphism taking $\vec{X}_i$ to $\vec{X}'_i$ for all $1 \leq i \leq n$.
By construction, we have $f(\sigma) = \sigma'$.  Moreover, we also have
$f \in \ker(\GL_n(\Z) \rightarrow \GL_n(\Field_p))$.  If $p \neq 2$, then this implies that $f \in \Gamma_n(p)$ and
we are done.  If $p=2$, then this might not hold since $f$ might have determinant $-1$ instead of $1$; however, in
this case we can replace $\vec{X}_1$ by $-\vec{X}_1$ and fix $f$ to have determinant $1$. 
\end{proof}

This completes the proof of Proposition \ref{proposition:tquotient}.
\end{proof}

\subsection{Resolution of the Lee--Szczarba conjecture}
\label{section:maintheorem}

The proof of Theorem \ref{theorem:main} is in \S \ref{section:mainproof}, which
is preceded by the preliminary \S \ref{section:preliminaryls}, which explains how to relate our complexes
of augmented partial bases to the Steinberg module. 

\subsubsection{Relating augmented partial bases to the Steinberg module}
\label{section:preliminaryls}

Recall from Lemma \ref{lemma:identifybuilding} 
that the Steinberg module $\St_n(\Q)$ is isomorphic to $\RH_{n-2}(\cT_n(\Z))$.  We now explain how to
relate this to our complexes of augmented partial bases.  We start with the following definition.

\begin{definition}
Let $R$ be a commutative ring.  Define $\BDA_n(R)'$ to be the subcomplex of
$\BDA_n(R)$ consisting of simplices $\{\pm \vec{v}_0,\ldots,\pm \vec{v}_k\}$
such that the $R$-span of the $\vec{v}_i$ is a proper submodule of $R^n$.
\end{definition}

In \cite{CP}, Church--Putman gave a new proof of a beautiful presentation for $\St_n(\Q)$ that was
originally proved by {\Byk} \cite{Byk}.  During their proof, they established the following 
result.  For a simplicial complex $X$, write $\Poset(X)$ for the poset of simplices of $X$.

\begin{lemma}[{\cite[\S 2.2]{CP}}]
\label{lemma:churchputman}
For $n \geq 2$, we have isomorphisms
\[\HH_{n-1}(\BDA_n(\Z),\BDA_n(\Z)') \xrightarrow[\cong]{\partial} \RH_{n-2}(\BDA_n(\Z)') \xrightarrow[\cong]{\Phi_{\ast}} \RH_{n-2}(\cT_n(\Z)) \cong \St_n(\Q),\]
where $\partial$ and $\Phi$ are as follows:
\begin{compactitem}
\item $\partial$ is the boundary map in the long exact sequence of a pair in reduced homology.
\item $\Phi\colon \Poset(\BDA_n(\Z)') \rightarrow \bT_n(\Z)$ is the poset map taking
a simplex $\{\pm \vec{v}_0,\ldots,\pm \vec{v}_k\}$ of $\BDA_n(\Z)'$ to the $\Z$-span of the $\vec{v}_i$.
\end{compactitem}
\end{lemma}

\begin{remark}
The map $\Phi$ in the above lemma makes sense since $\BDA_n(\Z)'$ is precisely the subcomplex
of $\BDA_n(\Z)$ where the indicated span is a {\em proper} summand of $\Z^n$.
\end{remark}

The Lee--Szczarba conjecture (Conjecture \ref{conjecture:leeszczarba}) concerns the map
\begin{equation}
\label{eqn:leeszczarba1}
(\St_n(\Q))_{\Gamma_n(p)} \longrightarrow \RH_{n-2}(\cT_n(\Q) / \Gamma_n(p)).
\end{equation}
Our next goal is to understand this map in terms of our complexes using 
Lemma \ref{lemma:churchputman}.  The first result is as follows.

\begin{lemma}
\label{lemma:coinvariants}
For all $n \geq 2$ and all primes $p$, we have
\[(\St_n(\Q))_{\Gamma_n(p)} \cong \HH_{n-1}(\BDA_n(\Field_p),\BDA_n(\Field_p)').\]
\end{lemma}
\begin{proof}
Since $\BDA_n(\Z)'$ is an $(n-1)$-dimensional complex containing the $(n-2)$-skeleton of $\BDA_n(\Z)$ 
and $\BDA_n(\Z)$ is $n$-dimensional, we have
\[\HH_{n-1}(\BDA_n(\Z),\BDA_n(\Z)') \cong \coker(\CC_n(\BDA_n(\Z)) \rightarrow \CC_{n-1}(\BDA_n(\Z),\BDA_n(\Z)')).\]
Similarly, we have
\[\HH_{n-1}(\BDA_n(\Field_p),\BDA_n(\Field_p)') \cong \coker(\CC_n(\BDA_n(\Field_p)) \rightarrow \CC_{n-1}(\BDA_n(\Field_p),\BDA_n(\Field_p)')).\]
The relative chains $\CC_{n-1}(\BDA_n(\Z),\BDA_n(\Z)'))$ are the free abelian group with basis
the standard simplices of $\BDA_n(\Z)$, and similarly over $\Field_p$.
Using this, the same argument as in the proof of Lemma \ref{lemma:quotientbd} shows
that 
\[(\CC_{n-1}(\BDA_n(\Z),\BDA_n(\Z)'))_{\Gamma_n(p)} \cong \CC_{n-1}(\BDA_n(\Field_p),\BDA_n(\Field_p)')).\]
Also, the same argument as in the proof of Lemma \ref{lemma:quotientbda} shows that
\[(\CC_{n}(\BDA_n(\Z)))_{\Gamma_n(p)} \cong \CC_{n}(\BDA_n(\Field_p)).\]
The lemma follows from the above four equations along with the fact that taking coinvariants is right-exact.
\end{proof}

Proposition \ref{proposition:tquotient} says that $\cT_n(\Q)/\Gamma_n(p) \cong \TD_n(\Field_p)$.  Combining
this with Lemma \ref{lemma:coinvariants}, we see that the map \eqref{eqn:leeszczarba1} can be
identified with a map
\begin{equation}
\label{eqn:leeszczarba2}
\HH_{n-1}(\BDA_n(\Field_p),\BDA_n(\Field_p)') \longrightarrow \RH_{n-2}(\TD_n(\Field_p)).
\end{equation}
This map is described in the following lemma.

\begin{lemma}
\label{lemma:themap}
For $n \geq 2$ and $p$ a prime, the map \eqref{eqn:leeszczarba2} equals the composition
\[\HH_{n-1}(\BDA_n(\Field_p),\BDA_n(\Field_p)') \xrightarrow{\partial} \RH_{n-2}(\BDA_n(\Field_p)') \xrightarrow{\Psi_{\ast}} \RH_{n-2}(\TD_n(\Field_p)),\]
where the maps are as follows:
\begin{compactitem}
\item $\partial$ is the boundary map in the long exact sequence of a pair in reduced homology.
\item $\Psi\colon \Poset(\BDA_n(\Field_p)') \rightarrow \bTD_n(\Field_p)$ is the poset map taking
a simplex $\sigma = \{\pm \vec{v}_0,\ldots,\pm \vec{v}_k\}$ of $\BDA_n(\Field_p)'$ to the $\Field_p$-span of the $\vec{v}_i$
equipped with the following $\pm$-orientation:
\begin{compactitem}
\item If $\sigma$ is a standard simplex, then the $\pm$-orientation is $\pm (\vec{v}_0 \wedge \cdots \wedge \vec{v}_k)$.
\item If $\sigma$ is an additive simplex and is ordered such that $\vec{v}_0 = \lambda \vec{v}_1 + \nu \vec{v}_2$ with
$\lambda, \nu \in \{\pm 1\}$, then the $\pm$-orientation is $\pm (\vec{v}_1 \wedge \cdots \wedge \vec{v}_k)$.
\end{compactitem}
\end{compactitem}
Moreover, $\partial$ is always surjective and is injective if $p \leq 5$.
\end{lemma}

\begin{remark}
It is an easy exercise to see that the $\pm$-orientations described in Lemma \ref{lemma:themap} are independent
of the various choices.
\end{remark}

\begin{proof}[Proof of Lemma \ref{lemma:themap}]
That \eqref{eqn:leeszczarba2} is the indicated map is immediate from the definitions,
so all we must prove are the claims about $\partial$.  The long exact sequence in reduced homology of
the pair $(\BDA_n(\Field_p),\BDA_n(\Field_p)')$ contains the segment
\begin{align*}
\RH_{n-1}(\BDA_n(\Field_p)) &\rightarrow \HH_{n-1}(\BDA_n(\Field_p),\BDA_n(\Field_p)') \\
&\quad\quad \xrightarrow{\partial} \RH_{n-2}(\BDA_n(\Field)') \rightarrow \RH_{n-2}(\BDA_n(\Field_p)).
\end{align*}
Proposition \ref{proposition:bases1acon} says that $\BDA_n(\Field_p)$ is $(n-2)$-connected, so 
$\RH_{n-2}(\BDA_n(\Field_p)) = 0$ and $\partial$ is surjective.
Also, Proposition \ref{proposition:bases1conimproved} says that if $p \leq 5$, then
$\BDA_n(\Field_p)$ is $(n-1)$-connected, so $\RH_{n-1}(\BDA_n(\Field_p))=0$ and
$\partial$ is injective.
\end{proof}

\subsubsection{The proof of Theorem \ref{theorem:main}}
\label{section:mainproof}

Theorem \ref{theorem:main} asserts that for a prime $p$ and $n \geq 2$, the induced map
\begin{equation}
\label{eqn:maintoprove}
(\St_n(\Q))_{\Gamma_n(p)} \longrightarrow \RH_{n-2}(\cT_n(\Q) / \Gamma_n(p))
\end{equation}
is always a surjection, but is an injection if and only if $p \leq 5$.  

We will prove something more precise than this.
Since the mechanisms
in the cases $n=2$ and $n \geq 3$ are slightly different, we will treat these two
cases separately.  The case $n=2$ is dealt with in the following theorem.

\begin{theorem}
\label{theorem:mainbetter2}
For all odd primes $p$, we have a short exact sequence
\[0 \longrightarrow \Z^{\frac{(p+2)(p-3)(p-5)}{12}} \longrightarrow (\St_2(\Q))_{\Gamma_2(p)} \longrightarrow \RH_{0}(\cT_2(\Q) / \Gamma_2(p)) \longrightarrow 0.\]
Also, $(\St_2(\Q))_{\Gamma_2(2)} \cong \RH_0(\cT_2(\Q)/\Gamma_2(2))$.
\end{theorem}

For the proof, we need the following observation.

\begin{lemma}
\label{lemma:orientedlines}
Let $V$ be a vector space over a field.  Then there is a bijection between the following two sets:
\begin{compactitem}
\item The set of $\pm \vec{v}$ with $\vec{v} \in V$ nonzero.
\item The set of $\pm$-oriented $1$-dimensional subspaces of $V$.
\end{compactitem}
\end{lemma}
\begin{proof}
The bijection takes $\pm \vec{v}$ with $\vec{v} \in V$ nonzero to the subspace spanned by $\vec{v}$ equipped
with the $\pm$-orientation $\pm \vec{v}$.
\end{proof}

\begin{proof}[Proof of Theorem \ref{theorem:mainbetter2}]
Lemma \ref{lemma:coinvariants} says that
\[(\St_2(\Q))_{\Gamma_2(p)} \cong \HH_{1}(\BDA_2(\Field_p),\BDA_2(\Field_p)'),\]
and Proposition \ref{proposition:tquotient} says that 
\[\cT_2(\Q)/\Gamma_2(p) \cong \TD_2(\Field_p).\]
Identifying the domain and codomain of \eqref{eqn:maintoprove} using these isomorphisms,
Lemma \ref{lemma:themap} says that for $n=2$ the map \eqref{eqn:maintoprove} can be identified with the composition
\[\HH_{1}(\BDA_2(\Field_p),\BDA_2(\Field_p)') \xrightarrow{\partial} \RH_{0}(\BDA_2(\Field_p)') \xrightarrow{\Psi_{\ast}} \RH_{0}(\TD_2(\Field_p)),\]
where $\partial$ is the boundary map in the long exact sequence of a pair and 
$\Psi\colon \Poset(\BDA_2(\Field_p)') \rightarrow \bTD_2(\Field_p)$ is a poset map
defined in that lemma.  

The simplicial complex $\BDA_2(\Field_p)'$ is the discrete
set $\Set{$\pm \vec{v}$}{$\vec{v} \in \Field_p^2$ nonzero}$, and $\bTD_2(\Field_p)$ is
the set of $\pm$-oriented $1$-dimensional subspaces of $\Field_p^2$.  By Lemma
\ref{lemma:orientedlines}, the map $\Psi$ is a bijection, so $\Psi_{\ast}$ is an isomorphism.  Moreover,
since $\BDA_2(\Field_p)'$ is discrete, we have $\HH_1(\BDA_2(\Field_p)') = 0$.  Finally,
Lemma \ref{lemma:themap} says that $\partial$ is a surjection onto $\RH_0(\BDA_2(\Field_p)')$.

Summarizing, we see that the long exact sequence for the pair
$(\BDA_2(\Field_p),\BDA_2(\Field_p)')$ contains the segment
\[\minCDarrowwidth20pt\begin{CD}
0 @>>> \HH_1(\BDA_2(\Field_p)) @>>> \HH_1(\BDA_2(\Field_p),\BDA_2(\Field_p)') @>>> \RH_0(\BDA_2(\Field_p)') @>>> 0  \\
@.     @VV{=}V                      @VV{\cong}V                                    @VV{\cong}V                   @. \\
0 @>>> \HH_1(\BDA_2(\Field_p)) @>>> (\St_2(\Q))_{\Gamma_2(p)}                 @>>> \RH_0(\TD_2(\Field_p))         @>>> 0.
\end{CD}\]
Lemma \ref{lemma:n2} implies that
\[\HH_1(\BDA_2(\Field_p)) = \begin{cases}
\Z^{\frac{(p+2)(p-3)(p-5)}{12}} & \text{if $p>2$},\\
0                          & \text{if $p=2$}.
\end{cases}\]
The theorem follows.
\end{proof}

The case $n \geq 3$ is as follows.

\begin{theorem}
\label{theorem:mainbetter3}
Fix a prime $p$ and some $n \geq 3$, and let $P_2^n$ denote the set of $\pm$-oriented $2$-dimensional
subspaces of $\Field_p^n$.  Then the map
\[(\St_n(\Q))_{\Gamma_n(p)} \longrightarrow \RH_{n-2}(\cT_n(\Q) / \Gamma_n(p))\]
is surjective.  It is injective for $p \leq 5$, while for $p>5$ its kernel surjects onto
$\Z[P_2^n] \otimes \RH_{n-4}(\TD_{n-2}(\F_p)) \otimes \HH_1(\BDA_2(\F_p))$.
\end{theorem}

\begin{remark}
To deduce the fact that the map is not injective for $p>5$, we need to know two things:
\begin{compactitem}
\item $\HH_1(\BDA_2(\F_p))$ is a nontrivial free $\Z$-module.  In fact, by Lemma \ref{lemma:n2} it
is isomorphic to $\Z^{\frac{(p+2)(p-3)(p-5)}{12}}$.
\item $\RH_{n-4}(\TD_{n-2}(\F_p))$ is a nontrivial free $\Z$-module (we remark that
in the degenerate case $n=3$, we have $\TD_{1}(\F_p) = \emptyset$ and thus
$\RH_{-1}(\TD_1(\F_p)) \cong \Z$).  In fact, Lemma \ref{lemma:tdconn}
says that $\TD_{n-2}(\Field_p)$ is Cohen--Macaulay of dimension $(n-4)$, so
$\RH_{n-4}(\TD_{n-2}(\F_p))$ is automatically a free $\Z$-module.  The fastest way to see that it
is nontrivial is to use the fact that forgetting the $\pm$-orientations gives a map
\[\RH_{n-4}(\TD_{n-2}(\Field_p)) \longrightarrow \RH_{n-4}(\cT_{n-2}(\Field_p)) \cong \St_{n-2}(\Field_p).\]
The Solomon--Tits theorem \cite{Solomon, BrownBuildings} says that $\St_{n-2}(\Field_p) \neq 0$, and
it is easy to see that its generators (given by ``apartments'') lift to nontrivial elements of
$\RH_{n-4}(\TD_{n-2}(\Field_p))$.  We remark that Theorem \ref{theorem:recursive} (proved in
\S \ref{section:computation} below) actually calculates $\RH_{n-4}(\TD_{n-2}(\Field_p))$. 
\end{compactitem}
\end{remark}

\begin{proof}[Proof of Theorem \ref{theorem:mainbetter3}]
Lemma \ref{lemma:themap} says that the map we are concerned with can be identified with the map
\[\HH_{n-1}(\BDA_n(\Field_p),\BDA_n(\Field_p)') \xrightarrow{\partial} \RH_{n-2}(\BDA_n(\Field_p)') \xrightarrow{\Psi_{\ast}} \RH_{n-2}(\TD_n(\Field_p)),\]
where $\partial$ is the boundary map in the long exact sequence of a pair and 
$\Psi\colon \Poset(\BDA_n(\Field_p)') \rightarrow \bTD_n(\Field_p)$ is a poset map
defined in that lemma.  Lemma \ref{lemma:themap} also says that $\partial$ is always surjective
and is injective for $p \leq 5$.  It is thus enough to show that the map
\[\Psi_{\ast}\colon \RH_{n-2}(\BDA_n(\Field_p)') \rightarrow \RH_{n-2}(\TD_n(\Field_p))\]
is always surjective, is injective for $p \leq 5$, and has a kernel surjecting onto
$\HH_1(\BDA_2(\F_p)) \otimes \Z[P_2^n] \otimes \RH_{n-4}(\TD_{n-2}(\F_p))$ for $p>5$.
Since $n \geq 3$, we have $n-2 \geq 1$ and thus we can work with unreduced homology.

We will do this by studying the map-of-posets spectral sequence
(Theorem \ref{theorem:spectralsequence}) of the 
poset map $\Psi\colon \Poset(\BDA_n(\Field_p)') \rightarrow \bTD_n(\Field_p)$.  
This takes the form
\[\EE^2_{kh}=\HH_k(\bTD_n(\Field_p);[V \mapsto \HH_h(\Psi_{\leq V})]) \implies \HH_{k+h}(\Poset(\BDA_n(\Field_p)')).\]
We wish to apply Lemma \ref{lemma:vanishingrange} to this to deduce a vanishing range.
This requires the following two facts.  Consider $V \in \bTD_n(\Field_p)$.
\begin{compactitem}
\item We have
\[\Psi_{\leq V} \cong \BDA_{\dim(V)}(\Field_p) = \BDA_{\height(V)+1}(\Field_p).\]
Proposition \ref{proposition:bases1acon} says that this is $(\height(V)-1)$-connected.  Since
it has dimension at most $(\height(V)+1)$, we conclude that
\begin{equation}
\label{eqn:vanish1}
\RH_h(|\Psi_{\leq V}|) = 0 \quad \quad \text{for all $h \notin [\height(V),\height(V)+1]$}.
\end{equation}
For later use, observe that the dimension is exactly $(\height(V)+1)$ except in the
degenerate case of $\height(V)=0$.  In this case, $\BDA_1(\Field_p)$ is a single point,
so $\RH_h(\BDA_1(\Field_p)) = 0$ for all $h$.  The upshot is that
\begin{equation}
\label{eqn:concentrate}
\text{the functor $V \mapsto \HH_1(\Psi_{\leq V})$ is supported on elements of height $1$.}
\end{equation}
\item We have
\[\left(\bTD_n\left(\Field_p\right)\right)_{>V} \cong \bTD_{n-\dim(V)}(\Field_p) = \bTD_{n-1-\height(V)}(\Field_p).\]
Lemma \ref{lemma:tdconn} says that this is Cohen--Macaulay of dimension $(n-3-\height(V))$, which implies
that 
\begin{equation}
\label{eqn:vanish2}
\RH_k(|(\bTD_n(\Field_p))_{>V}|) = 0 \quad \quad \text{for all $k \neq n-3-\height(V)$}.
\end{equation}
\end{compactitem}
Facts \eqref{eqn:vanish1} and \eqref{eqn:vanish2} imply that we can apply Lemma \ref{lemma:vanishingrange} with $d = 1$ and $r=1$ and
$e = n-2$.  This lemma implies that
\[\EE^2_{hk} = 0 \quad \text{for all $k \geq 0$ and $h \geq 1$ satisfying $k+h \notin [d+e-r,d+e] = [n-2,n-1]$}.\]
We now analyze the bottom row $E^2_{k0}$.  Proposition \ref{proposition:bases1acon} implies
that $\BDA_{\dim(V)}(\Field_p)$ is connected when $\dim(V) \geq 2$, and
$\BDA_1(\Field_p)$ is a single point and is also thus connected (this is one key place
where it is important that we are using $\pm$-vectors and 
requiring the determinant to be $\pm 1$).  We thus see that
\[\EE^2_{k0} = \HH_k(\bTD_n(\Field_p);[V \mapsto \HH_0(\Psi_{\leq V})]) = \HH_k(\bTD_n(\Field_p);\underline{\Z}) = \begin{cases}
\HH_{n-2}(\TD_n(\Field_p)) & \text{if $k=n-2$},\\
\Z & \text{if $k=0$},\\
0 & \text{if $k \neq 0,n-2$}.
\end{cases}\]
This last equality uses Lemma \ref{lemma:tdconn}.

Summarizing the above two calculations, the only potentially nonzero terms in our spectral sequence are of the form
\begin{center}
\begin{tabular}{|llllll}
\footnotesize{$\EE^2_{0,n-1}$} &               &          &               &               & \\
\footnotesize{$\EE^2_{0,n-2}$} & \footnotesize{$\EE^2_{1,n-2}$} &          &               &               & \\
              & \footnotesize{$\EE^2_{1,n-3}$} & $\ddots$         &               &               & \\
              &               & $\ddots$ & \footnotesize{$\EE^2_{n-4,3}$}             &               & \\
              &               &          & \footnotesize{$\EE^2_{n-4,2}$} & \footnotesize{$\EE^2_{n-3,2}$} & \\
              &               &          &               & \footnotesize{$\EE^2_{n-3,1}$} & \footnotesize{$\EE^2_{n-2,1}$} \\
\footnotesize{$\Z$}          &               &          &               &               & \footnotesize{$\HH_{n-2}(\TD_n(\Field_p))$}\\
\cline{1-6}
\end{tabular}
\end{center}

Observe that no nontrivial differentials come into or out of the
$\EE^2_{n-2,0} = \HH_{n-2}(\TD_n(\Field_p))$ term, so this term survives until
$E^{\infty}$.  This edge value in our spectral sequence is the image
of the map
\[\Psi_{\ast}\colon \HH_{n-2}(\BDA_n(\Field_p)') \rightarrow \HH_{n-2}(\TD_n(\Field_p)),\]
so we deduce that this map is surjective, which is one of the conclusions of the
theorem.

As for the other conclusions of the theorem, we separate things into two cases.

\begin{casea}
$p \leq 5$.
\end{casea}

In this case, we can replace our invocations of Proposition \ref{proposition:bases1acon}
with Proposition \ref{proposition:bases1conimproved}, which improves the
degree of connectivity of $\BDA_n(\Field_p)$ by $1$.  This causes
all the terms on the $k+h = n-2$ diagonal other than
$\EE^2_{n-2,0} = \HH_{n-2}(\TD_n(\Field_p))$ to vanish.  The conclusion
is that the map
\[\Psi_{\ast}\colon \HH_{n-2}(\BDA_n(\Field_p)') \rightarrow \HH_{n-2}(\TD_n(\Field_p))\]
is an isomorphism, as desired.

\begin{casea}
$p > 5$.
\end{casea}

In this case, observe that there are no nontrivial differentials going into or out
of the $\EE^2_{n-3,1}$-term, so this term survives until $E^{\infty}$.  By definition,
this implies that the kernel of the map
\[\Psi_{\ast}\colon \HH_{n-2}(\BDA_n(\Field_p)') \rightarrow \HH_{n-2}(\TD_n(\Field_p))\]
surjects onto $\EE^2_{n-3,1}$, so it is enough to prove that
\begin{equation}
\label{eqn:finaltoprove}
\EE^2_{n-3,1} \cong \Z[P_2^n] \otimes \RH_{n-4}(\TD_{n-2}(\F_p)) \otimes \HH_1(\BDA_2(\F_p)),
\end{equation}
where we recall that $P_2^n$ is the set of $\pm$-oriented $2$-dimensional
subspaces of $\Field_p^n$.

By definition,
\[\EE^2_{n-3,1} = \HH_{n-3}(\bTD_n(\Field);[V \mapsto \HH_1(\Psi_{\leq V})]).\]
As we observed in \eqref{eqn:concentrate}, 
the functor $V \mapsto \HH_1(\Psi_{\leq V})$ is supported on elements of height
$1$.  Applying Lemma \ref{lemma:supportm}, we see that $\EE^2_{n-3,1}$ is isomorphic
to
\begin{align*}
\EE^2_{n-3,1} 
&\cong \bigoplus_{\height(V)=1} \RH_{n-4}((\bTD_n(\Field_p))_{>V};\HH_1(\Psi_{\leq V})) \\
&\cong \bigoplus_{\dim(V)=2} \RH_{n-4}(\bTD_{n-2}(\Field_p);\HH_1(\BDA_2(\Field_p))) \\
&\cong \bigoplus_{\dim(V)=2} \RH_{n-4}(\bTD_{n-2}(\Field_p)) \otimes \HH_1(\BDA_2(\Field_p)).
\end{align*}
There are precisely $|P_2^n|$ terms in this direct sum, so
\eqref{eqn:finaltoprove} follows.
\end{proof}

\section{Computational results}
\label{section:computation}

We close the paper by proving Theorem \ref{theorem:recursive} in \S \ref{section:recursive} and
Theorem \ref{theorem:evenbetter} in \S \ref{section:evenbetter}.

\subsection{The recursive formula for the rank}
\label{section:recursive}

Our goal in this section is to prove Theorem \ref{theorem:recursive}, which gives a recursive
formula for the rank of $\RH_{n-2}(\cT_n(\Q)/\Gamma_n(p))$.
Before we do this, we will prove the following combinatorial lemma.  For a vector
space $V$ and a line $\ell \subset V$, write $X_k(V,\ell)$ for the set of all
$W \in \Gr_k(V)$ such that $\ell \notin W$.
In the following lemma (and throughout this section),
we emphasize to the reader that $|S|$ means the cardinality of the set $S$ 
(as opposed to something like a geometric realization).

\begin{lemma}
\label{lemma:countit}
Let $p$ be a prime.  For some $n \geq 2$, let $\ell$ be a line in $\Field_p^n$.  Then
\[|X_k(\Field_p^n,\ell)| = p^k |\Gr_k(\Field_p^{n-1})|.\]
\end{lemma}
\begin{proof}
Pick some nonzero $\vec{x} \in \ell$.  Define
\begin{align*}
X = \{\text{$(U,\zeta)$ $|$ }&\text{$U \subset \Field_p^n$ a $(k+1)$-dimensional subspace with $\vec{x} \in U$} \\
&\text{and $\zeta\colon U \rightarrow \Field_p$ a linear map with $\zeta(\vec{x}) = 1$}\}.
\end{align*}
The map $X \rightarrow X_k(\Field_p^n,\ell)$ taking $(U,\zeta)$ to $\ker(\zeta)$ is a bijection; its inverse
takes $W \in X_k(\Field_p^n,\ell)$ to the pair $(U,\zeta)$ where $U = \Span{W,\vec{x}}$ and 
$\zeta\colon U \rightarrow \Field_p$ is the unique linear map satisfying $\zeta|_W = 0$ and $\zeta(\vec{x})=1$.
It is thus enough to count $|X|$.  The possible choices for $U$ are in bijection with $\Gr_{k}(\Field_p^n/\ell)$, so there
are $|\Gr_{k}(\Field_p^{n-1})|$ of them.  For a fixed $U$, there are $p^k$ choices of $\zeta$ such that
$(U,\zeta) \in X$; indeed, picking a basis $\{\vec{x}_1,\ldots,\vec{x}_{k+1}\}$ for $U$ with $\vec{x}_1 = \vec{x}$,
the linear map $\zeta$ must satisfy $\zeta(\vec{x}_1) = 1$, but the values of $\zeta(\vec{x}_i)$ for
$2 \leq i \leq k+1$ can be arbitrary elements of $\Field_p$.  The lemma follows.
\end{proof}

\begin{proof}[Proof of Theorem \ref{theorem:recursive}]
Recall that the statement we must prove is as follows.  Fix a prime $p \geq 3$.  For
$n \geq 1$, let $t_n$ be the rank of $\RH_{n-2}(\cT_n(\Q)/\Gamma_n(p))$,
so trivially $t_1 = 1$.  We then must prove that
\[t_n = \left(\frac{p-3}{2}+\left(\frac{p-1}{2}\right)\cdot p^{n-1}\right) t_{n-1}
+ \frac{(p-1)(p-3)}{4} \sum_{k=1}^{n-2} p^k \cdot |\Gr_k(\Field_p^{n-1})| \cdot t_k t_{n-k-1}\]
for $n \geq 2$.

Proposition \ref{proposition:tquotient} says that 
\[\cT_n(\Q)/\Gamma_n(p) \cong \TD_n(\Field_p),\]
so we must calculate the rank of $\RH_{n-2}(\TD_n(\Field_p))$.  Our argument for this
is inspired by the discrete Morse theory proof of the Solomon--Tits theorem in \cite[Proof of Theorem 5.1]{BestvinaMorse}.

Fix a line $\ell \subset \Field_p^n$.
For $0 \leq k \leq n-1$, define subcomplexes $Y_k$ of $\TD_n(\Field_p)$ as follows.
\begin{compactitem}
\item Let $Y_0$ be the full subcomplex of $\TD_n(\Field_p)$ spanned by $\pm$-oriented subspaces $V$ of $\Field_p^n$
such that $\ell \subset V$.
\item For $1 \leq k \leq n-1$, let $Y_k$ be the full subcomplex of $\TD_n(\Field_p)$ spanned
by $Y_{k-1}$ along with all $\pm$-oriented subspaces $V$ of $\Field_p^n$ such that $\ell \not\subset V$
and $\dim(V) = k$.
\end{compactitem}
We thus have
\[Y_0 \subset Y_1 \subset \cdots \subset Y_{n-1} = \TD_n(\Field_p).\]
We inductively determine the homotopy type of these $Y_k$ as follows.  For a subspace $V$ of $\Field_p^n$, 
let $\Or(V)$ be the discrete set of all $\pm$-orientations on $V$, so $|\Or(V)| = \frac{p-1}{2}$ and
$\Or(V)$ is a wedge of $\frac{p-3}{2}$ copies of $S^0$.

\begin{claim}
$Y_0$ is homotopy equivalent to a wedge of $(\frac{p-3}{2}) \cdot t_{n-1}$ copies of $S^{n-2}$.
\end{claim}
\begin{proof}[Proof of claim]
Indeed, $Y_0$ is homeomorphic to the join of the following two spaces:
\begin{compactitem}
\item The discrete set $\Or(\ell)$, which is homeomorphic to a wedge of $\frac{p-3}{2}$ copies of $S^0$.
\item The full subcomplex of $\TD_n(\Field_p)$ spanned by $\pm$-oriented subspaces $V$ with $\ell \subsetneq V$.
This is homeomorphic to $\TD(\Field_p^n/\ell) \cong \TD_{n-1}(\Field_p)$, and thus is homotopy equivalent
to a wedge of $t_{n-1}$ copies of $S^{n-3}$.
\end{compactitem}
We conclude that $Y_0$ is homotopy equivalent to a wedge of $(\frac{p-3}{2}) \cdot t_{n-1}$
copies of $S^0 \ast S^{n-3} \cong S^{n-2}$, as desired.
\end{proof}

\begin{claim}
For $1 \leq k \leq n-2$, the complex $Y_k$ is homotopy equivalent to the wedge of $Y_{k-1}$ and
\[\frac{(p-1)(p-3)}{4} \cdot p^k \cdot |\Gr_k(\Field_p^{n-1})| \cdot t_k t_{n-k-1}\]
copies of $S^{n-2}$.
\end{claim}
\begin{proof}[Proof of claim]
The new vertices that are added to $Y_{k-1}$ to form $Y_k$ consist of the $\pm$-oriented
$k$-dimensional subspaces $V$ of $\Field_p^n$ such that $\ell \not\subset V$.  There are
$\frac{p-1}{2}$ possible $\pm$-orientations on each such $k$-dimensional subspace, so by
Lemma \ref{lemma:countit} there are
\begin{equation}
\label{eqn:yk1}
\frac{p-1}{2} \cdot p^k |\Gr_k(\Field_p^{n-1})|
\end{equation}
new vertices.
 
Let $V$ be one of these new vertices and let $L(V)$ be its link in $Y_k$.
The vertex $V$ is not adjacent to any other new vertices, so 
$L(V)$ is entirely contained in $Y_{k-1}$.  We deduce that $Y_k$ is homeomorphic
to the space obtained from $Y_{k-1}$ by coning off all these $L(V)$.  Below we will prove
that $L(V)$ is homotopy equivalent to a wedge of 
\begin{equation}
\label{eqn:yk2}
\left(\frac{p-3}{2}\right) \cdot t_k t_{n-k-1}
\end{equation}
copies of $S^{n-3}$.  Since by induction we already know that $Y_{k-1}$ is homotopy
equivalent to a wedge of copies of $S^{n-2}$, this will imply that $L(V)$ is nullhomotopic
in $Y_{k-1}$ and thus that coning it off changes the homotopy type of $Y_{k-1}$ by wedging
it with the suspension $\Sigma L(V)$, which is homotopy equivalent to a wedge of \eqref{eqn:yk2}
copies of $S^{n-2}$.  Since we are doing this \eqref{eqn:yk1} times, the claim follows.

It remains to prove that $L(V)$ is homotopy equivalent to a wedge of \eqref{eqn:yk2} copies
of $S^{n-3}$.  The link of $V$ in the whole complex $\TD_n(\Field_p)$ consists of
flags of $\pm$-oriented subspaces such that the flag does not contain $V$,
but such that $V$ can be inserted into it.  In other words, the link of 
$V$ in $\TD_n(\Field_p)$ consists of flags of $\pm$-oriented subspaces of 
$\Field_p^n$ of the form
\begin{equation}
\label{eqn:linkflag}
0 \subsetneq A_0 \subsetneq \cdots \subsetneq A_r \subsetneq B_{r+1} \subsetneq \cdots \subsetneq B_{r+s} \subsetneq \Field_p^n,
\end{equation}
where each $A_i$ is properly contained in $V$ and each $B_j$ properly contains $V$.
For this flag to lie in $L(V)$, each $\pm$-oriented subspace in it must lie
in $Y_k$.  The $A_i$ have dimension less than $\dim(V) = k$, so they automatically
lie in $Y_k$.  For the $B_{j}$'s, however, the only way they can lie in $Y_k$ is for
them to lie in $Y_0$, i.e.\ for them to contain $\ell$.  Since $B_j$ already contains
$V$, we deduce that it must contain $V' = \Span{V,\ell}$.  This containment need
not be proper, i.e.\ possibly $B_j = V'$ with some $\pm$-orientation.

In summary, $L(V)$ consists of flags of $\pm$-oriented subspaces of $\Field_p^n$ as
in \eqref{eqn:linkflag} where each $A_i$ is properly contained in $V$ and each
$B_j$ contains (and possibly even equals) $V'$.  This implies that $L(V)$ is homeomorphic
to the join of the following spaces:
\begin{compactitem}
\item The subcomplex $\TD(V)$, which since $V$ is $k$-dimensional is by induction homotopy equivalent to a wedge of
$t_k$ copies of $S^{k-2}$.
\item The discrete subspace $\Or(V')$ of $Y_{k-1}$, which is homotopy
equivalent to a wedge of $\frac{p-3}{2}$ copies of $S^0$.
\item The full subcomplex of $\TD_n(\Field_p)$ spanned by $\pm$-oriented subspaces $W$ with $V' \subsetneq W$.
This is homeomorphic to $\TD(\Field_p^n/V') \cong \TD_{n-k-1}(\Field_p)$, and thus is homotopy equivalent
to a wedge of $t_{n-k-1}$ copies of $S^{n-k-3}$.
\end{compactitem}
We conclude that $L(V)$ is homotopy equivalent to a wedge of 
\[t_k \cdot \left(\frac{p-3}{2}\right) \cdot t_{n-k-1}\]
copies of $S^{k-2} \ast S^0 \ast S^{n-k-3} \cong S^{n-3}$, as desired.
\end{proof}

\begin{claim}
The complex $Y_{n-1}$ is homotopy equivalent to the wedge of $Y_{n-2}$ and
\[\frac{p-1}{2} \cdot p^{n-1} \cdot t_{n-1}\]
copies of $S^{n-2}$.
\end{claim}
\begin{proof}[Proof of claim]
This is almost identical to the proof of the previous claim, so we only list the differences:
\begin{compactitem}
\item There are now
\[\frac{p-1}{2} \cdot p^{n-1} |\Gr_{n-1}(\Field_p^{n-1})| = \frac{p-1}{2} \cdot p^{n-1}\]
new vertices.
\item This time the link $L(V)$ is just homeomorphic to $\TD(V) \cong \TD_{n-1}(\Field_p)$ since $\Span{V,\ell}$ is the whole
vector space $\Field_p^n$ and thus does not contribute vertices to the building.  It is thus homotopy equivalent
to $t_{n-1}$ copies of $S^{n-2}$.\qedhere
\end{compactitem}
\end{proof}

Adding up the contributions coming from the above three claims, we deduce the desired recursive formula.
\end{proof}

\subsection{Improving the bound on the top cohomology group}
\label{section:evenbetter}

We close by proving Theorem \ref{theorem:evenbetter}.

\begin{proof}[Proof of Theorem \ref{theorem:evenbetter}]
We first recall what we must prove.  
Fix a prime $p \geq 3$.  For $n \geq 1$, let $t_n$ be the rank of 
\[\RH_{n-2}(\cT_n(\Q) / \Gamma_n(p)) \cong \RH_{n-2}(\TD_n(\Field_p))\]
given by Theorem \ref{theorem:recursive}.  Also, set $t_0 = 1$.  We must prove that for $n \geq 3$, the rank of
\[\HH^{\binom{n}{2}}(\Gamma_n(p)) \cong (\St_n(\Q))_{\Gamma_n(p)}\]
is at least
\[t_n+\frac{(p+2)(p-3)(p-5)(p-1)}{24} \cdot |\Gr_2(\Field_p)| \cdot t_{n-2}\]
with equality if $p=3$ or $p=5$.

For $p=3$ and $p=5$, this follows from
Theorem \ref{theorem:main}, so we can assume that $p>5$. 
Let $P_2^n$ be the set of $\pm$-oriented $2$-dimensional
subspaces of $\Field_p^n$.  Theorem \ref{theorem:mainbetter3} says that
there is a surjective map
\begin{equation}
\label{eqn:finalsurj}
(\St_n(\Q))_{\Gamma_n(p)} \longrightarrow \RH_{n-2}(\cT_n(\Q) / \Gamma_n(p))
\end{equation}
whose kernel surjects onto
\begin{equation}
\label{eqn:finalker}
\Z[P_2^n] \otimes \RH_{n-4}(\TD_{n-2}(\F_p)) \otimes \HH_1(\BDA_2(\F_p)).
\end{equation}
Since there are $\frac{p-1}{2}$ choices of $\pm$-orientation on a $2$-dimensional
subspace of $\Field_p^n$, the rank of $\Z[P_2^n]$ is $\frac{p-1}{2} \cdot |\Gr_2(\Field_p)|$.
Lemma \ref{lemma:n2} says that the rank of $\HH_1(\BDA_2(\Field_p))$ is 
$\frac{(p+2)(p-3)(p-5)}{12}$.  Finally, the rank of $\RH_{n-4}(\TD_{n-2}(\F_p))$ is $t_{n-2}$.
We deduce that the rank of \eqref{eqn:finalker} is
\[\frac{(p+2)(p-3)(p-5)(p-1)}{24} \cdot |\Gr_2(\Field_p)| \cdot t_{n-2}.\]
Since the rank of the target of \eqref{eqn:finalsurj} is $t_n$, the theorem follows.
\end{proof}

\begin{remark}
It follows from work of Lee--Schwermer (\cite{LeeSchwermer}; see \cite{Adem} for an alternate, more topological
proof) that the rank of $\HH^3(\Gamma_3(p))$ is at least $\frac{(p^3-1)(p^3-3p^2-p+15)}{12}+1$.  This 
is generally larger than the bound we give in Theorem \ref{theorem:evenbetter} for $n=3$.  One can likely use this
bound to give a lower bound for the rank of $\HH_{2}(\BDA_3(\Field_p))$ which can then be plugged into the 
map-of-posets spectral sequence to obtain an even better lower bound for the rank of 
$\HH^{\binom{n}{2}}(\Gamma_n(p))$ for $n>3$. We do not pursue this approach here.
\end{remark}

\begin{footnotesize}
\noindent
\begin{tabular*}{\linewidth}[t]{@{}p{\widthof{Department of Mathematics}+0.5in}@{}p{\widthof{Department of Mathematics}+0.5in}@{}p{\linewidth - \widthof{Department of Mathematics} - \widthof{Department of Mathematics} - 1in}@{}}
{\raggedright
Jeremy Miller\\
Department of Mathematics\\
Purdue University\\
150 N University St\\
West Lafayette, 47907 IN \\
{\tt jeremykmiller@purdue.edu}} &
{\raggedright
Peter Patzt\\
Department of Mathematics\\
Purdue University\\
150 N University St\\ 
West Lafayette, 47907 IN \\
{\tt ppatzt@purdue.edu}} &
{\raggedright
Andrew Putman\\
Department of Mathematics\\
University of Notre Dame \\
255 Hurley Hall\\
Notre Dame, IN 46556\\
{\tt andyp@nd.edu}}
\end{tabular*}\hfill
\end{footnotesize}

\end{document}